\documentclass[12pt,reqno]{amsart}

\usepackage{amsthm,amsmath,amssymb,amsfonts,bm}
\usepackage{amsaddr}
\usepackage{natbib}
\usepackage{txfonts}
\usepackage{fullpage}
\newtheorem{theorem}{Theorem}
\newtheorem{proposition}[theorem]{Proposition}%
\newtheorem{lemma}[theorem]{Lemma}%
\newtheorem{corollary}[theorem]{Corollary}
\newtheorem{remark}{Remark}%

\newtheorem{assumption}{Assumption}

\usepackage[utf8]{inputenc} 
\usepackage[T1]{fontenc}
\usepackage{url}
\usepackage{booktabs}
\usepackage{amsfonts}
\usepackage{nicefrac}
\usepackage{xcolor}
\usepackage{adjustbox}

\usepackage{graphicx}
\usepackage{mathtools}
\mathtoolsset{showonlyrefs=true}

\def\bzero{\bm{0}}
\def\bb{\bm{b}}

\def\bv{\bm{v}}
\def\bx{\bm{x}}
\def\by{\bm{y}}

\def\bA{\bm{A}}
\def\bB{\bm{B}}
\def\bC{\bm{C}}
\def\bD{\bm{D}}
\def\bI{\bm{I}}
\def\bL{\bm{L}}
\def\bP{\bm{P}}
\def\bT{\bm{T}}
\def\bU{\bm{U}}
\def\bV{\bm{V}}
\def\bW{\bm{W}}
\def\bX{\bm{X}}

\def\bZ{\bm{Z}}
\def\btheta{\bm{\theta}}
\def\bbeta{\bm{\beta}}
\def\bgamma{\bm{\gamma}}
\def\bmeta{\bm{\eta}}

\def\bSigma{\bm{\Sigma}}
\def\bTheta{\bm{\Theta}}
\renewcommand{\hat}{\widehat}

\newcommand{\mone}{\textbf{1}}
\newcommand{\as}{\overset{\mathrm{a.s.}}{\longrightarrow}}
\newcommand{\pconv}{\overset{{p}}{\to}}
\newcommand{\dconv}{\overset{{d}}{\to}}

\newcommand{\abs}[1]{\left\lvert#1\right\rvert}
\newcommand{\norm}[1]{\left\lVert#1\right\rVert}

\newcommand{\rbr}[1]{\left(#1\right)}
\newcommand{\sbr}[1]{\left[#1\right]}
\newcommand{\cbr}[1]{\left\{#1\right\}}

\newcommand{\N}{\mathbb{N}}

\newcommand{\R}{\mathbb{R}}

\newcommand{\E}{\mathbb{E}}
\newcommand{\mF}{\mathcal{F}}

\newcommand{\mY}{\mathcal{Y}}
\newcommand{\mR}{\mathcal{R}}

\renewcommand{\Pr}{\mathbb{P}}

\newcommand{\iid}{\overset{\mathrm{iid}}{\sim}}

\newcommand{\trace}{\mathrm{tr}}
\newcommand{\prox}{\mathrm{prox}}
\newcommand{\diag}{\mathrm{diag}}

\DeclareMathOperator*{\argmin}{arg\,min}

\usepackage{xcolor}
\newcommand{\rc}{\color{black}}

\DeclareMathOperator{\Var}{Var}

\allowdisplaybreaks

\begin{document}

\title{{\Large High-Dimensional Single-Index Models: \\Link Estimation and Marginal Inference}}

\allowdisplaybreaks
\sloppy
\raggedbottom

\let\origmaketitle\maketitle
\def\maketitle{
  \begingroup
  \def\uppercasenonmath##1{} 
  \let\MakeUppercase\relax 
  \origmaketitle
  \endgroup
}

\author{Kazuma Sawaya$^1$, Yoshimasa Uematsu$^2$, Masaaki Imaizumi$^{1,3,4}$
}
\address{
$^1$The University of Tokyo, Tokyo, Japan \\
$^2$Hitotsubashi University, Tokyo, Japan\\
$^3$RIKEN Advanced Intelligene Project, Tokyo, Japan\\
$^4$Kyoto University, Kyoto, Japan
}

\begin{abstract}    
This study proposes a novel method for estimation and hypothesis testing in high-dimensional single-index models. We address a common scenario where the sample size and the dimension of regression coefficients are large and comparable. Unlike previous approaches, which often overlook the estimation of the unknown link function, we introduce a new method for link function estimation. Leveraging the information from the estimated link function, we propose more efficient estimators that are better aligned with the underlying model. Furthermore, we rigorously establish the asymptotic normality of each coordinate of the estimator. This provides a valid construction of confidence intervals and $p$-values for any finite collection of coordinates. Numerical experiments validate our theoretical results.
\end{abstract}

\maketitle

\section{Introduction}

We consider $n$ i.i.d.\ observations $\{(\bX_i,y_i)\}_{i=1}^n$ with a $p$-dimensional Gaussian feature vector $\bX_i\sim\mathcal{N}_p(\bzero,\bSigma), \bSigma\in\R^{p\times p}$, and each scalar response $y_i$ belongs to a set $\mathcal{Y}$ (e.g., $\R,\R_+,\{0,1\},\N\cup\{0\}$), following the single-index model:
\begin{align}
\label{eq:model}
    \E[y_i\mid\bX_i=\bx]=g(\bbeta^\top\bx),
\end{align}
where 
$\bbeta = (\beta_1,\dots,\beta_p)^\top \in\R^p$ is an unknown deterministic coefficient vector, and 
$g(\cdot)$ is an unknown deterministic function, referred to as the link function, with $\bbeta^\top\bx$ being the index.
To identify the scale of $\bbeta$, we assume $\mathrm{Var}(\bbeta^\top\bX_i)=\bbeta^\top\bSigma\bbeta=1$. 
The model includes common scenarios such as:
\begin{itemize}
\item \textit{Linear regression}: $y_i\mid\bX_i\sim\mathcal{N}(\bbeta^\top\bX_i,\sigma_\varepsilon^2)$ 
with $\sigma_\varepsilon>0$ by setting $g(t)=t$. 
\item \textit{Poisson regression}: $y_i\mid\bX_i\sim\mathrm{Pois}(\exp(\bbeta^\top\bX_i))$ by setting $g(t)=\exp(t)$. 
\item \textit{Binary choice models}: $y_i\mid\bX_i\sim\mathrm{Bern}(g(\bbeta^\top\bX_i))$ with $g:\R\to[0,1]$. This includes logistic regression for $g(t)=1/(1+\exp(-t))$ and the probit model by setting $g(\cdot)$ as the cumulative distribution function of the standard Gaussian distribution.
\end{itemize}
We are interested in a high-dimensional setting, where both the sample size $n$ and the coefficient dimension $p := p(n)$ are large and comparable. 
Specifically, this study examines the proportionally high-dimensional regime defined by:
\begin{align}
n, p(n) \to \infty
\quad\text{and}\quad
p(n)/n =: \kappa \to \bar{\kappa}, \label{eq:high-dim}
\end{align}
where $\bar{\kappa}$ is a positive constant.

The single-index model \eqref{eq:model} possesses several practically important properties. First, it mitigates concerns about model misspecification, as it eliminates the need to specify $g(\cdot)$. Second, this model bypasses the curse of dimensionality associated with function estimation since the input index $\bbeta^\top\bX_i$ is a scalar. This advantage is particularly notable in comparison with nonparametric regression models, such as $y_i = \check{g}(\bX_i) + \varepsilon_i$, where $\check{g}:\R^p\to\R$ remains unspecified. Third, the model facilitates the analysis of the contribution of each covariate, $X_{ij}$ for $j=1,\ldots,p$, to the response $y_i$ by testing $\beta_j = 0$ against $\beta_j \neq 0$. 
Owing to these advantages, single-index models have been actively researched for decades \citep{powell1989semiparametric,hardle1989investigating,li1989regression,ichimura1993semiparametric,klein1993efficient,hardle1997semiparametric,hristache2001direct,nishiyama2005bootstrap,dalalyan2008new,alquier2013sparse,eftekhari2021inference,bietti2022learning,kuchibhotla2023semiparametric,fan2023understanding}, attracting interest across a broad spectrum of fields, particularly in econometrics \citep{horowitz2009semiparametric,li2023nonparametric}.

In the proportionally high-dimensional regime as defined in \eqref{eq:high-dim}, the single-index model and its variants have been extensively studied. For logistic regression, which is a particular instance of the single-index model, \citet{sur2019likelihood,salehi2019impact} have investigated the estimation and classification errors of the regression coefficient estimators $\bbeta$. Furthermore, \citet{sur2019modern,zhao2022asymptotic,yadlowsky2021sloe} have developed methods for asymptotically valid statistical inference. In the case of generalized linear models with a known link function $g(\cdot)$, \citet{rangan2011generalized,barbier2019optimal} have characterized the asymptotic behavior of the coefficient estimator, while \cite{sawaya2023statistical} have derived the coordinate-wise marginal asymptotic normality of an adjusted estimator of $\beta_j$. For the single-index model with an unknown link function $g(\cdot)$, the seminal work by \citet{bellec2022observable} establishes the (non-marginal) asymptotic normality of estimators, even when there is link misspecification. However, the construction of an estimator for the link function $g(\cdot)$ and the marginal asymptotic normality of the coefficient estimator are issues that have not yet been fully resolved.

Inspired by these seminal works, the following questions naturally arise:
\textit{
\begin{enumerate}
\item Can we consistently estimate the unknown link function $g(\cdot)$?
\item Can we rigorously establish marginal statistical inference for each coordinate of $\bbeta$?
\item Can we improve the estimation efficiency by utilizing the estimated link function?
\end{enumerate}}
\noindent
This paper aims to provide affirmative answers to these questions. Specifically, we propose a novel estimation methodology comprising three steps. First, we construct an estimator for the index $\bbeta^\top\bX_i$. 
Second, we develop an estimator for the link function $g(\cdot)$.
Third, we design a new estimator for $\bbeta$ with the estimated link function.
To conduct statistical inference, we investigate the estimation problem of inferential parameters necessary for establishing coordinate-wise asymptotic normality in high-dimensional settings.

Our contributions are summarized as follows:
\begin{itemize}
\item[-] \textit{Link estimation}: We propose a consistent estimator for the link function $g(\cdot)$, which is of practical significance as well as estimating coefficients. This aids in interpreting the model via the link function and mitigates negative impacts on coefficient estimation due to link misspecification.
\item[-] \textit{Marginal inference}: We establish the asymptotic normality for any finite subset of the coordinates of our estimator, facilitating coordinate-wise inference of $\bbeta$. This approach allows us not only to test each variable's contribution to the response but also to conduct variable selection based on importance statistics for each feature.
\item[-] \textit{Efficiency improvement}:
By utilizing the consistently estimated link function, we anticipate that our estimator of $\bbeta$ will be more efficient than previous estimators that rely on potentially misspecified link functions. We predominantly validate this efficiency through numerical simulations.
\end{itemize}

From a technical perspective,
we leverage the proof strategy in \citet{zhao2022asymptotic}
to demonstrate the marginal asymptotic normality of our estimator for $\bbeta$.
Specifically, we extend the arguments to a broader class of unregularized M-estimators, whereas \citet{zhao2022asymptotic} originally considered the maximum likelihood estimator (MLE) for logistic regression.

\subsection{Marginal Inference in High Dimensions} \label{sec:marginal_inference_highdim}
We review key technical aspects of statistical inference for each coordinate $\beta_j$ in the proportionally high-dimensional regime \eqref{eq:high-dim}. We maintain $\bbeta^\top\bX_i$ of constant order by considering the setting $\bbeta^\top\bSigma\bbeta=\Theta(1)$. We define $\bTheta=\bSigma^{-1}$ as the precision matrix for the distribution of $\bX_i$ and set $\tau_j^2:=\Theta_{jj}^{-1}>0$. An unbiased estimator of $\tau_j$ can be constructed using nodewise regression (cf. Section 5.1 in \citet{zhao2022asymptotic}). For simplicity, we assume $\tau_j$ is known, following prior studies.

In the high-dimensional regime \eqref{eq:high-dim}, statistical inference must address two components: the asymptotic distribution and the \textit{inferential parameters} of an estimator. We review the asymptotic distribution of the MLE $\hat\bbeta^\mathrm{m}$ for logistic regression. According to \citet{zhao2022asymptotic}, for all $j\in\{1,\ldots,p\}$ such that $\sqrt{p}\tau_j\beta_j=O(1)$ as $n\to\infty$, the estimator achieves the following asymptotic normality:
\begin{align}
\label{eq:asyN}
    \frac{\sqrt{p}(\hat\beta_j^\mathrm{m}-\mu_*\beta_j)}{\sigma_*/\tau_j}\dconv\mathcal{N}(0,1).
\end{align}
Here, we define $\mu_* \in\R$ and $\sigma_* >0$ as the asymptotic bias and variance, respectively, ensuring the convergence \eqref{eq:asyN}. It is crucial to note that both the estimator $\hat{\beta}_j^{\mathrm{m}}$ and the target $\beta_j$ scale as $O_p(1/\sqrt{p})$ here.

To perform statistical inference based on \eqref{eq:asyN}, it is necessary to estimate the inferential parameters $\mu_*$ and $\sigma_*$. Several studies including \citet{el2013robust,thrampoulidis2018precise,sur2019modern,loureiro2021learning} theoretically characterize these parameters as solutions to a system of nonlinear equations that depend on the data-generating process and the loss function. Additionally, various approaches have been developed to practically solve the equations by determining their hyperparameter $\bbeta^\top\bSigma\bbeta$ under different conditions. Specifically, \citet{sur2019modern} introduces \textit{ProbeFrontier} for estimating $\bbeta^\top\bSigma\bbeta$ based on the asymptotic existence/non-existence boundary of the maximum likelihood estimator (MLE) in logistic regression. \textit{SLOE}, proposed by \citet{yadlowsky2021sloe}, enhances this estimation using a leave-one-out technique. Moreover, \citet{sawaya2023statistical} takes a different moment-based approach to estimate $\bbeta^\top\bSigma\bbeta$ for generalized linear models.

For single-index models, \citet{bellec2022observable} introduces \textit{observable adjustments} that estimate the inferential parameters directly under the identification condition $\bbeta^\top\bSigma\bbeta = 1$ (as in Assumption \ref{asmp:feature}) irrespective of link misspecification, bypassing the system of equations. In our study, we develop an estimator for the single-index model satisfying the asymptotic normality \eqref{eq:asyN}, with corresponding estimators for the inferential parameters using observable adjustments.

\subsection{Related Works}
\label{sec:literature}
Research into the asymptotic behavior of statistical models in high-dimensional settings, where both $n$ and $p$ diverge proportionally, has gained momentum in recent years. Notable areas of exploration include (regularized) linear regression models \citep{donoho2009message,bayati2011lasso,krzakala2012probabilistic,bayati2013estimating,thrampoulidis2018precise,mousavi2018consistent,takahashi2018statistical,miolane2021distribution,guo2022moderate,hastie2022surprises,li2023spectrum}, robust estimation \citep{el2013robust,donoho2016high,bellec2025error}, generalized linear models \citep{rangan2011generalized,sur2019likelihood,sur2019modern,salehi2019impact,barbier2019optimal,zhao2022asymptotic,tan2023multinomial,sawaya2023statistical}, low-rank matrix estimation \citep{deshpande2017asymptotic,macris2020all,montanari2021estimation}, and various other models \citep{montanari2019generalization,loureiro2021learning,yang2021tensor,mei2022generalization}. These investigations focus primarily on the convergence limits of estimation and prediction errors. Theoretical analyses have shown that classical statistical estimation often fails to accurately estimate standard errors and may lack key desirable properties such as asymptotic unbiasedness and asymptotic normality.

In such analyses, the following theoretical tools have been employed: (i) the replica method \citep{mezard1987spin,charbonneau2023spin}, (ii) approximate message passing algorithms \citep{donoho2009message,bolthausen2014iterative,bayati2011dynamics,feng2022unifying}, (iii) the leave-one-out technique \citep{el2013robust,el2018impact}, (iv) the convex Gaussian min-max theorem \citep{thrampoulidis2018precise}, (v) second-order Poincaré inequalities \citep{chatterjee2009fluctuations,lei2018asymptotics}, and (vi) second-order Stein's formulae \citep{bellec2021second,bellec2022derivatives}. Although these tools were initially proposed for analyzing linear models with Gaussian design, they have been extensively adapted to a diverse range of models. In this study, we apply observable adjustments based on second-order Stein's formulae \citep{bellec2022observable} to directly estimate the asymptotic bias and variance of coefficient estimators. Furthermore, we provide a comprehensive proof of marginal asymptotic normality, extending the work of \citet{zhao2022asymptotic} to a wider array of estimators.

\subsection{Notation}
Define $[z]=\{1,\ldots,z\}$ for $z \in \N$. 
For a vector $\bb = (b_1,\dots,b_p)\in \R^p$, we write $\|\bb\|:=(\sum_{j=1}^p b_j^2)^{1/2}$ and $\|\bb\|_{\bSigma}^2:=\bb^\top\bSigma\bb$.
For a collection of indices $\mathcal{S}\subset[p]$, we define a sub-vector $\bb_\mathcal{S} := (b_j)_{j \in \mathcal{S}}$ as a slice of $\bbeta$. 
For a matrix $\bA\in\R^{p\times p}$, we define its minimum and maximum eigenvalues by $\lambda_{\rm min}(\bA)$ and $\lambda_{\rm max}(\bA)$, respectively. 
For a function $F: \R \to \R$, we say $F'$ the derivative of $F$ and $F^{(m)}$ the $m$th-order derivative.
For a function $f: \R \to \R$ and a vector $\bb \in \R^p$, $f(\bb) = (f(b_1),f(b_2),\dots,f(b_p))^\top \in \R^p$ denotes a vector by elementwise operations.

\subsection{Organization}
We organize the remainder of the paper as follows: Section \ref{sec:method} presents our estimation procedure. Section \ref{sec:main_results} describes the asymptotic properties of the proposed estimator and develops a statistical inference method. Section \ref{sec:experiments} provides several experiments to validate our estimation theory. Section \ref{sec:proof_outline} outlines the proofs of our theoretical results. Section \ref{sec:pilot} discusses alternative designs for estimators. Finally, Section \ref{sec:discussion} concludes with a discussion of our findings. The Appendix contains additional discussions and the complete proofs.

\section{Statistical Estimation Procedure}
\label{sec:method}
In this section, we introduce a novel statistical estimation method for single-index models as defined in \eqref{eq:model}. To give an overview, our estimator $\hat{\bbeta}$ is constructed through the following steps:
\begin{enumerate}
\item[(i)] Construct an index estimator $W_i$ for $\bbeta^\top \bX_i$ using the ridge regression estimator $\tilde{\bbeta}$, referred to as a pilot estimator. 
\item[(ii)] Develop a function estimator $\hat{g}(\cdot)$ for the link function $g(\cdot)$, based on the distributional characteristics of the index estimator $W_i$.
\item[(iii)] Construct our estimator $\hat{\bbeta}$ for the coefficients $\bbeta$, using the estimated link $\hat{g}(\cdot)$ function.
\end{enumerate}
Furthermore, statistical inference additionally involves a fourth step:
\begin{enumerate}
\item[(iv)] Estimate the inferential parameters $\mu_*$ and $\sigma_*$, conditional on the estimated link function $\hat{g}(\cdot)$.
\end{enumerate}

In our estimation procedure, we divide the dataset ${(\bX_i,y_i)}_{i=1}^n$ into two disjoint subsets ${(\bX_i,y_i)}_{i \in I_1}$ and ${(\bX_i,y_i)}_{i \in I_2}$, where $I_1, I_2 \subset [n]$ are index sets such that $I_1 \cap I_2 = \emptyset$ and $I_1 \cup I_2 = [n]$. Additionally, for $k=1,2$, let $\bX^{(k)}\in\R^{n_k\times p}$ and $\by^{(k)}\in\R^{n_k}$ denote the design matrix and response vector of subset $I_k$, respectively. We utilize the first subset to estimate the link function (Steps (i) and (ii)), and the second subset to estimate the regression coefficients (Step (iii)) and inference parameters (Step (iv)). From a theoretical perspective, this division helps to manage the complicated dependency structure caused by data reuse. Nonetheless, for practical applications, we recommend employing all observations in each step to maximize the utilization of the data's inherent signal strength. Here, $n_1$ and $n_2$ are the sample sizes for each partition, satisfying $n=n_1+n_2$ such that $n_1$ and $n_2$ are the same asymptotic order, i.e.,  we define $\kappa_1=p/n_1$ and $\kappa_2=p/n_2$ and there exist constants $c_1,c_2>0$ with $c_1\le c_2$ independent of $n$ such that $\kappa_1,\kappa_2\in[c_1,c_2]$ holds.

\subsection{Index Estimation} \label{sec:index}
In this step, we use the first subset $(\bX^{(1)}, \by^{(1)})$. We define the pilot estimator as the ridge estimator,
$\Tilde{\bbeta}=(({\bX^{(1)}})^\top\bX^{(1)}+n_1\lambda\bI_p)^{-1}({\bX^{(1)}})^\top\by^{(1)}$ where $\lambda > 0$ is the regularization parameter. Further, we consider inferential parameters $(\mu_1, \sigma_1)$ of $\tilde{\bbeta}$, which satisfy
$\sqrt{p}\tau_j (\Tilde{\beta}_j - \mu_1 \beta_j)/ \sigma_1 \dconv \mathcal{N}(0,1)$ for $j\in[p]$ such that $\sqrt{p}\tau_j\beta_j=O(1)$.
Using these parameters, we develop an estimator $W_i$ for the index $\bbeta^\top{\bX_i^{(1)}}$ as follows:
\begin{align}
    W_i:=\tilde\mu^{-1}\tilde\bbeta^\top{\bX_i^{(1)}}-\tilde\mu^{-1}\tilde\gamma\rbr{y_i^{(1)}-\Tilde{\bbeta}^\top{\bX_i^{(1)}}}
    \label{def:W}
\end{align}
for each $i \in [n_1]$. Here, $\tilde{\mu}$ and $\Tilde{\sigma}^2$ are estimators of $\mu_1$ and $\sigma_1$, defined as
\begin{align}
        \tilde\mu=\abs{\|\tilde\bbeta\|^2-\tilde\sigma^2}^{1/2}  \quad\mbox{and}\quad
        \tilde\sigma^2=\frac{n_1^{-1}\|\by^{(1)}-\bX^{(1)}\tilde\bbeta\|^2}{(\tilde{v}+\lambda)^2/\kappa_1},
    \end{align}
where $\tilde{\gamma} := \kappa_1 / (\Tilde{v} + \lambda)$ and $\tilde{v}=n_1^{-1}\trace(\bI_n-\bX^{(1)}({(\bX^{(1)}})^\top\bX^{(1)}+n_1\lambda\bI_p)^{-1}{(\bX^{(1)}})^\top)$.
These estimators are obtained by the observable adjustment technique described in \citet{bellec2022observable}.

This index estimator $W_i$ is approximately unbiased for the index $\bbeta^\top{\bX_i^{(1)}}$, yielding the following asymptotic result.
\begin{align}
    W_i \approx \bbeta^\top{\bX_i^{(1)}}+\mathcal{N}(0,\tilde\sigma^2/\tilde\mu^2). \label{eq:W}
\end{align}
We will provide its rigorous statement in Proposition \ref{prop:initial} in Section \ref{sec:proof-link}.

There are other options for the pilot estimator besides the ridge estimator $\Tilde{\bbeta}$. 
If $\kappa_1 \leq 1$ holds, the least squares estimator can be an alternative. 
If $y_i$ is a binary or non-negative integer, the MLE of logistic or Poisson regression can be a natural candidate, respectively, although the ridge estimator $\Tilde{\bbeta}$ is valid regardless of the form that $y_i$ takes. 
In each case, the estimated inferential parameters $(\Tilde{\gamma}, \tilde{\mu}, \tilde{\sigma}^2)$ should be updated accordingly.
Details are presented in Section \ref{sec:pilot}.

\subsection{Link Estimation} \label{sec:g}

We develop an estimator of the link function $g(\cdot)$ using $W_i$ in \eqref{def:W}.
If we could observe the true index $\bbeta^\top{\bX_i^{(1)}}$, it would be possible to estimate $g(x) = \E[y_1 \mid \bbeta^\top \bX_1 = x]$ by applying standard nonparametric methods to the pairs of responses and true indices $(\by^{(1)}, \bX^{(1)}\bbeta)$. However, as the true index is unobservable, we must estimate $g(\cdot)$ using given pairs of responses and contaminated indices ${(y_i^{(1)}, W_i)}_{i=1}^{n_1}$, where $W_i \approx \bbeta^\top \bX_i^{(1)} + \mathcal{N}(0, \tilde\varsigma^2)$ with $\tilde\varsigma^2=\tilde\sigma^2/\tilde\mu^2$ being a ratio of the inferential parameters. The type of error $\mathcal{N}(0, \tilde\varsigma^2)$ involving the regressor leads to an attenuation bias in the estimation of $g(\cdot)$, known as the \textit{errors-in-variables} problem. To address this issue, we utilize a deconvolution technique \citep{stefanski1990deconvolving} to remove the bias asymptotically. Further details of the deconvolution are provided in Appendix \ref{sec:deconv}.

We define an estimator of $g(\cdot)$.
In preparation, we specify a kernel function  $K:\R\rightarrow\R$, and define a deconvolution kernel $K_n:\R\to\R$  as follows:
\begin{align}
K_n(x)=\frac{1}{2\pi}\int_{-\infty}^\infty\exp(-\mathrm{i}tx)\frac{\phi_K(t)}{\phi_{\tilde\varsigma}(t/h_n)}dt,
\end{align} 
where $h_n > 0$ is a bandwidth, $\mathrm{i}=\sqrt{-1}$ is an imaginary unit, and $\phi_K:\R\rightarrow\R$ and $\phi_{\tilde\varsigma}:\R\to\R$ are the Fourier transform of $K(\cdot)$ and the density function of $\mathcal{N}(0,\tilde\varsigma^2)$, respectively. 
We then define our estimator of $g(\cdot)$ as 
\begin{align}
\label{eq:deconv}
\hat{g}(x):=\mathcal{R}[\breve{g}](x)\quad\mbox{with}\quad \breve{g}(x)=\frac{\sum_{i=1}^{n_1}y_i^{(1)} K_n\rbr{(x-W_i)/h_n}}{\sum_{i=1}^{n_1} K_n\rbr{(x-W_i)/h_n}},
\end{align} 
where $\mR[\cdot]$ is a monotonization operator, specified later, which maps any measurable function to a monotonic function, and $\breve{g}(\cdot)$ is a Nadaraya-Watson estimator obtained by the deconvolution kernel. We will prove the consistency of this estimator in Theorem \ref{thm:link-est} in Section \ref{sec:main_results}.

The monotonization operation $\mR[\cdot]$ on $\breve{g}(\cdot)$ is justifiable because the true link function $g(\cdot)$ is assumed to be monotonic.
One simple choice for $\mR[\cdot]$, applicable to any measurable function $f:\R\to\R$, is
\begin{align}
    \mR_\mathrm{naive}[f](x) = \sup_{x' \leq x} f(x'),\quad  x \in \R.
\end{align}
This definition holds for all $x \in \R$.
Another effective alternative is the rearrangement operator \citep{chernozhukov2009improving}. This operator monotonizes a measurable function $f:\R\to\R$ within a compact interval $ [\underline{x},\overline{x}] \subset \R$:
\begin{align}
    \mathcal{R}_\mathrm{r}[f](x)=\inf\cbr{t\in\R:\int_{[0,1]} 1\left\{f \left(\frac{u-\underline{x}}{\overline{x} - \underline{x}}\right)\le t \right\}du\ge \frac{x-\underline{x}}{\overline{x} - \underline{x}}},~ x \in  [\underline{x},\overline{x}]. \label{def:rearrangement}
\end{align}
This operator, which sorts the values of $f(\cdot)$ in increasing order, is robust against local fluctuations such as function bumps. Thus, it effectively addresses bumps in $\breve{g}(\cdot)$ arising from kernel-based methods.

\begin{remark}[Motivation for Nadaraya-Watson estimator]
    We employ the Nadaraya-Watson-type estimator in \eqref{eq:deconv} to achieve a consistent nonparametric estimator in the error-in-variables setup.
    In the presence of errors in variables, while several approaches based on deconvolution combined with B-spline \citep{koo1998b} or wavelet regression \citep{chichignoud2017adaptive} have been proposed, these methods do not cover the case in which the error in variables follows a Gaussian distribution. 
    There also exist studies 
    using neural networks \citep{kohler2011analysis}, but their consistency results rely on the assumption that the error in variables is asymptotically negligible, which is not realistic in practice.
\end{remark}

\subsection{Coefficient Estimation}
We next propose our estimator of $\bbeta$ using $\hat{g}(\cdot)$ obtained in \eqref{eq:deconv}. In this step, we consider the link estimator $\hat{g}(\cdot)$ from $\bX^{(1)}$ as given, and estimate $\bbeta$ using $\bX^{(2)}$. To facilitate this, we introduce the \textit{surrogate loss function} for $\bb \in \R^p$, with $ \bx \in \R^p$, $ y \in \R$, and any measurable function $\bar g:\R\to\R$:
\begin{align}
\label{eq:loss}
    \ell(\bb;\bx,y,\bar{g}):=\bar{G}(\bx^\top \bb)-y\bx^\top \bb,
\end{align}
where $\bar{G}:\R\to\R$ is a function such that $\bar{G}'(t)=\bar{g}(t)$. This function can be viewed as a natural extension of the matching loss \citep{auer1995exponentially} used in generalized linear models.
If $\bar{g}(\cdot)$ is strictly increasing, then the loss is strictly convex in $\bb$. Moreover, the surrogate loss is justified by the characteristics of the true parameter as follows \citep{agarwal2014least}:
\begin{align}
    \bbeta=\argmin_{\bb\in\R^p}\E\sbr{\ell(\bb;{\bX_1},{y_1},{g})|\bX_1},
\end{align}
provided that $G(\cdot)$ is integrable.
The surrogate loss aligns with the negative log-likelihood when $g(\cdot)$ is known and serves as a canonical link function, thereby making the surrogate loss minimizer a generalization of the MLEs in generalized linear models.

Using the second dataset $(\bX^{(2)}, \by^{(2)})$ with any given function $\bar{g}(\cdot)$, we define our estimator of $\bbeta$ as
\begin{align}
\label{eq:estimator}
    \hat{\bbeta}(\bar{g})=\argmin_{\bb\in\R^p}\sum_{i=1}^{n_2}\ell(\bb;\bX_i^{(2)},y_i^{(2)},\bar{g}) +J(\bb),
\end{align}
where $J:\R^p\to\R$ is a convex regularization function.
Finally, we substitute the link estimator $\hat{g}(\cdot)$ into \eqref{eq:estimator} to obtain our estimator $\hat\bbeta(\hat{g})$.
The use of a nonzero regularization term, $J(\cdot)$, is beneficial in cases where the minimizer \eqref{eq:estimator} is not unique or does not exist; see, for example, \cite{candes2020phase} for the logistic regression case.

\begin{remark}[Motivation of sample splitting]
    We adopt sample splitting because, when the same data are used both for estimating the link function and for estimating the regression coefficients, the resulting dependence introduces a complex bias structure, whose impact on the asymptotic normality with observable adjustments is technically challenging to analyze.
    On the other hand, from a conceptual perspective, classical arguments based on Neyman orthogonality \citep{chernozhukov2018double,chernozhukov2022locally} for single-index models suggest that the first-order effect of the link-function estimation error on the asymptotic distribution of the coefficient estimator vanishes. As a consequence, even without sample splitting, one may expect the asymptotic normality of the coefficient estimator to remain valid, despite reusing the same data for both stages.
\end{remark}

\subsection{Inferential Parameter Estimation} \label{sec:inferential_estimator}
We finally study estimators for the inferential parameters of our estimator $\hat{\bbeta}(\hat{g})$, which are essential for statistical inference as discussed in Section \ref{sec:marginal_inference_highdim}. As established in \eqref{eq:asyN}, it is necessary to estimate the asymptotic bias $\mu_*(\hat{g})$ and variance $\sigma_*^2(\hat{g})$ that satisfy the following relationship:
\begin{align}
 \frac{\sqrt{p}(\hat\beta_j(\hat{g})-\mu_*(\hat{g})\beta_j)}{\sigma_*(\hat{g})}\dconv\mathcal{N}(0,1),\quad j\in[p],
\end{align}
conditional on $(\bX^{(1)},\by^{(1)})$ and consequently on $\hat{g}(\cdot)$.

We develop estimators for these inferential parameters using observable adjustments as suggested by \citet{bellec2022observable}, in accordance with the estimator \eqref{eq:estimator}.
For any measurable function $\bar g:\R\to\R$, we define $\bD=\diag(\bar{g}'(\bX^{(2)}\hat{\bbeta}(\bar{g})))$ and  $\hat{v}_\lambda =n_2^{-1}\trace(\bD-\bD\bX^{(2)}(({\bX^{(2)}})^\top\bD\bX^{(2)}+n_2\lambda\bI_p)^{-1}({\bX^{(2)}})^\top\bD)$ for $\lambda \geq 0$.
When incorporating $J(\bb)=\lambda\|\bb\|^2/2$ into \eqref{eq:estimator} with $\lambda>0$, we propose the following estimators:
\begin{align}
    \hat\mu(\bar{g})&=\abs{\|\hat\bbeta(\bar{g})\|^2-\hat\sigma^2(\bar{g})}^{1/2} \mbox{~~and~~}\hat\sigma^2(\bar{g})=  \frac{\|\by^{(2)}-\bar{g}(\bX^{(2)}\hat\bbeta(\bar{g}))\|^2}{n_2 (\hat{v}_\lambda+\lambda)^2/\kappa_2}.
\end{align}
In the case where $J(\cdot) \equiv \bzero$ holds, we define
\begin{align}
    \hat{\mu}_0(\bar{g})=\abs{\frac{{\|\bX^{(2)}\hat\bbeta(\bar{g})\|^2}}{n_2}-(1-\kappa_2)\hat{\sigma}_0^2(\bar{g})}^{1/2} \mbox{~~and~~}
    \hat{\sigma}_0^2(\bar{g})=  \frac{\|\by^{(2)}-\bar{g}(\bX^{(2)}\hat\bbeta(\bar{g}))\|^2}{n_2 \hat{v}_0^2/\kappa_2}.\label{eq:sigma-hat}
\end{align}
A theoretical justification for the asymptotic normality with these estimators and their application in inference is provided in Section \ref{sec:inference}.

\section{Main Theoretical Results of Proposed Estimators}\label{sec:main_results}

This section presents theoretical results for our estimation framework. Specifically, we prove the consistency of the estimator $\hat{g}(\cdot)$ for the link function $g(\cdot)$, and the asymptotic normality of the estimator $\hat{\bbeta}(\hat{g})$ for the coefficient vector $\bbeta$. Outlines of the proofs for each assertion will be provided in Section \ref{sec:proof_outline}.

\begin{assumption}[Gaussian covariates and identification]
\label{asmp:feature}
    Each row of the matrix $\bX$ independently follows $\mathcal{N}_p(\bzero,\bSigma)$ with $\bSigma$ obeying $\bbeta^\top\bSigma\bbeta=1$ and $0<c_\Sigma^{-1}\le\lambda_{\mathrm{min}}(\bSigma)\le\lambda_{\mathrm{max}}(\bSigma)\le c_\Sigma<\infty$ for some constant $c_\Sigma$.
\end{assumption}
It is common to assume Gaussianity of covariates in the proportionally high-dimensional regime, as mentioned in Section \ref{sec:literature}. The condition $\bbeta^\top\bSigma\bbeta=1$ is necessary to identify the scale of $\bbeta$, which ensures the uniqueness of the estimator in the single-index model with an unknown link function $g(\cdot)$. For example, without this condition, it would be impossible to distinguish between $g(\bX_1^\top\bb)$ and $f(2\bX_1^\top\bb)$, where $f(t) = g(t/2)$, for any $\bb \in \R^p$. Furthermore, the assumption of upper and lower bounds on the eigenvalues of $\bSigma$ implies that $\|\bbeta\| = \Theta(1)$.
{
\begin{assumption}[Coherency of the single-index model] 
\label{asmp:SIM}
A distribution of $(y_i, \bX_i)$, which follows the model \eqref{eq:model}, satisfies that there exist an (unknown) deterministic function $F:\R^2\to\R$ such that we have 
    \begin{align}
\label{eq:SIM}
    y_i=F(\bbeta^\top\bX_i,U_i),
\end{align}
where 
$U_i$ is some random variable independent of $\bX_i$.
\end{assumption}
This assumption ensures the coherency of the single-index model \eqref{eq:model}. 
The description has been commonly used \cite{li1989regression, bellec2022observable} and includes a wide class of models, e.g., the linear regression with $F(v,u) =  v + u$ and $U_i \sim N(0,\sigma_U^2)$ with $\sigma_U^2 > 0$, and the logistic regression model with $F(v,u)=\mone\{u \leq 1/(1+e^{-v})\}$ and $U_i \sim \mathrm{Unif}[0,1]$.
Specifically, this assumption excludes a case where $\bX_i$ affects $y_i$ in a way that is independent of the index $\bbeta^\top\bX_i$, such as $y_i=g(\bbeta^\top\bX_i)+\bgamma^\top\bX_i\epsilon_i$, where $\epsilon_i$ is a mean zero random variable independent of $\bX_i$ and $\bgamma\in\R^p$ is an additional coefficient. 
A property of $F$ is indirectly constrained by assumptions about the link $g$ that follow.

}

\begin{assumption}[Monotone and smooth link function]
    \label{asmp:link}
    {\rc
    There exists $m \in \N$ and constants $a<b$ such that $g^{(\ell)}(x)$ exists for every $\ell =0,1\ldots,m$ and $x \in [a,b]$.
    Also, there exists a constant $B > 0$ such that $ \max_{\ell = 0,1,...,m} \max_{x \in [a,b]} |g^{(\ell)}(x)|\le B$ holds.
    Furthermore, there exists $c_g \in (0,\infty)$ such that $c_g^{-1}\le \min_{x \in [a,b]} g'(x)$ holds.
    }
\end{assumption}

Assumption \ref{asmp:link} restricts the class of link functions to those that are monotonic. This class has been extensively reviewed in the literature, with \citet{balabdaoui2019least} providing a comprehensive discussion. It encompasses a wide range of applications, including utility functions, growth curves, and dose-response models \citep{matzkin1991semiparametric,wan2017monotonic,foster2013variable}. Furthermore, under a monotonically increasing link function, the sign of $\bbeta$ is identified, so that we can identify $\bbeta$ only by the scale condition $\bbeta^\top\bSigma\bbeta=1$.

{\rc
The lower boundedness of $g'(\cdot)$ on the closed interval implies that the loss function \eqref{eq:estimator} for the coefficient estimation is strictly convex on the interval. This assumption holds for the negative log-likelihood of the logistic regression that is not strictly convex on the real line.
}

\begin{assumption}[Moment conditions of $\by$]
    \label{asmp:y}
    $\E[y_1^2]<\infty$ holds.
    Further, $m_2(x):=\E[y_1^2\mid \bX_1^\top\bbeta=x]$ is continuous in $x  \in [a,b]$ for $a,b \in \R$ defined in Assumption \ref{asmp:link}.
\end{assumption}

The continuity of $m_2(\cdot)$ is maintained in many commonly used models, particularly when $g(\cdot)$ is continuous. For instance, the Poisson regression model defines $m_2(x) = \exp(x)(1 + \exp(x))$, and binary choice models specify $m_2(x) = g(x)$.

\subsection{Consistency of Link Estimation}
We demonstrate the uniform consistency of the link estimator $\hat{g}(\cdot)$ in \eqref{eq:deconv} over closed intervals. We consider the $m\mathrm{th}$-order kernel $K(\cdot)$ that satisfies
    \begin{align}
        \int_{-\infty}^\infty K(t)dt=1,~~~ \int_{-\infty}^\infty t^mK(t)dt\neq0,~\mbox{and}~
        \int_{-\infty}^\infty t^{\ell}K(t)dt=0,
    \end{align}
for $\ell\in[m-1]$.
We then obtain the following:

\begin{theorem}
\label{thm:link-est}
Suppose that Assumptions \ref{asmp:feature}--\ref{asmp:y} hold and the Fourier transform $\phi_K(t)$ of the kernel $K(\cdot)$ has a bounded support in $[-M_0, M_0]$ with some $M_0 > 0$, and the bandwidth $h_n=(c_h\log n_1)^{-1/2}$ satisfies $2M_0^2(\sigma_1 / \mu_1)^2c_h<1$.
Also, for $Z_i\sim\mathcal{N}(0,1),\ i\in[n]$ defined in Proposition \ref{prop:initial}, suppose that $(Z_1,\ldots,Z_n)^\top\sim\mathcal{N}_n(\bzero,\bI_n)$ holds. 
Then, we have the following as $n_1\to\infty$:
\begin{align}
\label{eq:rate-link}
    \sup_{a\le x\le b}|\hat{g}(x)-g(x)|=O_\mathrm{p}\rbr{\frac{1}{(\log n_1)^{m/2}}}.
\end{align}
\end{theorem}
This result shows the consistency of the link estimator.
About the convergence rate, according to \citet{fan1993nonparametric}, the logarithmic rate $O_\mathrm{p}(1/(\log n_1)^{m/2})$ reaches a lower bound, indicating that this rate cannot be improved.

Regarding the condition on $(Z_1,...,Z_n)^\top \sim\mathcal{N}_n(\bzero,\bI_n)$, it ensures that the index estimator $W_i$  can be regarded as asymptotically i.i.d. and is necessary for constructing a consistent link function estimator. We numerically discuss the plausibility of the condition using the statistical testing of independence in Appendix \ref{sec:indep}.

\subsection{Marginal Asymptotic Normality of Coefficient Estimators}
\label{sec:inference}
This section demonstrates the marginal asymptotic normality of our estimator $\hat{\bbeta}(\hat{g})$ for $\bbeta$, facilitated by the estimators of the inferential parameters, $\hat{\mu}(\hat{g})$ and $\hat{\sigma}(\hat{g})$. These results are directly applicable to hypothesis testing and the construction of confidence intervals for any finite subset of the $\beta_j$'s.

\subsubsection{Unit Covariance and $p > n$ Case}

As previously noted, the inferential parameters vary depending on the estimator considered. In this section, we focus on the ridge regularized estimator with unit covariance $\bSigma = \bI_p$. We will also present additional results for generalized covariance matrices and the ridgeless scenario later.
\begin{theorem}
\label{thm:asyN-ridge}
We consider the coefficient estimator $\hat{\bbeta}(\hat{g})$ with $J(\bb)=\lambda\|\bb\|^2/2$, and the inferential estimators $(\hat{\mu}(\hat{g}), \hat{\sigma}(\hat{g}))$, associated with the link estimator $\hat{g}(\cdot)$.
    Suppose that $\bSigma=\bI_p$, $\|\bbeta\|=1$, and  Assumptions \ref{asmp:SIM}-\ref{asmp:link} hold.
    Then, a conditional distribution of $(\hat{\bbeta}(\hat{g}), \hat{\mu}(\hat{g}), \hat{\sigma}(\hat{g}))$ with a fixed event on $\hat{g}(\cdot)$ satisfies the following: for any coordinate $j\in[p]$ satisfying 
    $\sqrt{p}\beta_j=O(1)$,
    we have
    \begin{align}
    \label{eq:asyN-ridge1}
        T_j:=\frac{\sqrt{p}({\hat{\beta}_j(\hat{g})-\hat{\mu}(\hat{g})\beta_j})}{\hat{\sigma}(\hat{g})}\dconv \mathcal{N}(0,1)
    \end{align}
    as $n,p\to\infty$ with the regime \eqref{eq:high-dim}.
    Moreover, for any finite set of coordinates $\mathcal{S}\subset[p]$ satisfying $\sqrt{p}\|\bbeta_\mathcal{S}\|=O(1)$, we have, as $n,p \to \infty$,
    \begin{align}
    \label{eq:asyN-ridge3}
        \frac{\sqrt{p}({\hat{\bbeta}_\mathcal{S}(\hat{g})-\hat{\mu}(\hat{g})\bbeta_\mathcal{S})}}{\hat\sigma(\hat{g})}\dconv \mathcal{N}(\bzero,\bI_{|\mathcal{S}|}).
    \end{align}
\end{theorem}
This result also implies that $\hat{\beta}_j(\hat{g})/\hat{\mu}(\hat{g})$ is an asymptotically unbiased estimator of $\beta_j$. Note that the convergence of the conditional distribution is ensured by the non-degeneracy property of the conditional event, as defined by $({\bX}^{(1)}, {\by}^{(1)})$; see \citet{goggin1994convergence} for details. 
{\rc
We can improve the condition $\sqrt{p}\beta_j=O(1)$ to $\beta_j=o(1)$ under the Assumption E in \citet{bellec2022observable} which derives the explicit rate of convergence for the estimation error of $\hat\mu(\cdot)$.
}

{\rc 
We highlight two key contributions of Theorem \ref{thm:asyN-ridge}.
First, it is the first result in the literature of a single-index model to demonstrate the marginal asymptotic normality of the coefficient estimator that leverages the link estimator.
Second, it remains valid even when the ratio $\kappa = p/n$ exceeds one, a notable distinction from a similar marginal asymptotic result (Theorem 5.2 in \cite{bellec2022observable}), which holds only when $\kappa$ is less than one.
Although Section 4 of \cite{bellec2022observable} addresses the case $\kappa > 1$ and considers many penalty functions, its results pertain not to marginal asymptotic normality, but rather to the average behavior of the estimator.
}
\\

\noindent
\textbf{Application to Statistical Inference}:
From Theorem \ref{thm:asyN-ridge}, we construct a confidence interval $\mathrm{CI}_{{1-\alpha}}^j$ for each $\beta_j$ with a confidence level $(1-\alpha)$ as follows:
\begin{align}
    \mathrm{CI}_{{1-\alpha}}^j := \frac{1}{\hat\mu(\hat{g})}\sbr{\hat{\beta}_j(\hat{g})-z_{(1-\alpha/2)}\frac{\hat{\sigma}(\hat{g})}{\sqrt{p}},~
    \hat{\beta}_j(\hat{g})+z_{(1-\alpha/2)}\frac{\hat{\sigma}(\hat{g})}{\sqrt{p}}},
\end{align}
where $j \in [p]$ and $z_{(1-\alpha/2)}$ is the $(1-\alpha/2)$-quantile of the standard normal distribution. This construction ensures the coverage probability adheres to the specified confidence level asymptotically.

\begin{corollary}
\label{cor:CI-consistent}
Under the settings of Theorem \ref{thm:asyN-ridge}, for any $\alpha\in(0,1)$, we have the following as $n,p \to \infty$ with the regime \eqref{eq:high-dim}:
\begin{align}
    \sup_{1\le j\le p}\abs{\Pr\rbr{\beta_j\in\mathrm{CI}_{1-\alpha}^j}-(1-\alpha)}\to0.
\end{align}
\end{corollary}
Hence, for testing the hypothesis $H_0^j:\beta_j=0$ against $H_1^j:\beta_j\neq0$ at level $\alpha\in(0,1)$, we can use the corrected $t$-statistics in \eqref{eq:asyN-ridge1}. The test that rejects the null hypothesis $H_0^j$ if
\begin{align}
    \frac{\hat\sigma(\hat{g})z_{(1-\alpha/2)}}{\sqrt{p}\tau_j}\le|\hat{\beta}_j(\hat{g})|
\end{align}
controls the asymptotic size of the test at the level $\alpha$. 
Additionally, it is feasible to develop a variable selection procedure that identifies variables related to the response. Specifically, the $p$-value associated with $H_0^j$ and the statistic $\sqrt{p}\hat{\beta}_j(\hat{g})/\hat{\sigma}(\hat{g})$ can serve as importance statistics for the $j$th covariate. This approach facilitates variable selection procedures that control the false discovery rate, as detailed in sources such as \citet{benjamini1995controlling,candes2018panning,xing2023controlling,dai2023false}.

\subsubsection{General Covariance and $p < n$ Case}
\label{sec:Sigma}
We extend Theorem \ref{thm:asyN-ridge} to scenarios with a general covariance matrix $\bSigma$ in unregularized settings. To this end, we utilize the estimators $(\hat{\mu}_0(\hat{g}), \hat{\sigma}_0(\hat{g}))$, which are defined for inferential parameters in Section \ref{sec:inferential_estimator}. Recall that the precision matrix $\bTheta$ is defined as $\bSigma^{-1}$.
\begin{theorem}
\label{thm:asyN}
    We consider the coefficient estimator $\hat{\bbeta}(\hat{g})$ with $J(\bb) \equiv 0$, and the inferential estimators $(\hat{\mu}_0(\hat{g}), \hat{\sigma}_0(\hat{g}))$, associated with the link estimator $\hat{g}(\cdot)$.
    Suppose that  Assumptions \ref{asmp:feature}-\ref{asmp:link} hold.
    Then, a conditional distribution of $(\hat{\bbeta}(\hat{g}), \hat{\mu}_0(\hat{g}), \hat{\sigma}_0(\hat{g}))$ with a fixed event on $\hat{g}(\cdot)$ satisfies the following: for any coordinate $j\in[p]$ satisfying 
    $\sqrt{p} \tau_j \beta_j=O(1)$,
    we have
    \begin{align}
    \label{eq:asyN1}
        \frac{\sqrt{p}({\hat{\beta}_j(\hat{g})-\hat{\mu}_0(\hat{g})\beta_j})}{\hat{\sigma}_0(\hat{g})/\tau_j}\dconv \mathcal{N}(0,1)
    \end{align}
     as $n,p\to\infty$ with the regime \eqref{eq:high-dim} with $\bar{\kappa} \in (0,1)$.
    Moreover, for a finite set of coordinates $\mathcal{S}\subset\{1,\ldots,p\}$, we have
    \begin{align}
    \label{eq:asyN3}
        \frac{\sqrt{p}\bTheta_\mathcal{S}^{-1/2}({\hat{\bbeta}_\mathcal{S}(\hat{g})-\hat{\mu}_0(\hat{g})\bbeta_\mathcal{S})}}{\hat\sigma_0(\hat{g})}\dconv \mathcal{N}(\bzero,\bI_{|\mathcal{S}|}),
    \end{align}
    where the submatrix $\bTheta_\mathcal{S}$ consists of $\Theta_{ij}$ for $i,j\in\mathcal{S}$.
\end{theorem}
{\rc
We can also improve the condition $\sqrt{p} \tau_j \beta_j=O(1)$ to $\tau_j \beta_j=o(1)$, under the Assumption E in \citet{bellec2022observable}.
}

\section{Experiments} \label{sec:experiments}

This section provides numerical validations of our estimation framework as outlined in Section \ref{sec:method}. The efficiency of our proposed estimator is subsequently compared with that of other estimators.

We examine two high-dimensional scenarios: $n<p$ and $n>p$. For the scenario where $n>p$, we assume the true coefficient vector follows a uniform distribution on the sphere: $\bbeta \sim \mathrm{Unif}(\mathbb{S}^{p-1})$. In the case of $n<p$, we set $\beta_1=\cdots=\beta_{100}=\sqrt{p/100}$ and $\beta_{101}=\cdots=\beta_p=0$. We generate response variables $y_i$ for Gaussian predictors $\bX_i$ with an identity covariance matrix $\bSigma = \bI_p$, under the following four scenarios:
\begin{enumerate}
\item[(i)] \textsf{Cloglog}: $y_i\mid\bX_i\sim\mathrm{Bern}(g_\mathrm{(i)}(\bbeta^\top\bX_i))$ with $g_\mathrm{(i)}(t)=1-\exp(-\exp(t))$; 
\item[(ii)] \textsf{xSqrt}: $y_i\mid\bX_i\sim\mathrm{Pois}(g_\mathrm{(ii)}(\bbeta^\top\bX_i))$ with $g_\mathrm{(ii)}(t)=t+\sqrt{t^2+1}$; 
\item[(iii)] \textsf{Cubic}: cubic regression $y_i=g_\mathrm{(iii)}(\bbeta^\top\bX_i)+\varepsilon_i$ with  $\varepsilon_i\sim\mathcal{N}(0,1/2)$ and $g_\mathrm{(iii)}(t)=t^3/3$; 
\item[(iv)] \textsf{Piecewise}: piecewise regression $y_i=g_\mathrm{(iv)}(\bbeta^\top\bX_i)+\varepsilon_i$ with $\varepsilon_i\sim\mathcal{N}(0,1/5)$ and $g_\mathrm{(iv)}(t)=(0.2t-2.3)1_{(-\infty,-1]}+2.5t1_{(-1,1)}+(0.2t+2.3)1_{[1,\infty)}$. 
\end{enumerate}

\subsection{Index Estimator}
\label{sec:sim-w}
We validate the normal approximation of the index estimator $W_i$ as shown in \eqref{eq:W}. For cases where $n > p$, we set $(n, p) = (500, 50)$ for the \textsf{Cloglog} model and $(n, p) = (500, 200)$ for the other models. For cases where $n < p$, we set $(n, p) = (250, 500)$ and apply the ridge regularized estimator to all models. We assign the maximum likelihood estimator (MLE) of logistic regression to the pilot estimator for (i) \textsf{Cloglog}, the MLE of Poisson regression for (ii) \textsf{xSqrt}, and the least squares estimator for both (iii) \textsf{Cubic} and (iv) \textsf{Piecewise} models. We calculate $\tilde\mu(\bW - \bX \bbeta) / \tilde\sigma$ using $1,000$ replications for each setup.

Figure \ref{fig:pilot} displays the results. In all settings, the difference between the index estimator and the index follows a Gaussian distribution, as expected.

\begin{figure}[ht]
  \centering
\includegraphics[keepaspectratio, width=\hsize]{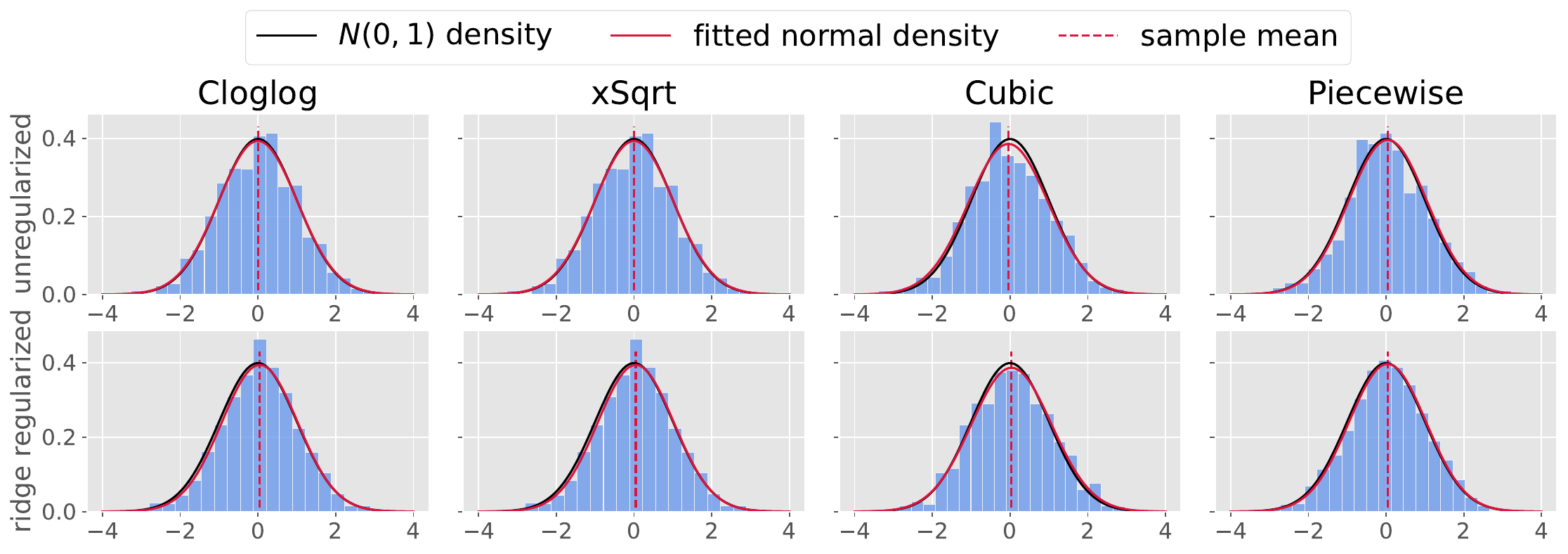}
  \caption{Histograms of the first coordinate of the statistics $\tilde\mu(\bW-\bX\bbeta)/\tilde\sigma$ over $1,000$ replications. According to Proposition \ref{prop:initial}, these histograms are expected to resemble the $\mathcal{N}(0,1)$ density. The columns correspond to each model, ranging from (i) \textsf{Cloglog} to (iv) \textsf{Piecewise}, while the rows represent unregularized and ridge-regularized estimations for cases where $n>p$ and $n<p$, respectively.
  }\label{fig:pilot}
\end{figure}

\subsection{Link Function Estimator}
\label{sec:experiment-link}
Next, we evaluate the numerical performance of the link estimator $\hat{g}(\cdot)$, constructed from $(W_1,\ldots,W_n)$, using a fixed bandwidth for each $n$. Figure \ref{fig:link} (left panel) shows that the estimation error of $\hat{g}(\cdot)$ for (iv) \textsf{Piecewise} uniformly approaches zero as the sample size increases. The right four panels in Figure \ref{fig:link} display the squared losses of $\hat{g}(\cdot)$ evaluated over the interval $[-3,3]$, which all decrease as $n$ increases,  
while their convergence rate is slow as shown in \eqref{eq:rate-link}.

\begin{figure}[ht]
  \centering
\includegraphics[keepaspectratio, width=\hsize]{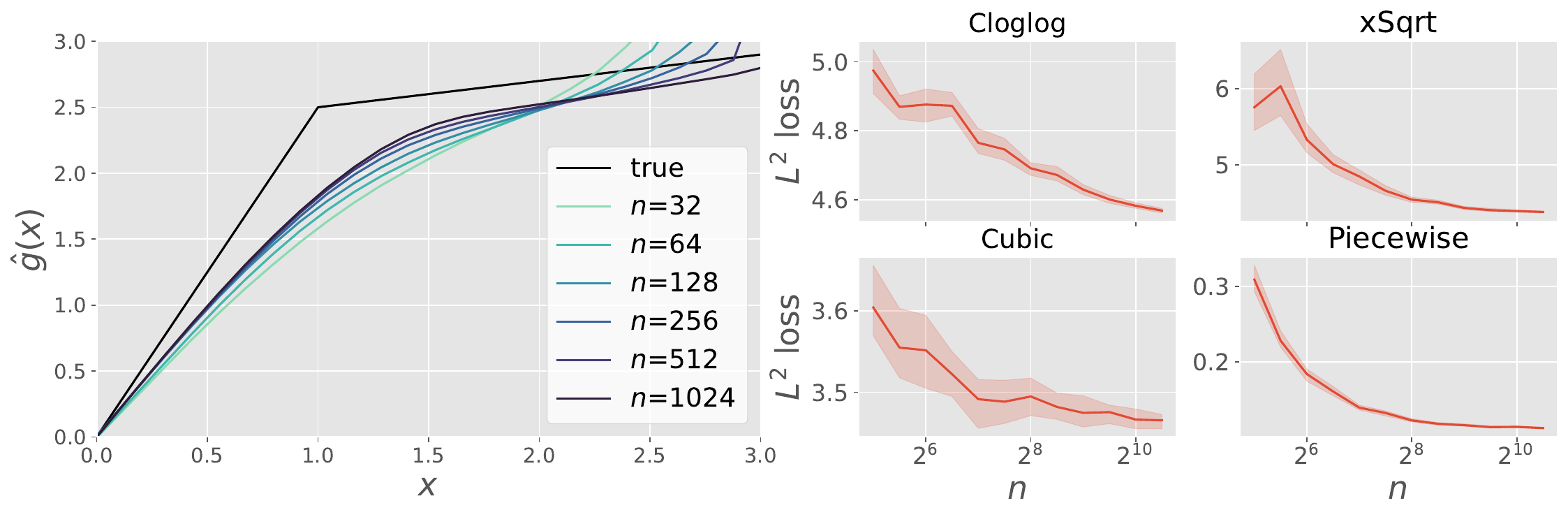}
  \caption{Estimated link functions $\hat{g}(\cdot)$ for (iv) \textsf{Piecewise} were obtained with a fixed ratio $p/n=0.6$ and $n=32, 64, \ldots, 1024$, averaged over $1,000$ replications (left). The squared loss for $\hat{g}(\cdot)$, evaluated over the interval $[-3, 3]$, for (i) \textsf{Cloglog} to (iv) \textsf{Piecewise}, as defined in the previous section, with a fixed ratio $p/n=0.4$ averaged over $1,000$ replications (right).}\label{fig:link}
\end{figure}

\subsection{Our Coefficient Estimator}
\label{sec:experiment-coef}
We examine the asymptotic normality of each coordinate of the estimator $\hat\bbeta(\hat{g})$ for the true coefficients. As in Section \ref{sec:experiment-link}, we construct the estimator using a fixed bandwidth and apply the rearrangement operator $\mR_\mathrm{r}[\cdot]$ as defined in \eqref{def:rearrangement} over the interval $[-3,3]$ to obtain $\hat{g}(\cdot)$. We then compute $\hat\bbeta(\hat{g})$ according to \eqref{eq:estimator} using $J(\cdot)\equiv\bzero$ when $n>p$ and $J(\bb)=\|\bb\|^2$ when $n\leq p$. Figure \ref{fig:beta_hat} shows the marginal normal approximation of the estimators under these conditions. All histograms closely resemble the standard normal density, corroborating the asymptotic normality stated in Theorems \ref{thm:asyN-ridge} and \ref{thm:asyN}.
Also, since we reused entire samples in each step, the numerical result suggests that the effect of data reuse is small in practice in our settings. 
\begin{figure}[ht]
  \centering
\includegraphics[keepaspectratio, width=\hsize]{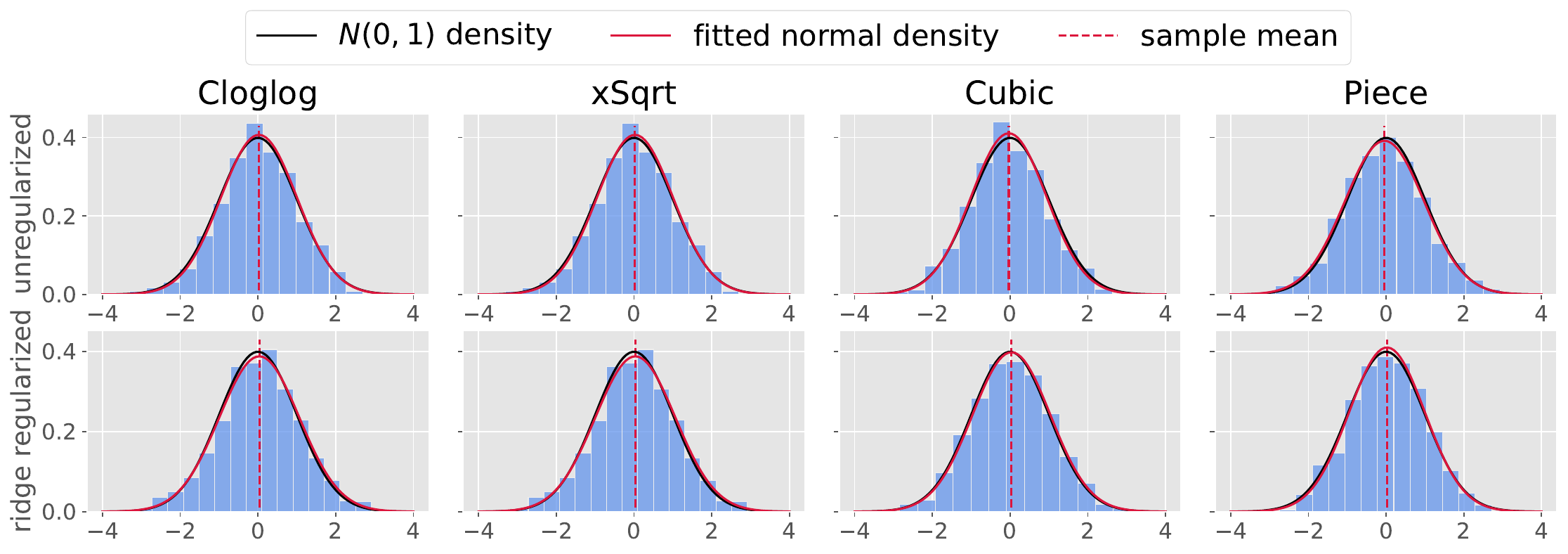}
  \caption{Histograms of the first coordinate of the statistics $\sqrt{p}(\hat\beta_1(\hat{g})-\hat{\mu}(\hat{g})\beta_1)/\hat\sigma(\hat{g})$ made from $1,000$ replications.  The columns correspond to each model from (i) \textsf{Cloglog} to (iv) \textsf{Piecewise}, and the rows correspond to unregularized and ridge-regularized estimation under $n>p$ and $n<p$, respectively. They are expected to approach $\mathcal{N}(0,1)$ density by Theorems \ref{thm:asyN-ridge} and \ref{thm:asyN}. }
\label{fig:beta_hat}
\end{figure}

\subsection{Efficiency Comparison}
Finally, we compare the estimation efficiency of the proposed estimator with several pilot estimators. 
We use the effective asymptotic variance $\sigma^2_*/\mu_*^2$ as an efficiency measure, which is the inverse of the effective signal-to-noise ratio as described in \cite{feng2022unifying}. We estimate this variance using the statistic ${\hat\bbeta^\top\hat\bbeta}/({\hat\bbeta^\top\bbeta})-1$ for an estimator $\hat\bbeta$. This statistic is a reasonable approximation of the asymptotic variance of the debiased version of $\hat\bbeta$ and converges almost surely to the effective asymptotic variance under certain conditions (see Section \ref{sec:proof_outline} for details).

We introduce new data-generating processes: \textsf{Logit}, where $y_i \mid \bX_i \sim \mathrm{Bern}(1/(1+\exp(\bbeta^\top\bX_i)))$; \textsf{Poisson}, where $y_i \mid \bX_i \sim \mathrm{Pois}(\exp(\bbeta^\top\bX_i))$; \textsf{Cubic+}, where $y_i = g_\mathrm{(iii)}(\bbeta^\top\bX_i) + \varepsilon_i$, with $\varepsilon_i \sim \mathcal{N}(5, 1/2)$; and \textsf{Piecewise+}, where $y_i = g_\mathrm{(iv)}(\bbeta^\top\bX_i) + \varepsilon_i$, with $\varepsilon_i \sim \mathcal{N}(5, 1/5)$.

From a practical perspective, we can analyze the scatter plot of $(W_i, y_i)$ and manually specify a functional form for $g(\cdot)$ to conduct parametric regression, which corresponds to \textsf{Proposed} in Table \ref{tab:efficiency}. 
We estimate parameters $a, b, c \in \R$ in different forms: $\check{g}(t) = 1/(1+\exp(-at+b))$ for case (i), $\check{g}(t) = a\exp(t)+b$ for case (ii), $\check{g}(t) = at^3 + bt^2 + ct$ for case (iii), and $\check{g}(t) = a/(1+\exp(-bt+c))-a/2$ for case (iv). We then use these estimates to construct the link function. 
These parametric regressions are expected to be useful for lightweight computation, and they avoid the sensitivity of the Nadaraya–Watson estimator to the choice of bandwidth.
Also, \textsf{Proposed2} in Table \ref{tab:efficiency} reports the proposed estimator using the deconvolution Nadaraya-Watson link estimator.

Table \ref{tab:efficiency} displays the efficiency measures for our proposed estimator and others across 100 replications. We find that our proposed estimator is more efficient in most settings, except when the estimators are specifically tailored to the models. This highlights the broad applicability of our estimator.
Theoretical insights into the efficiency comparison are deferred to Appendix \ref{sec:theory-eff-compare}.

\begin{table}[ht]
\centering
\begin{tabular}{l|ccccc}
                    & \textsf{LeastSquares}   & \textsf{LogitMLE}        & \textsf{PoisMLE} & \textsf{Proposed}        & \textsf{Proposed2} \\ \hline
\textsf{Logit}      & $3.77\pm.967$           & \textbf{.525$\pm$.157}   & --               & $.527\pm.157$            & $.583\pm.177$ \\
\textsf{Cloglog}    & $3.13\pm .599$          & $.294\pm.080$            & --               & \textbf{.271$\pm$.068}   & $.275\pm.079$ \\
\textsf{Poisson}    & $3.77\pm .967$          & --                       & \textbf{.630$\pm$.124} & \textbf{.630$\pm$.124}  & $1.32\pm.374$ \\
\textsf{xSqrt}      & $2.32\pm.692$           & --                       & {1.12$\pm$.290} & {1.12$\pm$.290}  & \textbf{1.05$\pm$.171} \\
\textsf{Cubic}      & {1.15$\pm$.258}  & --                       & --               & $1.74\pm.440$            & \textbf{1.05$\pm$.182} \\
\textsf{Cubic+}     & $33.9\pm50.7$           & --                       & --               & {1.74$\pm$.439}   & \textbf{1.03$\pm$.209} \\
\textsf{Piecewise}  & $.541\pm .031$          & --                       & --               & {.391$\pm$.157}   & \textbf{.389$\pm$.103} \\
\textsf{Piecewise+} & $6.32\pm3.26$           & --                       & --               & \textbf{.330$\pm$.184}   & $.456\pm.076$
\end{tabular}
\caption{Efficiency measure for each pair of model and estimator. We report the average $\pm$ standard deviation.}
\label{tab:efficiency}
\end{table}

\subsection{Real Data Applications}
\label{sec:real-data}
We utilize two datasets from the UCI Machine Learning Repository \citep{Dua:2019} to illustrate the performance of the proposed estimator. The DARWIN dataset \citep{cilia2018experimental} comprises handwriting data from 174 participants, including both Alzheimer's disease patients and healthy individuals. The second dataset \citep{misc_parkinson's_disease_classification_470} features 753 attributes derived from the sustained phonation of the vowel sounds of patients, both with and without Alzheimer's disease. We employ a leave-one-out strategy for splitting each dataset. For each $n-1$ subset, we compute the regularized MLE of logistic regression alongside the proposed estimate derived from it. We then estimate the effective asymptotic variance, $\sigma^2_*/\mu_*^2$, for each estimator. The results, presented in Tables \ref{tab:real1}--\ref{tab:real2}, indicate that the proposed estimator consistently provides a more accurate estimation of the true coefficient vector compared to conventional logistic regression.
Detailed results for $\lambda=0.01,0.02,\ldots,1$ are deferred to Appendix \ref{sec:real-data-detail}.

\begin{table}[]
\centering
\begin{tabular}{c|ccc}
               & $\lambda=1$   & $\lambda=5$   & $\lambda=10$  \\ \hline
\textsf{Logit} & 1.87$\pm$0.06 & 0.46$\pm$0.01 & 0.30$\pm$0.00 \\
\textsf{Proposed}       & 0.61$\pm$0.01 & 0.25$\pm$0.00 & 0.18$\pm$0.00
\end{tabular}
\caption{Estimated effective asymptotic variance of the MLE of logistic regression and the proposed estimator for DARWIN data. We provide the average $\pm$ standard deviation by using leave-one-out split datasets.}
\label{tab:real1}
\end{table}

\begin{table}[]
\centering
\begin{tabular}{c|ccc}
               & $\lambda=1$   & $\lambda=5$   & $\lambda=10$  \\ \hline
\textsf{Logit} & 2.22$\pm$0.03 & 0.30$\pm$0.00 & 0.16$\pm$0.00 \\
\textsf{Proposed}       & 0.15$\pm$0.00 & 0.06$\pm$0.00 & 0.05$\pm$0.00
\end{tabular}
\caption{Estimated effective asymptotic variance of the MLE of logistic regression and the proposed estimator for speech data. We provide the average $\pm$ standard deviation by using leave-one-out split datasets.}
\label{tab:real2}
\end{table}

\section{Proof Outline} \label{sec:proof_outline}

We outline the proofs for each theorem in Section \ref{sec:main_results}.

\subsection{Consistency of Link Estimation (Theorem \ref{thm:link-est})}
\label{sec:proof-link}
We provide an overview of the proof for Theorem \ref{thm:link-est}, which comprises two primary steps: (i) the asymptotic characteristics of the index estimator $W_i$ discussed in Section \ref{sec:index}, and (ii) demonstrating the consistency of the estimator $\hat{g}(\cdot)$ in Section \ref{sec:g}, related to $W_i$.

\subsubsection{Error of Index Estimator}
We consider the distributional approximation \eqref{eq:W} for the index estimator $W_i$, established through observable adjustments by \citet{bellec2022observable}. Theorems 4.3 and 4.4 in \citet{bellec2022observable} support Proposition \ref{prop:initial}, i.e., under suitable assumptions we have the following statement. 
Recall that we defined
\begin{align}
    W_i=\tilde\mu^{-1}\tilde\bbeta^\top{\bX_i^{(1)}}-\tilde\mu^{-1}\tilde\gamma\rbr{y_i^{(1)}-\Tilde{\bbeta}^\top{\bX_i^{(1)}}}.
\end{align}
\begin{proposition}
\label{prop:initial}
    Under Assumptions \ref{asmp:feature} and $\E[y_1^2]<\infty$, there exists $Z_i\sim\mathcal{N}(0,1)$ 
    {\rc
    independent of $\bbeta^\top\bX_i^{(1)}\sim\mathcal{N}(0,1)$}
    such that, for each $i\in[n_1]$, as $n_1\to\infty$, it holds that
    {\rc
    \begin{align}
    \label{eq:ols-bellec}
    \abs{\tilde\mu W_i-  \tilde\mu\bbeta^\top{\bX_i^{(1)}}-\tilde\sigma Z_i}\pconv0.
    \end{align}
    }
\end{proposition}
This asserts that each $\tilde{\bbeta}^\top{{\bX}_i^{(1)}}$ for $i \in [n_1]$ is approximately equal to the sum of the biased true index $\tilde{\mu}\bbeta^\top{{\bX}_i^{(1)}}$, a Gaussian error, and an additive bias term. 
We can see that the estimation error of the index estimator $W_i$ is asymptotically represented by the Gaussian term as shown in Equation \eqref{eq:W}.

\subsubsection{Error of Link Estimator}
Next, we prove the consistency of the link estimator $\hat{g}(\cdot)$ using the index estimator $W_i$.
To this aim, we define a noise-contaminated index
\begin{align}
    \Tilde{W}_i := \bbeta^\top\bX_i^{(1)} + \tilde{Z}_i,
\end{align}
where $\tilde{Z_i} \sim \mathcal{N}(0, \sigma_1^2/\mu_1^2)$ is an independent Gaussian variable.
If $W_i$ were exactly equivalent to $\Tilde{W}_i$, 
we could apply the classical result of nonparametric error-in-variables regression \citep{fan1993nonparametric} to demonstrate the uniform consistency of $\hat{g}(\cdot)$. However, this equivalence is only asymptotic as shown in \eqref{eq:ols-bellec}. Therefore, we establish that the error due to this asymptotic equivalence is negligibly small in the estimation of $\hat{g}(\cdot)$ to complete the proof.

Specifically, we take the following steps. 
First, we decompose the error of $\hat{g}(\cdot)$ into two terms.
In preparation, we define $\tilde{g}(\cdot)$ as a deconvolution estimator using a deconvolution kernel $\tilde{K}_n(x)=({2\pi})^{-1}\int_{-\infty}^\infty\exp(-\mathrm{i}tx){\phi_K(t)}/{\phi_{\varsigma}(t/h_n)}dt$ using the true inferential parameters as $\varsigma=\sigma_*/\mu_*$ (its precise definition is given in Supplementary Material). 
This estimator corresponds to the estimator for the error-in-variable setup developed by \cite{fan1993nonparametric}.
Then, from the effect of the monotonization operator, we obtain the following decomposition:
\begin{align}
    \sup_{a\le x\le b}\abs{\hat{g}(x)-g(x)} &\leq \sup_{a\le x\le b}\abs{\breve{g}(x)-{g}(x)} \\
    &\leq \sup_{a\le x\le b}\abs{\breve{g}(x)-\tilde{g}(x)} + \sup_{a\le x\le b}\abs{\tilde{g}(x)-{g}(x)}. \label{ineq:decomp_g}
\end{align}
The second term $\sup_{a\le x\le b}\abs{\tilde{g}(x)-{g}(x)}$ in \eqref{ineq:decomp_g} is the estimation error by the deconvolution estimator $\tilde{g}(\cdot)$, which is proven to be $o_{\mathrm{p}}(1)$ according to the result of \cite{fan1993nonparametric}. 

On the other hand, the first term $\sup_{a\le x\le b}\abs{\breve{g}(x)-\tilde{g}(x)}$ in \eqref{ineq:decomp_g} represents how our pre-monotonized estimator $\breve{g}(\cdot)$ in \eqref{eq:deconv} approximates the  estimator $\Tilde{g}(\cdot)$.
Rigorously, we obtain
\begin{align}
    \abs{\breve{g}(x)-\tilde{g}(x)} &\lesssim \underbrace{ \frac{1}{n_1h_n}\abs{\sum_{i=1}^{n_1}{K}_n\rbr{\frac{\tilde{W}_i-x}{h_n}}-{K}_n\rbr{\frac{{W}_i-x}{h_n}}}}_{=:T_1} \\
    & \quad + \underbrace{\frac{1}{n_1h_n}\abs{\sum_{i=1}^{n_1}{K}_n\rbr{\frac{{W}_i-x}{h_n}}-\tilde{K}_n\rbr{\frac{{W}_i-x}{h_n}}}}_{=:T_2},
\end{align}
where $\lesssim$ is an inequality up to some universal constant.
The first term $T_1$ describes the discrepancy between the estimator with 
{\rc
the index estimator ${W}_i$ 
and the contaminated index $\tilde{W}_i$. 
}
We develop an upper bound on $T_1$ by using the result of Proposition \ref{prop:initial}.
The second term $T_2$ represents the discrepancy between the convolution kernels $K_n(\cdot)$ and $\tilde{K}_n(\cdot)$.
Note that $K_n(\cdot)$ depends on the estimator $\tilde\varsigma^2=\tilde\sigma^2/\tilde\mu^2$ of the inferential parameter, and $\tilde{K}_n(\cdot)$ depends on the true value of the inferential parameter $\varsigma=\sigma_*/\mu_*$.
We derive its upper bound by evaluating the error of the estimators $\tilde{K}_n(\cdot)$.

By integrating these results into \eqref{ineq:decomp_g}, we prove that the estimation error of $\hat{g}(\cdot)$ is $O_p((\log n_1)^{-m/2})$

\subsection{Marginal Asymptotic Normality (Theorem \ref{thm:asyN})}
This section provides a proof sketch of Theorem \ref{thm:asyN}. We specifically present a general theorem that characterizes the asymptotic normality of each coordinate of the unregularized estimator in high-dimensional settings. This discussion extends the proof provided by \citet{zhao2022asymptotic} for logistic regression.

Consider the single-index model given by \eqref{eq:SIM} and an arbitrary loss function $\bar\ell: \R \times \mY \to \R$. We define an M-estimator $\bar\bbeta$, based on the loss function $\bar\ell(\cdot)$, as follows:
\begin{align}
\label{eq:ell-bar}
    \bar\bbeta\in\argmin_{\bb\in\R^p}\sum_{i=1}^n\bar\ell(\bb^\top\bX_i;y_i).
\end{align}
With this general setup, we establish the following statement:
\begin{theorem}
\label{thm:master-zsc}
    Suppose that Assumption \ref{asmp:feature} holds.  
    Also, suppose that the M-estimator $\bar\bbeta\in\R^p$ in \eqref{eq:ell-bar} is uniquely determined and there exists a constant $C>0$ such that $\Pr(\|\bar\bbeta\|<C)\ge1-o(1)$ holds. 
    With the true parameter $\bbeta\in\R^p$, define 
    \begin{align}
        \mu_{\bar\bbeta}=\frac{\bar\bbeta^\top\bSigma\bbeta}{\bbeta^\top\bSigma\bbeta},\quad\mathrm{and}\quad
        \sigma_{\bar\bbeta}^2=\|\bP_{\bSigma^{1/2}\bbeta}^\perp\bSigma^{1/2}\bar\bbeta\|^2=\|\bar\bbeta-\mu_{\bar\bbeta}\bbeta\|_{\bSigma}^2, \label{def:mu_sigma}
    \end{align}
    where $\bP_{\bSigma^{1/2}\bbeta}^\perp=\bI_p-\bSigma^{1/2}\bbeta\bbeta^\top\bSigma^{1/2}/\bbeta^\top\bSigma\bbeta$.
    Then, for any coordinates $j\in[p]$ with $\sqrt{p}\tau_j\beta_j=O(1)$, we obtain
    \begin{align}
        T_j:=\frac{\sqrt{p}(\bar{\beta}_j-\mu_{\bar\bbeta}\beta_j)}{\sigma_{\bar\bbeta}/\tau_j}\dconv\mathcal{N}(0,1).
    \end{align}
    as $n,p\to\infty$ with $p/n\to\bar\kappa<1$.
\end{theorem}
This theorem establishes the marginal asymptotic normality for a broad class of estimators defined by the minimization of convex loss functions. Additionally, it demonstrates that the limiting distributional behavior of $\bar{\bbeta}$ is characterized by $\mu_{\bar{\bbeta}}$ and $\sigma_{\bar{\bbeta}}^2$ in the high-dimensional setting \eqref{eq:high-dim}. 
Intuitively, $\mu_{\bar\bbeta}$ is a scaled inner product of $\bar\bbeta$ and $\bbeta$, and $\sigma_{\bar\bbeta}^2$ denotes the magnitude of the orthogonal component of $\bar\bbeta$ to $\bbeta$. 

The rigorous proof in Supplementary Material is conducted in the following steps:
\begin{enumerate}
    \item[(i)] Since we have $\bbeta^\top\bX_i=(\bSigma^{-1/2}\bX_i)^\top(\bSigma^{1/2}\bbeta)$, we achieve the replacements $\bX_i$ to $\bSigma^{-1/2}\bX_i\sim\mathcal{N}(\bzero,\bI_p)$, $\bbeta$ to $\bmeta=\bSigma^{1/2}\bbeta$, and $\bar\bbeta$ to $\hat\bmeta=\bSigma^{1/2}\bar\bbeta$. From the Cholesky factorization of $\bSigma$, we have
    \begin{align}
        T_j=\frac{\sqrt{p}(\bar{\beta}_j-\mu_{\bar\bbeta}\beta_j)}{\sigma_{\bar\bbeta}/\tau_j}
        =\frac{\sqrt{p}(\hat{\eta}_j-\mu_{\bar\bbeta}\eta_j)}{\sigma_{\bar\bbeta}}.
    \end{align}
    \item[(ii)] Considering the rotation $\bU$ around $\bmeta$ (i.e., $\bU\bmeta=\bmeta$ and $\bU\bU^\top=\bI_p$), several calculations give, for $\bT:=(T_1,\ldots,T_p)^\top/\sqrt{p}$,
    \begin{align}
        \bT=\frac{\bP_{\bmeta}^\perp\hat\bmeta}{\|\bP_{\bmeta}^\perp\hat\bmeta\|}
        \overset{\rm d}{=}\frac{\bU\bP_{\bmeta}^\perp\hat\bmeta}{\|\bP_{\bmeta}^\perp\hat\bmeta\|}.
    \end{align}
    This means that $\bT$ is uniformly distributed on the unit sphere in $\bmeta^\perp$ (See Figure \ref{fig:proof}).
    \item[(iii)] Drawing on the analogy to the asymptotic equivalence between the $p$-dimensional standard normal distribution and $\mathrm{Unif}(\sqrt{p}\mathbb{S}^{p-1})$, we obtain the asymptotic normality of $T_j$. 
\end{enumerate}

We apply this general theorem to obtain Theorem \ref{thm:asyN}. 
A similar argument implies Theorem \ref{thm:asyN-ridge} for the regularized estimator.

\begin{figure}[htbp]
  \centering
\includegraphics[keepaspectratio, width=8cm]{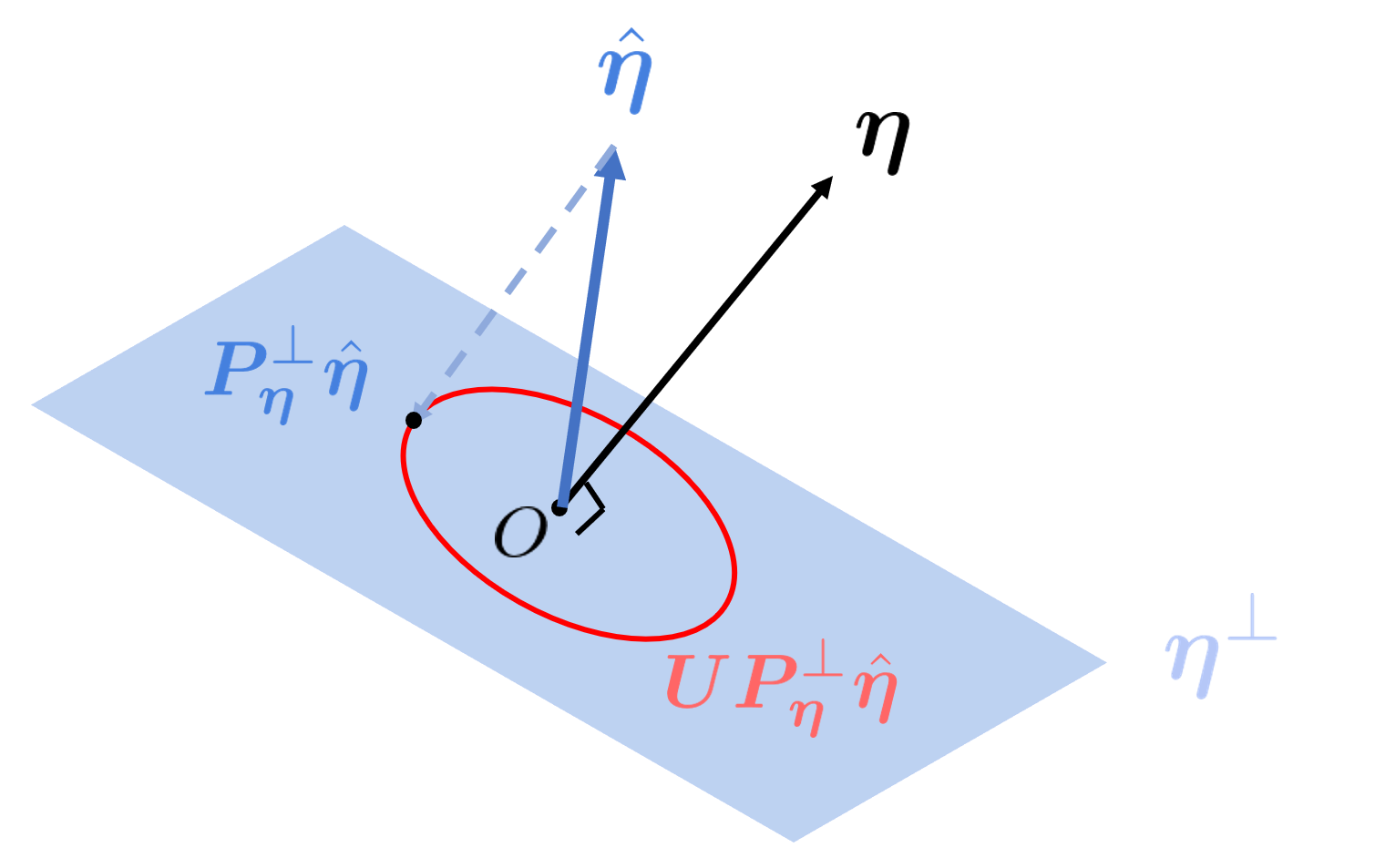}
  \caption{Illustration of the proof technique of Theorem \ref{thm:master-zsc}. $\mu_{\bar\bbeta}$ is the inner product of ($\bmeta, \hat\bmeta$). The radius of the set depicted by the red circle corresponds to $\sigma_{\bar\bbeta}$.}\label{fig:proof}
\end{figure}

The proof strategy for the marginal asymptotic normality developed here is fundamentally different from Bellec’s approach \citep{bellec2022observable}, which is based on second-order Stein's formulae \citep{bellec2021second,bellec2022derivatives}.

\section{Other Design of Pilot Estimator} \label{sec:pilot}

We can consider alternative estimators as the pilot estimator discussed in Section \ref{sec:index}. Depending on the context, choosing an appropriate pilot estimator can enhance the asymptotic efficiency of the overall estimation process. Below, we list the various estimator options and their associated values necessary for estimating their inferential parameters.

\subsection{Least Squares Estimators}

In the case of $\kappa_1<1$, we can use the least squares estimator 
\begin{align}
    \tilde\bbeta_\mathrm{LS}=({(\bX^{(1)}})^\top\bX^{(1)})^{-1}{(\bX^{(1)}})^\top\by.
\end{align}
In this case, there exist corresponding inferential parameters of $\tilde\bbeta_\mathrm{LS}$.

We obtain the following marginal asymptotic normality of the least-squares estimator.
We recall the definition of inferential parameters in \eqref{def:mu_sigma} and consider the corresponding parameter $\mu_{\Tilde{\bbeta}_{\mathrm{LS}}}$ and $\sigma_{\Tilde{\bbeta}_{\mathrm{LS}}}$ by substituting $\Tilde{\bbeta}_{\mathrm{LS}}$.
Then, we obtain the following result by a straightforward application of Theorem \ref{thm:master-zsc}.
\begin{corollary}
\label{cor:pilot-asyN}
    Suppose Assumption \ref{asmp:feature} holds.  
    Then, for any coordinates $j=1,\ldots,p$ obeying $\sqrt{p}\tau_j\beta_j=O(1)$, 
    we have the following as $n,p\to\infty$ with {\rc $\lim_{n\to\infty} \kappa_1\in(0,1)$}:
    \begin{align}
        \frac{\sqrt{p}(\tilde\beta_{\mathrm{LS},j}-\mu_{{\Tilde{\bbeta}_{\mathrm{LS}}}}\beta_j)}{{\sigma_{\Tilde{\bbeta}_{\mathrm{LS}}}}/\tau_j}\dconv\mathcal{N}(0,1).
    \end{align}
\end{corollary}

We also define the following values $(\Tilde{\gamma}_\mathrm{LS}, \tilde{\mu}_\mathrm{LS}, \tilde{\sigma}^2_\mathrm{LS})$ to estimate the inferential parameters $\mu_{\Tilde{\bbeta}_{\mathrm{LS}}}$ and $\sigma_{\Tilde{\bbeta}_{\mathrm{LS}}}$.
Namely, we define $\tilde\gamma_\mathrm{LS}={\kappa_1}/({1-\kappa_1})$, and
\begin{align}
        \tilde\mu_\mathrm{LS}=\abs{\frac{\|\bX^{(1)}\Tilde{\bbeta}_\mathrm{LS}\|^2}{n_1}-(1-\kappa_1)\tilde\sigma_\mathrm{LS}^2}^{1/2},\tilde\sigma_\mathrm{LS}^2=\frac{\tilde\gamma_\mathrm{LS}}{n_1(1-\kappa_1)}\|\by^{(1)}-\bX^{(1)}\Tilde{\bbeta}_\mathrm{LS}\|^2.
\end{align}
If we employ the least squares estimator $\Tilde{\bbeta}_{\mathrm{LS}}$ as the pilot estimator in Section \ref{sec:index}, we replace $(\tilde{\mu}, \tilde{\sigma}^2)$ for the index estimator $W_i$ in \eqref{def:W} by $(\tilde{\mu}_\mathrm{LS}, \tilde{\sigma}^2_\mathrm{LS})$.

\subsection{Maximum Likelihood Estimators}

When $y_i$ takes discrete values, a more appropriate pilot estimator can be proposed. 
For binary outcomes such as in classification problems where $y_i \in \{0,1\}$, we can employ MLE for logistic regression:
\begin{align}
    \tilde\bbeta_\mathrm{mle}\in\argmin_{\bb\in\R^p}\sum_{i=1}^{n_1}\log(1+\exp(\bb^\top{\bX_i^{(1)}}))-y_i\bb^\top{\bX_i^{(1)}}.
\end{align}
In the case with $y_i \in \N\cup\{0\}$, we can consider the MLE for the Poisson regression 
\begin{align}
    \tilde\bbeta_\mathrm{mle}\in\argmin_{\bb\in\R^p}\sum_{i=1}^{n_1}\exp(\bb^\top{\bX_i^{(1)}})-y_i\bb^\top{\bX_i^{(1)}}.
\end{align}
Its asymptotic normality is obtained by applying Theorem \ref{thm:master-zsc}.
\begin{corollary}
\label{cor:pilot-mle-asyN}
    Under Assumption \ref{asmp:feature}, 
    {\rc suppose that $\tilde\bbeta_\mathrm{mle}$ is uniquely determined and there exists a constant $C>0$ such that $\Pr(\|\tilde\bbeta_\mathrm{mle}\|<C)\ge1-o(1)$ holds as $n,p \to \infty$. Then,}
    for any coordinates $j=1,\ldots,p$ obeying $\sqrt{p}\tau_j\beta_j=O(1)$, we have the following as $n,p\to\infty$ with {\rc $\lim_{n\to\infty} \kappa_1\in(0,1)$}:
    \begin{align}
        \frac{\sqrt{p}(\tilde\beta_{\mathrm{mle},j}-\mu_{{\Tilde{\bbeta}_{\mathrm{mle}}}}\beta_j)}{{\sigma_{\Tilde{\bbeta}_{\mathrm{mle}}}}/\tau_j}\dconv\mathcal{N}(0,1),
    \end{align}
\end{corollary}

In these cases, we can define values $(\tilde\gamma_\mathrm{mle}, \tilde\mu_\mathrm{mle}, \tilde\sigma_\mathrm{mle})$ for estimating their inferential parameters.
Define $g_0(x)=1/(1+\exp(-x))$ for logistic regression and $g_0(x)=\exp(x)$ for Poisson regression. 
Then, we define the values as $\tilde\gamma_\mathrm{mle}=\kappa_1\tilde{v}_\mathrm{mle}^{-1}$ and
\begin{align}
    \tilde\mu_\mathrm{mle}=\abs{\frac{\|\bX^{(1)}\tilde\bbeta_\mathrm{mle}\|^2}{n_1}-(1-\kappa_1)\tilde\sigma_\mathrm{mle}^2}^{1/2},~~~
    \tilde\sigma_\mathrm{mle}^2=\frac{\|\by^{(1)}-g_0(\bX^{(1)}\Tilde{\bbeta}_\mathrm{mle})\|^2}{n_1\tilde{v}_\mathrm{mle}/\kappa_1},
\end{align}
where we define $\Tilde{v}_\mathrm{mle}=n_1^{-1}\trace(\Tilde{\bD}-\Tilde{\bD}\bX^{(1)}((\bX^{(1)})^\top\Tilde{\bD}\bX^{(1)})^{-1}(\bX^{(1)})^\top\Tilde{\bD})$ and $\Tilde{\bD}=\diag(g_0'(\bX^{(1)}\tilde\bbeta_\mathrm{mle}))$.
Based on this definition, we can develop a corresponding index estimator by replacing $(\Tilde{\mu}, \tilde{\sigma})$ in \eqref{eq:W} by $\tilde\mu_\mathrm{mle}$ and $\tilde\sigma_\mathrm{mle}$.

\section{Conclusion and Discussion} \label{sec:discussion}
This study establishes a novel statistical inference procedure for high-dimensional single-index models. Specifically, we develop a consistent estimation method for the link function. Furthermore, using the estimated link function, we formulate an efficient estimator and confirm its marginal asymptotic normality. This verification allows for the accurate construction of confidence intervals and $p$-values for any finite collection of coordinates.

We identify several avenues for future research: (a) extending these results to cases where the covariate distribution is non-Gaussian, (b) generalizing our findings to multi-index models, and (c) confirming the marginal asymptotic normality of our proposed estimators under any form of regularization and covariance. These prospects offer intriguing possibilities for further exploration.
{\rc
Particularly, (c) can be achieved by imposing the exchangeability condition on the true coefficient $\bbeta$ and employing a technique similar to the proof of Corollary 3.8 in \citet{li2023spectrum}.
}

\appendix

\section{Effect of Link Estimation on Inferential Parameters} \label{sec:link_estimation_inferential}
The following theorem reveals that the estimation error of the link function is asymptotically negligible with respect to the observable adjustments.

Specifically, we consider a slightly modified version of the inferential estimator.
In preparation, we define a censoring operator $\iota: \R \to \R$ on a interval $[a,b] \subset \R$ as $\iota(z) = \max(a,\min(b,z))$.
Then, for any $\Bar{g}:\R \to \R$, we define a truncation version of $\bD$ as $\bD_c=\diag(\bar{g}'(\iota(\bX^{(2)}\hat{\bbeta}(\bar{g}))))$, and $\hat{v}_{0c} =n_2^{-1}\trace(\bD_c-\bD_c\bX^{(2)}(({\bX^{(2)}})^\top\bD_c\bX^{(2)})^{-1}({\bX^{(2)}})^\top\bD_c)$.
Further, in the case of $J(\cdot)  
\equiv\bzero$, we define the modified estimator as
\begin{align}
    \hat{\mu}_{0c}(\bar{g})&=\abs{{\|\iota(\bX^{(2)}\hat\bbeta(\bar{g}))\|^2}/n_2-(1-\kappa_2)\hat{\sigma}_{0c}^2(\bar{g})}^{1/2} ,\label{eq:mu-hat-tr} \mbox{~~and}\\
    \hat{\sigma}_{0c}^2(\bar{g})&=  \frac{\|\by^{(2)}-\bar{g}(\iota(\bX^{(2)}\hat\bbeta(\bar{g})))\|^2}{n_2 \hat{v}_{0c}^2/\kappa_2}.\label{eq:sigma-hat-tr}
\end{align}

Using the modified definition, we obtain the following result.
\begin{theorem}
\label{thm:adj-equiv}
Suppose that $J(\cdot)\equiv\bzero$ holds and the estimator \eqref{eq:estimator} exists. 
Further, suppose that Assumptions \ref{asmp:feature}-\ref{asmp:y} hold.
Then, we have the following as $n_1 \to \infty$:
\begin{align}
\label{eq:mu-lip-g}
    \abs{\hat{\mu}_{0c}(\hat{g})-\hat{\mu}_{0c}(g)}\pconv0,\quad\mathrm{and}\quad\abs{\hat{\sigma}_{0c}^2(\hat{g})-\hat{\sigma}_{0c}^2(g)}\pconv0.
\end{align}
\end{theorem}
This result indicates that, since the link estimator $\hat{g}(\cdot)$ is consistent, we can estimate the inferential parameters under the true link $g(\cdot)$.

The difficulty in this proof arises from the dependence between the elements of the estimator, which cannot be handled by the triangle inequality or H\"{o}lder's inequality, 
To overcome the difficulty, we utilize the Azuma-Hoeffding inequality for martingale difference sequences.

\section{Theoretical Efficiency Comparison}
\label{sec:theory-eff-compare}

We compare the efficiency of our estimator $\hat\bbeta(\hat{g})$ with that of the ridge estimator $\tilde{\bbeta}$ as the pilot.
As shown in \citet{bellec2022observable}, the ridge estimator is a valid estimator for the single-index model in the high-dimensional scheme \eqref{eq:high-dim} even without estimating the link function $g(\cdot)$. 

To the aim, we define the \textit{effective asymptotic variance} based on inferential parameters, which is a ratio of the asymptotic bias and the asymptotic variance. That is, our estimator $\hat{\bbeta}(\hat{g})$ has its effective asymptotic variance $\hat\sigma^2(\hat{g})/\hat\mu^2(\hat{g})$, and the ridge estimator $\tilde{\bbeta}$ has $\tilde\sigma^2/\tilde\mu^2$.
The effective asymptotic variance corresponds to the asymptotic variance of each coordinate of the estimators with bias correction.

We give the following result for necessary and sufficient conditions for the proposed estimator to be more efficient than the least squares estimator and the ridge estimator. 

\begin{proposition}
\label{prop:efficiency}
    We consider the coefficient estimator $\hat{\bbeta}(\hat{g})$ with $J(\bb)\equiv\lambda\|\bb\|^2$ and the setup $n_1=n_2$.
    We use the regularization parameter $\lambda_1>0$ for the pilot estimator $\tilde{\bbeta}$.
    Suppose that Assumptions \ref{asmp:feature}-\ref{asmp:link} are fulfilled.
    Then, $\hat\sigma^2(\hat{g})/\hat\mu^2(\hat{g})< \tilde\sigma^2/\tilde\mu^2$ holds if and only if we have
    \begin{align}
        \frac{\|\hat\bbeta(\hat{g})\|}{\|\tilde\bbeta\|}\cdot
        \frac{\abs{\hat{v}_\lambda+\lambda}}{\abs{\Tilde{v}+\lambda_1}}\cdot
        \frac{\|\by^{(1)}-\bX^{(1)}\tilde\bbeta\|}{\|\by^{(2)}-\hat{g}(\bX^{(2)}\hat\bbeta(\hat{g}))\|}>1.
    \end{align}   
\end{proposition}
This necessary and sufficient condition suggests that our estimator may have an advantage by exploiting the nonlinearity of the link function $g(\cdot)$. 
The first reason is that, when $\by$ has nonlinearity in $\bX \bbeta$, the residual $\|\by-\hat{g}(\bX\hat\bbeta)\|^2$ of the proposed method is expected to be asymptotically smaller than $\|\by-\bX\tilde\bbeta\|^2$. 
The second reason is that $\hat{v}_\lambda$ approximates the gradient mean $ n^{-1}\sum_{i=1}^n\hat{g}'(\bX_i^\top\hat\bbeta(\hat{g}))$, so this element increases when $g(\cdot)$ has a large gradient. Using these facts, the proposed method incorporates the nonlinearity of $g(\cdot)$ and helps improve efficiency.

\begin{proposition} \label{prop:efficiency_additional}
If $J(\cdot)\equiv\bzero$, $n_1=n_2$, and  are fulfilled, then $\hat\sigma_0^2(\hat{g})/\hat\mu_0^2(\hat{g})<\tilde\sigma_\mathrm{LS}^2/\tilde\mu_\mathrm{LS}^2$ if and only if
    \begin{align}
        {\frac{\|\bX^{(2)}\hat\bbeta(\hat{g})\|}{\|\bX^{(1)}\tilde\bbeta_\mathrm{LS}\|}
        \cdot\frac{\abs{\hat{v}_0}}{1-\kappa_1}
        \cdot\frac{\|\by^{(1)}-\bX^{(1)}\tilde\bbeta_\mathrm{LS}\|}{\|\by^{(2)}-\hat{g}(\bX^{(2)}\hat\bbeta(\hat{g}))\|}}>1.
    \end{align}
\end{proposition}

This result does not imply that the estimator $\hat\bbeta(\hat{g})$ always improves the pilot estimator $\tilde\bbeta$. Indeed, in the setting of linear regression models, \citet{bean2013optimal} proves that the maximum likelihood estimator is not efficient in the regime \eqref{eq:high-dim}. Instead, they provide the form of the optimal loss function that depends on the dimensionality $\kappa$.

As an alternative approach to efficiency comparison, one could characterize the limits of $\hat{\mu}(\hat{g})$ and $\hat{\sigma}(\hat{g})$ as solutions to self-consistency equations, in the spirit of analyses based on AMP or CGMT \citep{donoho2016high,thrampoulidis2018precise,loureiro2021learning,feng2022unifying}. However, due to the strong nonlinearity of these equations, deriving tractable sufficient conditions under which the proposed method provably dominates pilot estimators is extremely challenging, and in practice one must rely on numerical solutions. An illustrative example is shown in Figure \ref{fig:compare}, where we plot efficiency measures for ordinary least squares (OLS) and the maximum likelihood estimator (MLE) applied to data generated from a logistic regression model, comparing theoretical values obtained from numerical solutions of the self-consistency equations with empirical values computed via observable adjustments. The two are nearly indistinguishable. From a practical standpoint, this suggests that leveraging observable adjustments, which do not require knowledge of the true data-generating process, provides a preferable and effective means of assessing efficiency.
\begin{figure}[htbp]
  \begin{center}
    \includegraphics[clip,width=0.65\columnwidth]{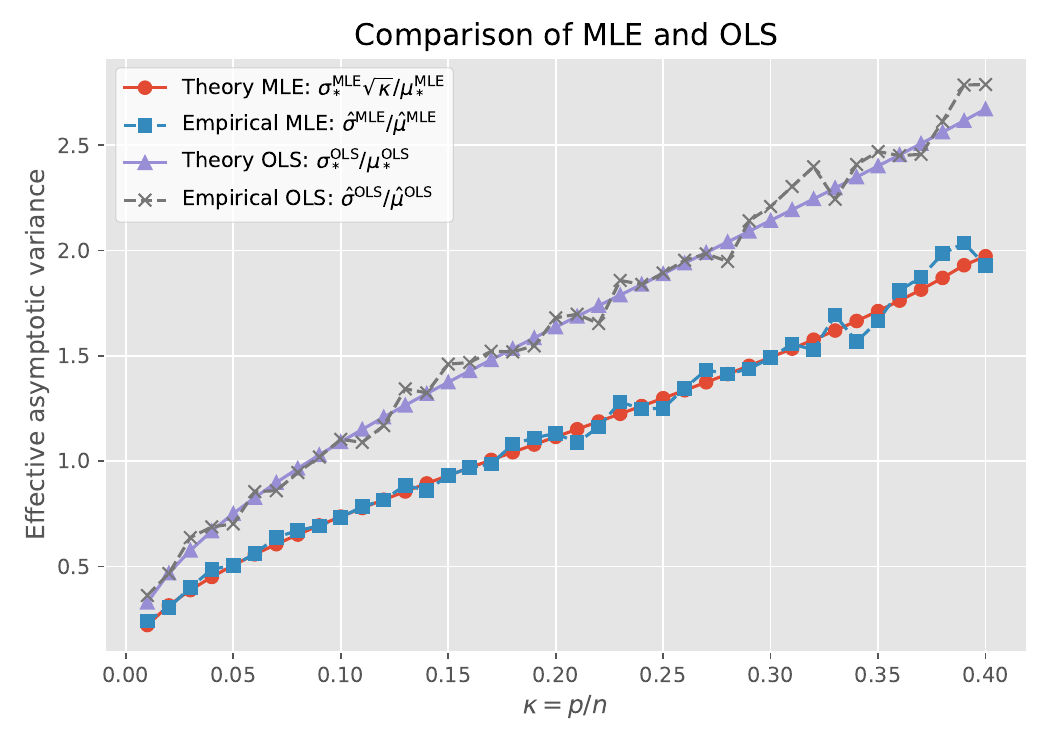}
  \end{center}
  \caption{Comparison between MLE and OLS for the simulated data following $y_i\mid\bX_i\sim\mathrm{Bern}(1/(1+\exp(-\bbeta^\top\bX_i)))$ with $\bX_i\iid\mathcal{N}(\bzero,\bI_p)$ for $i\in[n]$, $\beta_j\iid\mathcal{N}(0,p^{-1})$ for $j\in[p]$, $n=10000$, and $p=\lfloor\kappa n\rfloor$.
}\label{fig:compare}
\end{figure}

{
\section{Remarks on the condition in Theorem \ref{thm:link-est}}
\label{sec:indep}
In this section, we discuss the validity of the condition on $(Z_1,...,Z_n)$ in Theorem \ref{thm:link-est} by numerical experiments. To this end, we conduct independence tests on $(W_1,\ldots,W_n)$ across various models and estimators. Specifically, we perform a permutation test as described below. For each model defined in Section \ref{sec:experiments}, we fix $n=1000$ and vary $p$ for $\kappa=p/n=0.1,0.5,1,2$.
\begin{enumerate}
    \item We generate samples following each model, and compute the corresponding $W_i,\ i\in[n]$.
    \item We repeat it 100 times for independently generated samples and make a $100\times n$ matrix of $W_i$'s, say $\breve\bW$.
    \item We randomly permute the columns of the matrix $\breve\bW$ 1000 times. For each permutation, we compute the dependence measure $\frac{1}{n}\sum_{i=1}^n\max_{i'\neq i}|\hat\bC_{i,i'}|$ where $\hat{\bC}=\frac{1}{100}\breve\bW^\top\breve\bW$, called Mean Maximum (absolute) Correlation (MMC).
    \item We compute $p$-values as the proportion of times the MMC of the original $\breve\bW$ falls below the MMC from 1000 permutations.
\end{enumerate}
From the result displayed in Table \ref{tab:indep}, the null hypotheses of independence are not rejected at conventional significance levels, thus numerically, there is no evidence suggesting that the condition does not hold.

\begin{table}[]
\centering
\caption{p-values of statistical tests on independence}
\label{tab:indep}
\begin{tabular}{l|l|l|l}
             & \multicolumn{1}{c|}{\textsf{logit}} & \multicolumn{1}{c|}{\textsf{Pois}} & \textsf{Cubic} \\ \hline \hline
$\kappa=0.1$ & 0.759                      & 0.441                     & 0.478 \\ 
$\kappa=0.5$ & 0.912                      & 0.123                     & 0.383 \\ 
$\kappa=1$   & 0.602                      & 0.060                     & 0.825 \\ 
$\kappa=2$   & 0.931                      & 0.531                     & 0.704
\end{tabular}
\end{table}

}

\section{Nonparametric Regression with Deconvolution}
\label{sec:deconv}
In this section, we review the concept of nonparametric regression with deconvolution to address the errors-in-variable problem. To begin with, we redefine the notation only for this section. For a pair of random variables 
$(X,Y,Z)$, suppose that the model is
\begin{align}
    \E[Z\mid X=x]=m(x),
\end{align}
and that we can only observe $n$ i.i.d. realizations of $Y=X+\varepsilon$ and $Z$. Here, $\varepsilon$ is a random variable called measurement error or error in variables. For the identification, we assume that the distribution of $\varepsilon$ is known. 
Let the joint distribution of $(X,Z)$ be $f(x,z)$. By the definition of the conditional expectations, $m(x)={r(x)}/{f(x)}$ with
\begin{align}
    r(x)=\int_{-\infty}^\infty zf(x,z)dz,\quad f(x)=\int_{-\infty}^\infty f(x,z)dz,
\end{align}
for the continuous random variables. The goal of the problem is to estimate the function $m(\cdot)$.

If we could observe $X$, a popular estimator of $m(x)$ is Nadaraya-Watson estimator ${\tilde{r}(x)}/{\tilde{f}(x)}$ with
\begin{align}
    \tilde{r}(x)=\frac{1}{nh_n}\sum_{i=1}^nZ_iK\rbr{\frac{x-X_i}{h_n}},\quad \tilde{f}(x)=\frac{1}{nh_n}\sum_{i=1}^nK\rbr{\frac{x-X_i}{h_n}},
\end{align}
where $K(\cdot)$ is a kernel function and $h_n$ is the bandwidth. Since $X$ is unobservable, we alternatively construct the \textit{deconvolution} estimator \citep{stefanski1990deconvolving}. Let the characteristic function of $X$, $Y$ and $\varepsilon$ be 
$\phi_X(\cdot)$, $\phi_Y(\cdot)$ and $\phi_\varepsilon(\cdot)$, respectively. Since the density of $Y$ is the convolution of that of $X$ and $\varepsilon$, and the convolution in the frequency domain is just a multiplication, we have $\phi_X(t)=\phi_Y(t)/\phi_\varepsilon(t)$. Thus, the inverse Fourier transform of $\phi_Y(t)/\phi_\varepsilon(t)$ gives the density of $X$. Since we know the distribution of $\varepsilon$ and we can approximate $\phi_Y(t)$ by the characteristic function of the kernel density estimator of $Y$, we can construct an estimator of $f(x)$ as
\begin{align}
\label{eq:decnv-kde}
    \hat{f}(x)=\frac{1}{2\pi}\int_{-\infty}^\infty\exp(-\mathrm{i}tx)\phi_K(th_n)\frac{\hat\phi_Y(t)}{\phi_\varepsilon(t)}dt,
\end{align}
where we use the fact that the Fourier transform of $\tilde{f}_Y(y)=(nh_n)^{-1}\sum_{i=1}^nK((y-Y_i)/h_n)$ is $\phi_K(th_n)\hat{\phi}_Y(t)$, which approximates $\phi_Y(\cdot)$.
Here, $\hat\phi_Y(t)$ is the empirical characteristic function:
\begin{align}
    \hat\phi_Y(t)=\frac{1}{n}\sum_{i=1}^n\exp(\mathrm{i}tY_i).
\end{align}
We can rewrite \eqref{eq:decnv-kde} in a kernel form
\begin{align}
    \hat{f}(x)=\frac{1}{nh_n}\sum_{i=1}^nK_n\rbr{\frac{x-Y_i}{h_n}},
\end{align}
with
\begin{align}
    K_n(x)=\frac{1}{2\pi}\int_{-\infty}^\infty\exp(-\mathrm{i}tx)\frac{\phi_K(t)}{\phi_\varepsilon(t/h_n)}dt.
\end{align}
Using this, \citet{fan1993nonparametric} proposes a kernel regression estimator $\hat{m}(x)=\hat{r}(x)/\hat{f}(x)$ involving errors in variables with
\begin{align}
\hat{r}(x)=\frac{1}{nh_n}\sum_{i=1}^nZ_iK_n\rbr{\frac{x-Y_i}{h_n}}.
\end{align}
To establish the theoretical guarantee, we impose the following assumptions:
\begin{enumerate}
    \item[(N1)] (Super-smoothness of the distribution of $\varepsilon$) There exists constants $d_0,d_1,\beta,\gamma>0$ and $\beta_0,\beta_1\in\R$ satisfying, as $t\to\infty$,
    \begin{align}
        d_0\abs{t}^{\beta_0}\exp(-\abs{t}^\beta/\gamma)
        \le\abs{\phi_\varepsilon(t)}
        \le d_1\abs{t}^{\beta_1}\exp(-\abs{t}^\beta/\gamma).
    \end{align}
    \item[(N2)] The characteristic function of the error distribution $\phi_\varepsilon(\cdot)$ does not vanish.
    \item[(N3)] Let $a<b$. The marginal density $f_X(\cdot)$ of the unobserved $X$ is bounded away from zero on the interval $[a,b]$, and has a bounded $k$-th derivative.
    \item[(N4)] The true regression function $m(\cdot)$ has a continuous $k$-th derivative on $[a,b]$.
    \item[(N5)] The conditional second moment $\E[Z^2\mid X=x]$ is continuous on $[a,b]$, and $\E[Z^2]<\infty$.
    \item[(N6)] The kernel $K(\cdot)$ is a $k$-th order kernel. Namely,for $j=1,\ldots,k-1$, it holds that
    \begin{align}
        \int_{-\infty}^\infty K(t)dt=1,\quad \int_{-\infty}^\infty t^kK(t)dt\neq0,\quad
        \int_{-\infty}^\infty t^jK(t)dt=0.
    \end{align}
\end{enumerate}
(N1) includes Gaussian distributions for $\beta=2$ and Cauchy distributions for $\beta=1$.
For a positive constant $B$, define a set of function
\begin{align}
    \mathcal{F}=\cbr{f(x,z):
    \abs{f_X^{(k)}(\cdot)}\le B,
    \min_{a\le x\le b}f_X(x)\ge B^{-1}, 
    \sup_{a\le x\le b}\abs{m^{(j)}(x)}\le B, j=0,1,\ldots,k}.
\end{align}
In this setting, we have the uniform consistency of $\hat{m}(\cdot)$ and its rate of convergence.
\begin{lemma}[Theorem 2 in \citet{fan1993nonparametric}]
\label{lem:fan1993}
Assume (N1)-(N6) and that $\phi_K(t)$ has a bounded support on $\abs{t}<M_0$. Then, for bandwidth $h_n=c(\log n)^{-1/\beta}$ with $c>M_0(2/\gamma)^{1/\beta}$,
\begin{align}
    \lim_{d\to\infty}\limsup_{n\to\infty}\Pr\rbr{\sup_{a\le x\le b}|\hat{m}(x)-m(x)|\ge d(\log n)^{-k/\beta}}=0,
\end{align}
holds for any $f\in\mathcal{F}$.
\end{lemma}
Furthermore, we can show the uniform convergence of the derivative of $\hat{m}(\cdot)$. We use the following result in the proof of Theorem \ref{thm:link-est}.
\begin{lemma}
\label{lem:deriv-unif}
Under the condition of Lemma \ref{lem:fan1993}, we have, for any $f\in\mathcal{F}$,
\begin{align}
    \sup_{a\le x\le b}|\hat{m}'(x)-m'(x)|\pconv0.
\end{align}
\end{lemma}
To prove this, we use the following two lemmas.
\begin{lemma}
\label{lem:phiy}  We have, for any $t\in\R$,
\begin{align}
    \E\sbr{\abs{\hat{\phi}_Y(t)-\phi_Y(t)}^2}\le n^{-1},
\end{align}
and
\begin{align}
    \E\sbr{\abs{\frac{1}{n}\sum_{i=1}^nZ_i\exp(\mathrm{i}tY_i)-\E[Z\exp(\mathrm{i}tY)]}^2}\le n^{-1}\E[Z^2].
\end{align}
\end{lemma}
\begin{proof}[Proof of Lemma \ref{lem:phiy}] We decompose the term on the left-hand side in the first statement by Euler's formula as
\begin{align}
    &\E\sbr{\abs{\hat{\phi}_Y(t)-\phi_Y(t)}^2}\\
    &=\E\sbr{\abs{\frac{1}{n}\sum_{i=1}^ne^{\mathrm{i}tY_i}-\E e^{\mathrm{i}tY}}^2}\\
    &=\E\sbr{\abs{\frac{1}{n}\sum_{i=1}^n\cbr{\cos(tY_i)-\E\cos(tY)}-\frac{\mathrm{i}}{n}\sum_{i=1}^n\cbr{\sin(tY_i)-\E\sin(tY)}}^2}\\
    &=\E\sbr{\cbr{\frac{1}{n}\sum_{i=1}^n\cos(tY_i)-\E\cos(tY)}^2-\cbr{\frac{1}{n}\sum_{i=1}^n\sin(tY_i)-\E\sin(tY)}^2}\\
    &\le\Var\rbr{n^{-1}\sum_{i=1}^n\cos(tY_i)}+\Var\rbr{n^{-1}\sum_{i=1}^n\sin(tY_i)}\\
    &\le n^{-1}{\E\sbr{\cos(tY_1)^2+\sin(tY_1)^2}}=n^{-1}.
\end{align}
Similarly, we obtain
\begin{align}
    &\E\sbr{\abs{\frac{1}{n}\sum_{i=1}^nZ_i\exp(\mathrm{i}tY_i)-\E[Z\exp(\mathrm{i}tY)]}^2}\\
    &=\frac{1}{n}\Var(Z_1\cos(tY_1))+\frac{1}{n}\Var(Z_1\sin(tY_1))\\
    &\le n^{-1}\E[Z^2].
\end{align}
This completes the proof.
\end{proof}
\begin{lemma}
\label{lem:supKn}  Under the setting of Lemma \ref{lem:fan1993}, for bandwidth $h_n=c(\log n)^{-1/\beta}$ with $c>M_0(2/\gamma)^{1/\beta}$, we have
\begin{align}
    n^{-1}\sup_{x}\abs{K_n(x)}^2=o(1).
\end{align}
\end{lemma}
\begin{proof}[Proof of Lemma \ref{lem:supKn}] 
At first, (N1) implies that there exists a constant $M$ such that
\begin{align}
    \abs{\phi_\varepsilon(t)}>\frac{d_0}{2}|t|^{\beta_0}\exp(-|t|^\beta/\gamma),
\end{align}
for $|t|>M$.
By the fact that $\abs{\exp(-\mathrm{i}tx)}=1$ and that the support of $\phi_K(\cdot)$ is bounded by $M_0$, we have
\begin{align}
    \sup_{x}\abs{K_n(x)}
    &\le\int_{-\infty}^\infty\frac{\abs{\phi_K(t)}}{\abs{\phi_\varepsilon(t/h_n)}}dt\\
    &\le2\int_0^{Mh_n}\frac{\abs{\phi_K(t)}}{\abs{\phi_\varepsilon(t/h_n)}}dt
    +\frac{4}{d_0}\int_{Mh_n}^{M_0}|\phi_K(t)|\abs{\frac{t}{h_n}}^{-\beta_0}\exp\rbr{\frac{|t/h_n|^\beta}{\gamma}}dt\\
    &\le2h_n\int_0^M\frac{1}{|\phi_\varepsilon(u)|}du
    +\frac{4}{d_0}(M_0-Mh_n)h_n^{\beta_0}M_0^{-\beta_0}\exp\rbr{\frac{|M_0/h_n|^\beta}{\gamma}}\\
    &=O(h_n)+O(h_n^{\beta_0}\exp(|M_0/h_n|^\beta/\gamma)).
\end{align}
Here, we use the fact that $|\phi_K(t)|\le{\int|e^{-itx}||K(x)|dx}<\infty$.
Since we choose $h_n=c(\log n)^{-1/\beta}$ with $c>M_0(2/\gamma)^{1/\beta}$, we obtain the conclusion.
\end{proof}

\begin{proof}[Proof of Lemma \ref{lem:deriv-unif}] Let $a\le x\le b$. To begin with, by the triangle inequality, we have
\begin{align}
    &\sup_{a \leq x \leq b}\abs{\hat{m}'(x)-m'(x)}\\
    &=\sup_{a \leq x \leq b}\abs{\frac{\hat{r}'(x)\hat{f}(x)-\hat{r}(x)\hat{f}'(x)}{\hat{f}(x)^2}-\frac{{r}'(x){f}(x)-{r}(x){f}'(x)}{{f}(x)^2}}\\
    &\le\sup_{a \leq x \leq b}\abs{\frac{\hat{r}'(x)\hat{f}(x)-\hat{r}(x)\hat{f}'(x)-{r}'(x){f}(x)+{r}(x){f}'(x)}{f(x)^2}}\\
    &\quad+\sup_{a \leq x \leq b}\abs{\frac{\hat{r}'(x)\hat{f}(x)-\hat{r}(x)\hat{f}'(x)}{f(x)^2}\rbr{\frac{f(x)^2}{\hat{f}(x)^2}-1}}\\
    &\le B^2\sup_{a \leq x \leq b}\abs{\hat{r}'(x)\hat{f}(x)-\hat{r}(x)\hat{f}'(x)-{r}'(x){f}(x)+{r}(x){f}'(x)}\\
    &\quad+B^2\sup_{a \leq x \leq b}\abs{\hat{r}'(x)\hat{f}(x)-\hat{r}(x)\hat{f}'(x)}\abs{\frac{f(x)^2-\hat{f}(x)^2}{\hat{f}(x)^2}},\label{eq:mprime}
\end{align}
where the last inequality uses the assumption $\min_{a \leq x \leq b}|f(x)|\ge B^{-1}$.
We consider showing the convergence in probability by showing the $L^1$ convergence. Using the triangle inequality and the Cauchy-Schwarz inequality, we have
\begin{align}
    &\E\sup_{a \leq x \leq b}\abs{\hat{r}'(x)\hat{f}(x)-\hat{r}(x)\hat{f}'(x)-{r}'(x){f}(x)+{r}(x){f}'(x)}\\
    &\le\E\sup_{a \leq x \leq b}\abs{\hat{f}(x)\rbr{\hat{r}'(x)-r'(x)}}
    +\E\sup_{a \leq x \leq b}\abs{r'(x)\rbr{\hat{f}(x)-f(x)}}\\
    &\quad+\E\sup_{a \leq x \leq b}\abs{f'(x)\rbr{r(x)-\hat{r}(x)}}
    +\E\sup_{a \leq x \leq b}\abs{\hat{r}(x)\rbr{f'(x)-\hat{f}'(x)}}\\
    &\le\sqrt{\E\sbr{\sup_{a \leq x \leq b}\abs{\hat{f}(x)}^2}}\sqrt{\E\sbr{\sup_{a \leq x \leq b}\abs{\hat{r}'(x)-r'(x)}^2}}\\
    &\quad +\sup_{a \leq x \leq b}|r'(x)|\sqrt{\E\sbr{\sup_{a \leq x \leq b}\abs{\hat{f}(x)-f(x)}^2}}\\
    &\quad+\sup_{a \leq x \leq b}|f'(x)|\sqrt{\E\sbr{\sup_{a \leq x \leq b}\abs{\hat{r}(x)-r(x)}^2}}\\
    &\quad +\sqrt{\E\sbr{\sup_{a \leq x \leq b}\abs{\hat{r}(x)}^2}}\sqrt{\E\sbr{\sup_{a \leq x \leq b}\abs{\hat{f}'(x)-f'(x)}^2}}.
\end{align}
Thus, to bound the right-hand side of \eqref{eq:mprime}, we need to show that $\E[{\sup_x|{\hat{f}(x)}|^2}]$ and $\E[{\sup_x\abs{\hat{r}(x)}^2}]$ are bounded by constants and that $\E[{\sup_x|{\hat{f}(x)-f(x)}|^2}]$, $ \E[{\sup_x\abs{\hat{r}(x)-r(x)}^2}]$, $\E[{\sup_x|{\hat{f}'(x)-f'(x)}|^2}],$ and $\E[{\sup_x\abs{\hat{r}'(x)-r'(x)}^2}]$ converge to zero.

\noindent $\bullet$ \textbf{Bound for $\E\sbr{\sup_{a \leq x \leq b}|{\hat{f}(x)-f(x)}|^2}$}. 
By triangle inequality and the fact that $(x+y)^2\le 2x^2+2y^2$ for $x,y\in\R$, we have
\begin{align}
    &\E\sbr{\sup_{a \leq x \leq b}|{\hat{f}(x)-f(x)}|^2} \\
    &\le 2\E\sbr{\sup_{a \leq x \leq b}|{\hat{f}(x)-\E\hat{f}(x)}|^2}+2\sup_{a \leq x \leq b}|{\E\hat{f}(x)-f(x)}|^2.\label{eq:fhat-diff}
\end{align}
For the first term of the left-hand side of \eqref{eq:fhat-diff}, the Cauchy-Schwarz inequality gives
\begin{align}
\label{eq:kde-consis}
    &\E\sbr{\sup_{a \leq x \leq b}|{\hat{f}(x)-\E\hat{f}(x)}|^2}\\
    &\le\frac{1}{(2\pi)^2}\E\sbr{\cbr{\int_{-\infty}^\infty\frac{|\phi_K(th_n)|}{|\phi_\varepsilon(t)|}\abs{\hat{\phi}_Y(t)-\phi_Y(t)}dt}^2}\\
    &\le\frac{1}{(2\pi)^2}\cbr{\int_{-\infty}^\infty\frac{|\phi_K(th_n)|}{|\phi_\varepsilon(t)|}dt}\cbr{\int_{-\infty}^\infty\E\sbr{\abs{\hat{\phi}_Y(t)-\phi_Y(t)}^2}\frac{|\phi_K(th_n)|}{|\phi_\varepsilon(t)|}dt}.
\end{align}
Lemma \ref{lem:phiy} and the proof of Lemma \ref{lem:supKn} imply that this converges to zero as $n\to\infty$. Next, we consider the second term in \eqref{eq:fhat-diff}. We obtain
\begin{align}
    \E\sbr{\hat{f}(x)}
    &=\frac{1}{2\pi}\int_{-\infty}^\infty\exp(-\mathrm{i}tx)\phi_K(th_n)\frac{\E_Y[\exp(\mathrm{i}tY)]}{\phi_\varepsilon(t)}dt\\
    &=\frac{1}{2\pi}\E_X\sbr{\int_{-\infty}^\infty\exp(\mathrm{i}tx)\phi_K(th_n)\exp(-\mathrm{i}tX)dt}\\
    &=\E_X\sbr{\frac{1}{2\pi h_n}\int_{-\infty}^\infty\exp\rbr{\mathrm{i}t\frac{x-X}{h}}\phi_K(t)dt}\\
    &=\frac{1}{h_n}\E_X\sbr{K\rbr{\frac{x-X}{h_n}}}.
\end{align}
Thus, a classical result for the kernel density estimation gives $\sup_x|\E[\hat{f}(x)]-f(x)|\to0$ as $n\to0$.

\noindent $\bullet$ \textbf{Bound for $\E\sbr{\sup_{a \leq x \leq b}|{\hat{r}(x)-r(x)}|^2}$}. 
By triangle inequality and the fact that $(x+y)^2\le 2x^2+2y^2$,
\begin{align}
    &\E\sbr{\sup_{a \leq x \leq b}|{\hat{r}(x)-r(x)}|^2}\\
    &\le2\E\sbr{\sup_{a \leq x \leq b}|{\hat{r}(x)-\E\hat{r}(x)}|^2}+2\sup_{a \leq x \leq b}|{\E\hat{r}(x)-r(x)}|^2.\label{eq:rhat-diff}
\end{align}
For the first term of the left-hand side of \eqref{eq:rhat-diff}, Cauchy-Schwarz inequality gives
\begin{align}
    &\E\sbr{\sup_{a \leq x \leq b}|{\hat{r}(x)-\E\hat{r}(x)}|^2}\\
    &\le\frac{1}{(2\pi)^2}\E\sbr{\cbr{\int_{-\infty}^\infty\frac{|\phi_K(th_n)|}{|\phi_\varepsilon(t)|}\abs{\frac{1}{n}\sum_{i=1}^nZ_i\exp(\mathrm{i}tY_j)-\E[Z\exp(\mathrm{i}tY)]}dt}^2}\\
    &\le\frac{1}{(2\pi)^2}\cbr{\int_{-\infty}^\infty\frac{|\phi_K(th_n)|}{|\phi_\varepsilon(t)|}dt}\cbr{\frac{1}{n}\int_{-\infty}^\infty\frac{|\phi_K(th_n)|}{|\phi_\varepsilon(t)|}dt},
\end{align}
where we use the proof of Lemma \ref{lem:supKn} for the last inequality.
Lemma \ref{lem:phiy} implies that this term converges to zero as $n\to\infty$. Next, we consider the second term in \eqref{eq:rhat-diff}. We have
\begin{align}
    \E\sbr{\hat{r}(x)}
    =\frac{1}{h_n}\E_{X,Z}\sbr{ZK\rbr{\frac{x-X}{h_n}}}.
\end{align}
Thus we have $\sup_{a \leq x \leq b}|\E[\hat{r}(x)]-r(x)|\to0$.

\noindent $\bullet$ \textbf{Bound for $\E\sbr{\sup_{a \leq x \leq b}|{\hat{f}'(x)-f'(x)}|^2}$}. By triangle inequality and the fact that $(x+y)^2\le 2x^2+2y^2$ for $x,y\in\R$, we have
\begin{align}
    &\E\sbr{\sup_{a \leq x \leq b}|{\hat{f}'(x)-f'(x)}|^2}\\
    &\le2\E\sbr{\sup_{a \leq x \leq b}|{\hat{f}'(x)-\E\hat{f}'(x)}|^2}+2\sup_{a \leq x \leq b}|{\E\hat{f}'(x)-f'(x)}|^2.\label{eq:fphat-diff}
\end{align}
For the first term of the left-hand side of \eqref{eq:fphat-diff}, since $\partial\exp(-\mathrm{i}tx)/(\partial x)=-\mathrm{i}t\exp(-\mathrm{i}tx)$ and $|\mathrm{i}|=|\exp(-\mathrm{i}tx)|=1$,
\begin{align}
    &\E\sbr{\sup_{a \leq x \leq b}|{\hat{f}'(x)-\E\hat{f}'(x)}|^2}\\
    &=\E\sbr{\sup_{a \leq x \leq b}\abs{\frac{1}{2\pi}\int_{-\infty}^\infty-\mathrm{i}t\exp(-\mathrm{i}tx)\frac{\phi_K(th_n)}{\phi_\varepsilon(t)}\cbr{\hat{\phi}_Y(t)-\phi_Y(t)}dt}^2}\\
    &\le\frac{1}{(2\pi)^2}\E\sbr{\cbr{\int_{-\infty}^\infty\frac{|t\phi_K(th_n)|}{|\phi_\varepsilon(t)|}\abs{\hat{\phi}_Y(t)-\phi_Y(t)}dt}^2}.
\end{align}
Thus, this converges to zero in the same way as \eqref{eq:fhat-diff}. For the second term in \eqref{eq:fphat-diff}, by the integration by parts,
\begin{align}
    \E\sbr{\hat{f}'(x)}
    &=\frac{1}{h_n^2}\int_{-\infty}^\infty K'\rbr{\frac{x-y}{h_n}}f(y)dy\\
    &=\frac{1}{h_n}\int_{-\infty}^\infty K\rbr{\frac{x-y}{h_n}}f'(y)dy
    -\frac{1}{h_n}\sbr{K\rbr{\frac{x-y}{h_n}}f'(y)}_{-\infty}^\infty.
\end{align}
Here, the second term is zero and the first term converges to $f'(x)$ uniformly.

\noindent $\bullet$ \textbf{Bound for $\E\sbr{\sup_{a \leq x \leq b}|{\hat{r}'(x)-r'(x)}|^2}$}. By triangle inequality and the fact that $(x+y)^2\le 2x^2+2y^2$ for $x,y\in\R$, we have
\begin{align}
    \E\sbr{\sup_{a \leq x \leq b}|{\hat{r}'(x)-r'(x)}|^2}
    \le2\E\sbr{\sup_{a \leq x \leq b}|{\hat{r}'(x)-\E\hat{r}'(x)}|^2}+2\sup_{a \leq x \leq b}|{\E\hat{r}'(x)-r'(x)}|^2.\label{eq:rphat-diff}
\end{align}
For the first term of the left-hand side of \eqref{eq:rphat-diff}, since $|\mathrm{i}|=|\exp(-\mathrm{i}tx)|=1$, we have
\begin{align}
    &\E\sbr{\sup_x|{\hat{f}'(x)-\E\hat{f}'(x)}|^2}\\
    &=\E\sbr{\sup_x\abs{\frac{1}{2\pi}\int_{-\infty}^\infty-\mathrm{i}t\exp(-\mathrm{i}tx)\frac{\phi_K(th_n)}{\phi_\varepsilon(t)}\cbr{\frac{1}{n}\sum_{i=1}^nZ_i\exp(-\mathrm{i}tY_i)-\E[Z\exp(-\mathrm{i}tZ)]}dt}^2}\\
    &\le\frac{1}{(2\pi)^2}\E\sbr{\cbr{\int_{-\infty}^\infty\frac{|t\phi_K(th_n)|}{|\phi_\varepsilon(t)|}\abs{\frac{1}{n}\sum_{i=1}^nZ_i\exp(-\mathrm{i}tY_i)-\E[Z\exp(-\mathrm{i}tZ)]}dt}^2}.
\end{align}
Thus, this converges to zero in the same way as \eqref{eq:rhat-diff}. For the second term in \eqref{eq:rphat-diff}, by the integration by parts,
\begin{align}
    \E\sbr{\hat{r}'(x)}&=\frac{1}{h_n^2}\int_{-\infty}^\infty \int_{-\infty}^\infty zK'\rbr{\frac{x-y}{h_n}}f(y,z)dydz\\
    &=\frac{1}{h_n}\int_{-\infty}^\infty\int_{-\infty}^\infty zK\rbr{\frac{x-y}{h_n}}\frac{\partial}{\partial y}f(y,z)dy dz\\
    &\quad -\frac{1}{h_n}\int_{-\infty}^\infty\sbr{zK\rbr{\frac{x-y}{h_n}}\frac{\partial}{\partial y}f(y,z)}_{-\infty}^\infty dz.
\end{align}
Here, the second term is zero and the first term converges to $r'(x)=(\partial/\partial x)\int zf(x,z)dz$ uniformly.

\noindent $\bullet$ \textbf{Bound for $\E\sbr{\sup_{a \leq x \leq b}|{\hat{f}(x)}|^2}$ and $\E\sbr{\sup_{a \leq x \leq b}|{\hat{r}(x)}|^2}$}. By triangle inequality and the fact that $(x+y)^2\le 2x^2+2y^2$,
\begin{align}
    \E\sbr{\sup_{a \leq x \leq b}|{\hat{f}(x)}|^2}
    &\le2\sup_{a \leq x \leq b}|f(x)|^2+2\E\sbr{\sup_{a \leq x \leq b}|\hat{f}(x)-f(x)|^2}.
\end{align}
We have already shown $\E[{\sup_{a \leq x \leq b}|\hat{f}(x)-f(x)|^2]=o(1)}$, $\E[{\sup_{a \leq x \leq b}|{\hat{f}(x)}|^2}]$ is asymptotically bounded by a constant.
Similarly, we can show that the other expectation $\E\sbr{\sup_{a \leq x \leq b}|{\hat{r}(x)}|^2}$ is asymptotically bounded by a constant. 
Combining these results together, we conclude that the first term of \eqref{eq:mprime} is $o(1)$.

Next, we consider the second term of \eqref{eq:mprime}. Since $\hat{f}(x)$ is asymptotically bounded uniformly on $[a,b]$ by the results above, we have only to show that $\sup_{a \leq x \leq b}|\hat{f}(x)^2-f(x)^2|=o(1)$. This holds since
\begin{align}
    \sup_{a \leq x \leq b}|\hat{f}(x)^2-f(x)^2|
    \le \sup_{a \leq x \leq b}|\hat{f}(x)+f(x)| \sup_{a \leq x \leq b}|\hat{f}(x)-f(x)|
    =o_\mathrm{p}(1).
\end{align}
This concludes that $\sup_{a \leq x \leq b}|\hat{m}'(x)-m'(x)|\pconv0$ as $n\to\infty$.
\end{proof}

\section{Proofs of the Results}
For a convex function $f:\R\to\R$ and a constant $\gamma>0$, define the proximal operator $\prox_{\gamma f}:\R\to\R$ as
\begin{align}
    \prox_{\gamma f}(x)=\argmin_{z\in\R}\cbr{\gamma f(z)+\frac{1}{2}(x-z)^2}.
\end{align}
\subsection{Proof of Master Theorem}
\label{sec:proof-master}
First, we define the notation used in the proof. We consider an invertible matrix $\bL\in\R^{p\times p}$ that satisfies $\bSigma=\bL\bL^\top$. Define, for each $i\in\{1,\ldots,n\}$,
\begin{align}
\label{eq:eta}
    \tilde\bX_i=\bL^{-1}\bX_i,\quad
    \bmeta = \bL^\top\bbeta,\quad
    \hat\bmeta=\bL^\top\bar\bbeta.
\end{align}
\begin{proof}[Proof of Theorem \ref{thm:master-zsc}]
We consider the following three steps.

\textbf{Step 1: Reduction to standard Gaussian features}.
Note that the single-index model $y_i=g(\bX_i^\top\bbeta)+\varepsilon_i$ is equivalent to $y_i=g(\tilde\bX_i^\top\bmeta)+\varepsilon_i$.
Since $\bX_i^\top\bb=\tilde\bX_i^\top(\bL^\top\bb)$ holds, we have $\bar\ell(\bX_i^\top\bb;y_i)=\bar\ell(\tilde\bX_i^\top\bL^\top \bb;y_i)$. Hence, $\hat\bmeta\in\argmin_{\tilde\bb\in\R^p}\sum_{i=1}^n\bar\ell(\tilde\bX_i^\top\tilde\bb;y_i)$ is the estimator corresponding to the true parameter $\bmeta\in\R^p$ and features $\tilde\bX_i\sim\mathcal{N}_p(\bzero,\bI_p)$.

We can choose $\bSigma=\bL\bL^\top$ to be a Cholesky factorization so that $\eta_p=\tau_p\beta_p$ and $\hat\eta_p=\tau_p\bar\beta_p$ with $\tau_p=(\bSigma^{-1})_{pp}^{-1/2}$ by \eqref{eq:eta}. This follows from the fact that $L_{pp}=\tau_p$ since $\tau_p^2=\mathrm{Var}(X_{ip}\mid\bX_{i\setminus p})=\mathrm{Var}(X_{ip}\mid\tilde\bX_{i\setminus p})$, where $\bX_{i\setminus p}\in\R^{p-1}$ denotes the vector $\bX_i$ without $p$th coordinate. Since we can generalize this to any coordinate by permutation, we obtain
\begin{align}
    \tau_j\frac{\bar\beta_j-\mu_{\bar\bbeta}\beta_j}{\sigma_{\bar\bbeta}}=\frac{\hat\eta_j-\mu_{\bar\bbeta}\eta_j}{\sigma_{\bar\bbeta}},
\end{align}
for each $j\in\{1,\ldots,p\}$ and any pair $(\mu_{\bar\bbeta},\sigma_{\bar\bbeta})$.

\textbf{Step 2: Reduction to uniform distribution on sphere}.
Define an orthogonal projection matrix $\bP_{\bmeta}=\bmeta\bmeta^\top/\norm{\bmeta}^2$ onto $\bmeta$, and an orthogonal projection matrix $\bP_{\bmeta}^\perp=\bI_p-\bP_{\bmeta}$ onto the orthogonal complement of $\bmeta$. Let $\bU\in\R^{p\times p}$ be any orthogonal matrix obeying $\bU\bmeta=\bmeta$, namely, any rotation operator about $\bmeta$. Then, since $\hat\bmeta=\bP_{\bmeta}\hat\bmeta+\bP_{\bmeta}^\perp\hat\bmeta$, we have
\begin{align}
    \bU\hat\bmeta
    =\bU\bP_{\bmeta}\hat\bmeta+\bU\bP_{\bmeta}^\perp\hat\bmeta
    =\bP_{\bmeta}\hat\bmeta+\bU\bP_{\bmeta}^\perp\hat\bmeta.
\end{align}
Using this, we obtain
\begin{align}
\label{eq:rotation}
    \frac{\bU\bP_{\bmeta}^\perp\hat\bmeta}{\|\bP_{\bmeta}^\perp\hat\bmeta\|}
    \overset{\rm d}{=}\frac{\bP_{\bmeta}^\perp\hat\bmeta}{\|\bP_{\bmeta}^\perp\hat\bmeta\|}
    =\frac{\hat\bmeta-\mu_{\bar\bbeta}\bmeta}{\sigma_{\bar\bbeta}},
\end{align}
where the first identity follows from the fact that $\bU\hat\bmeta\overset{\rm d}{=}\hat\bmeta$ since $\bU\hat\bmeta$ is the estimator with a true coefficient $\bU\bmeta=\bmeta$ and features drawn iid from $\mathcal{N}_p(\bzero,\bI_p)$, by $\bar\ell(\tilde\bX_i^\top\tilde\bb;y_i)=\bar\ell((\bU^\top\tilde\bX_i)^\top\bU\tilde\bb;y_i)$ and $\bU^\top\tilde\bX_i\overset{\rm d}{=}\tilde\bX_i$.
\eqref{eq:rotation} reveals that $(\hat\bmeta-\mu_{\bar\bbeta}\bmeta)/\sigma_{\bar\bbeta}$ is rotationally invariant about $\bmeta$, lies in $\bmeta^\perp$, and has a unit norm. This means $(\hat\bmeta-\mu_{\bar\bbeta}\bmeta)/\sigma_{\bar\bbeta}$ is uniformly distributed on the unit sphere lying in $\bmeta^\perp$.

\textbf{Step 3: Deriving asymptotic normality}.
The result of the previous step gives us
\begin{align}
\label{eq:deq-PZ}
    \frac{\hat\bmeta-\mu_{\bar\bbeta}\bmeta}{\sigma_{\bar\bbeta}}
    \overset{\rm d}{=}\frac{\bP_{\bmeta}^\perp\bZ}{\|\bP_{\bmeta}^\perp\bZ\|},
\end{align}
where $\bZ\sim\mathcal{N}_p(\bzero,\bI_p)$. Triangle inequalities yield that
\begin{align}
    \frac{\norm{\bZ}}{\sqrt{p}}-\frac{|\bmeta^\top\bZ|}{\sqrt{p}\norm{\bmeta}}
    \le\frac{\|\bP_{\bmeta}^\perp\bZ\|}{\sqrt{p}}
    \le\frac{\norm{\bZ}}{\sqrt{p}}+\frac{|\bmeta^\top\bZ|}{\sqrt{p}\norm{\bmeta}}.
\end{align}
Since $|\bmeta^\top\bZ|/(\sqrt{p}\norm{\bmeta})\as0$ and $\norm{\bZ}/\sqrt{p}\as1$, we obtain $\|\bP_{\bmeta}^\perp\bZ\|/\sqrt{p}\as1$. Therefore, this fact and \eqref{eq:deq-PZ} imply that
\begin{align}
    \sqrt{p}\frac{\hat\eta_j-\mu_{\bar\bbeta}\eta_j}{\sigma_{\bar\bbeta}}
    \overset{\rm d}{=}\check\sigma_j Q+o_\mathrm{p}(1),\quad\check\sigma_j^2=1-\frac{\eta_j^2}{\norm{\bmeta}^2},
\end{align}
where $Q\sim\mathcal{N}(0,1)$. 
Here we use the fact that the covariance matrix of $\bP_{\bmeta}^\perp \bZ$ is $\bP_{\bmeta}^\perp \bP_{\bmeta}^\perp=\bI_p-\bmeta\bmeta^\top/\|\bmeta\|^2$. The assumptions $\eta_j=o(1)$ and $\|\bmeta\|=1$ complete the proof.
\end{proof}

\subsection{Proof of Theorem \ref{thm:link-est}}
Let $\varsigma^2$ be the ratio $\sigma_1^2/\mu_1^2$, where $\mu_1^2$ and $\sigma_1^2$ are the true inferential parameters of the pilot estimator $\tilde\bbeta$.
\begin{proof}[Proof of Proposition \ref{prop:initial}]
First, define $\gamma_1=\trace(\bSigma({\bX^{(1)}}^\top {\bX^{(1)}}+n\lambda\bI_p)^{-1})$.
We also define
\begin{align}
    \mu_1=\bbeta^\top\bSigma\tilde{\bbeta}, \quad\mbox{and}\quad
    \sigma_1^2=\tilde{\bbeta}^\top\bSigma\tilde{\bbeta}-\mu_1^2.\label{eq:mu1}
\end{align}
Let $\tilde{r}^2$ be the mean squared error $n_1^{-1}\|{\by^{(1)}-\bX^{(1)}\Tilde{\bbeta}\|}^2$. Since $\Tilde{\bbeta}$ is a ridge estimator, Theorem 4.3 in \citet{bellec2022observable} implies, 
\begin{align}
\label{eq:belThm43}
    \max_{1\le i\le n_1}\E\sbr{\tilde{r}^{-2}\abs{\Tilde{\bbeta}^\top{\bX_i^{(1)}}-\prox_{\gamma_1f}\rbr{\mu_1\bbeta^\top{\bX_i^{(1)}}+\sigma_1Z_i}}^2}\le\frac{C}{n_1},
\end{align}
with $f(t)=t^2/2$ and $Z_i\sim\mathcal{N}(0,1)$ independent of $\bbeta^\top\bX^{(1)}$. 
Since ridge regression satisfies $\|\tilde\bbeta\|\le C_\lambda$ with a constant $C_\lambda>0$ depending on the regularization parameter $\lambda>0$ by the KKT condition, we have $\tilde{r}^2=O_\mathrm{p}(1)$. Hence,
\begin{align}
    \abs{\Tilde{\bbeta}^\top{\bX_i^{(1)}}-\prox_{\gamma_1f}\rbr{\mu_1\bbeta^\top{\bX_i^{(1)}}+\sigma_1Z_i}}\pconv0,
\end{align}
as $n_1\to\infty$ for each $i\in[n_1]$. By using the fact that $\prox_{\gamma f}(a)=a-f'(\prox_{\gamma f}(a))$ for $a\in\R$, $\gamma>0$, and $f:\R\to\R$ by the definition of the proximal operator, we obtain
\begin{align}
\label{eq:index-pconv}
    \abs{\Tilde{\bbeta}^\top{\bX_i^{(1)}}+\gamma_1\rbr{{y_i^{(1)}}-\Tilde{\bbeta}^\top{\bX_i^{(1)}}}-\mu_1\bbeta^\top{\bX_i^{(1)}}-\sigma_1Z_i}\pconv0,
\end{align}
as $n_1\to\infty$. 
Next, we consider to replace $(\mu_1,\sigma_1,\gamma_1)$ with observable adjustments $(\tilde\mu,\tilde\sigma,\tilde\gamma)$. Theorem 4.4 in \citet{bellec2022observable} gives their consistency:
\begin{align}
    &\E\sbr{\tilde{v}\abs{\tilde\gamma-\gamma_1}}\le C_1 n^{-1/2},\\
    &\E\sbr{\Tilde{v}^2\Tilde{t}^2\Tilde{r}^{-2}\rbr{\abs{\tilde\mu^2-\mu_1^2}+\abs{\tilde\sigma^2-\sigma_1^2}}}\le C_2 n^{-1/2},
\end{align}
where we define $\Tilde{t}^2=\|(\lambda\bSigma^{-1/2}+\Tilde{v}\bSigma^{1/2})\tilde\bbeta\|^2-\kappa_1\Tilde{r}^2$ and $C_1, C_2$ are positive constants.
Proposition 3.1 in \citet{bellec2022observable} implies  that $\Tilde{v}\ge1/(1+\Bar{c})-4\Bar{c}/n_1$ for a constant $\Bar{c}>0$. Also, Theorem 4.4 in \citet{bellec2022observable} implies that $\Tilde{t}^2\pconv(\bbeta^\top(\Tilde{v}\bSigma+\lambda)\tilde\bbeta)^2$. By using these results, we have
$\tilde\gamma\pconv\gamma_1$, $\tilde\mu\pconv\mu_1$, and $\tilde\sigma^2\pconv\sigma_1^2$ as $n_1\to\infty$ since the sign of $\mu_1$ is specified by an assumption. 
Then, triangle inequality implies
\begin{align}
    &\abs{\Tilde{\bbeta}^\top{\bX_i^{(1)}}+\tilde\gamma\rbr{y_i^{(1)}-\Tilde{\bbeta}^\top{\bX_i^{(1)}}}-\tilde\mu\bbeta^\top{\bX_i^{(1)}}-\tilde\sigma Z_i}\\
    &\le\abs{\Tilde{\bbeta}^\top{\bX_i^{(1)}}+\gamma_1\rbr{y_i^{(1)}-\Tilde{\bbeta}^\top{\bX_i^{(1)}}}-\mu_1\bbeta^\top{\bX_i^{(1)}}-\sigma_1Z_i}\\
    &+\abs{(\gamma_1-\tilde\gamma)\rbr{y_i^{(1)}-\tilde\bbeta^\top{\bX_i^{(1)}}}}
    +\abs{(\mu_1-\tilde\mu)\tilde\bbeta^\top{\bX_i^{(1)}}}
    +\abs{(\sigma_1-\tilde\sigma)Z_i},
\end{align}
which converges in probability to zero.
\end{proof}

To prove Theorem \ref{thm:link-est}, we first introduce an approximation $\Tilde{g}(\cdot)$ for the link estimator $\hat{g}(\cdot)$ defined in \eqref{eq:deconv}.

\begin{lemma}
\label{lem:g-approx}
    For $i=1,\ldots,n_1$, define $\tilde{W_i}=\bbeta^\top{\bX_i^{(1)}}+\varsigma Z_i$ with $\varsigma=\sigma_1/|\mu_1|$ and 
    {\rc
    $Z_i$ independent of $\bbeta^\top\bX_i^{(1)}$ satisfying $(Z_1,\ldots,Z_n)\sim\mathcal{N}_n(\bzero,\bI_n)$.
    }
    We also define
    \begin{align}
    \label{eq:lip-deconv}
        \tilde{g}(x):=\frac{\displaystyle\sum_{i=1}^{n_1}y_i^{(1)}\int_{-\infty}^\infty \exp\rbr{t^2\varsigma^2/(2h_n^2)-\mathrm{i}t(x-\tilde{W_i}  )/h_n}\phi_K(t)dt}{\displaystyle\sum_{i=1}^{n_1}\int_{-\infty}^\infty \exp\rbr{t^2\varsigma^2/(2h_n^2)-\mathrm{i}t(x-\tilde{W_i}  )/h_n}\phi_K(t)dt}.
    \end{align}
    Then, under the setting of Theorem \ref{thm:link-est}, we have, as $n_1\to\infty$,
    \begin{align}
        \sup_{a\le x\le b}\abs{\breve{g}(x)-\tilde{g}(x)}=O_\mathrm{p}\rbr{\frac{1}{(\log n_1)^{m/2}}}.
    \end{align}
\end{lemma}
\begin{proof}[Proof of Lemma \ref{lem:g-approx}]
We rewrite the kernel function for deconvolution in \eqref{eq:deconv} as
\begin{align}
    K_n(x)=\frac{1}{2\pi}\int_{-\infty}^\infty\exp(-\mathrm{i}tx)\frac{\phi_K(t)}{\exp(-t^2\hat{\varsigma}^2/(2h_n^2))}dt,
\end{align}
and also introduce an approximated version of the kernel function as
\begin{align}
    \tilde{K}_n(x)=\frac{1}{2\pi}\int_{-\infty}^\infty\exp(-\mathrm{i}tx)\frac{\phi_K(t)}{\exp(-t^2{\varsigma}^2/(2h_n^2))}dt.
\end{align}
The difference here is that the parameter $\hat{\varsigma}$ is replaced by $\varsigma$.
We also define $\phi_\varsigma(t)=\exp(-t^2\varsigma^2/2)$.

At first, we have
\begin{align}
    &\abs{\breve{g}(x)-\tilde{g}(x)}\\
    &=\abs{\frac{\sum_{i=1}^{n_1}y_i^{(1)}K_n\rbr{\frac{W_i-x}{h_n}}}{\sum_{i=1}^{n_1}K_n\rbr{\frac{W_i-x}{h_n}}}
    -\frac{\sum_{i=1}^{n_1}y_i^{(1)}\tilde{K}_n\rbr{\frac{\tilde{W}_i-x}{h_n}}}{\sum_{i=1}^{n_1}\tilde{K}_n\rbr{\frac{\tilde{W}_i-x}{h_n}}}}\\
    &\le C_{1,n}C_{2,n}
    \abs{\frac{1}{n_1h_n}\sum_{i=1}^{n_1}y_i^{(1)}\cbr{K_n\rbr{\frac{W_i-x}{h_n}}-\tilde{K}_n\rbr{\frac{\tilde{W}_i-x}{h_n}}}}\\
    &\quad+C_{1,n}C_{3,n}\frac{1}{n_1h_n}\abs{\sum_{i=1}^{n_1}\tilde{K}_n\rbr{\frac{\tilde{W}_i-x}{h_n}}-{K}_n\rbr{\frac{{W}_i-x}{h_n}}}\\
    &\le C_{1,n}C_{2,n}
    \abs{\frac{1}{n_1h_n}\sum_{i=1}^{n_1}y_i^{(1)}\cbr{\tilde{K}_n\rbr{\frac{\tilde{W}_i-x}{h_n}}-\tilde{K}_n\rbr{\frac{W_i-x}{h_n}}}}\label{eq:lemdcnv1}\\
    &\quad+C_{1,n}C_{2,n}
    \abs{\frac{1}{n_1h_n}\sum_{i=1}^{n_1}y_i^{(1)}\cbr{\tilde{K}_n\rbr{\frac{W_i-x}{h_n}}-{K}_n\rbr{\frac{{W}_i-x}{h_n}}}}\label{eq:lemdcnv2}\\
    &\quad+C_{1,n}C_{3,n}\frac{1}{n_1h_n}\abs{\sum_{i=1}^{n_1}\tilde{K}_n\rbr{\frac{\tilde{W}_i-x}{h_n}}-\tilde{K}_n\rbr{\frac{{W}_i-x}{h_n}}}\label{eq:lemdcnv3}\\
    &\quad+C_{1,n}C_{3,n}\frac{1}{n_1h_n}\abs{\sum_{i=1}^{n_1}\tilde{K}_n\rbr{\frac{{W}_i-x}{h_n}}-{K}_n\rbr{\frac{{W}_i-x}{h_n}}},\label{eq:lemdcnv4}
\end{align}
where we define $C_{1,n}=\abs{n^2h_n^2(\sum_{i=1}^{n_1}\tilde{K}_n\rbr{\frac{\tilde{W}_i-x}{h_n}})^{-1}(\sum_{i=1}^{n_1}K_n\rbr{\frac{W_i-x}{h_n}})^{-1}}$, 
$C_{2,n}=\abs{\frac{1}{n_1h_n}\sum_{i=1}^{n_1}\tilde{K}_n\rbr{\frac{\tilde{W}_i-x}{h_n}}}$, and $C_{3,n}=\abs{\frac{1}{n_1h_n}\sum_{i=1}^{n_1}y_i^{(1)}\tilde{K}_n\rbr{\frac{\tilde{W}_i-x}{h_n}}}$.
{\rc Here, $C_{2,n}$ and $C_{3,n}$ converge to upper-bounded non-negative functions of $x$ from the consistency of the deconvoluted kernel density estimator by the i.i.d. assumption of $Z_1,\ldots,Z_n$.}
We proceed to bound {\rc $C_{1,n}$} and each term on the right-hand side.
First, we bound \eqref{eq:lemdcnv4}. $(|t|e^{-1}/\sqrt{2})$-Lipschitz continuity of $\phi_\varsigma(t)$ with respect to $\varsigma$ yields
\begin{align}
    &\frac{1}{n_1h_n}\abs{\sum_{i=1}^{n_1}\tilde{K}_n\rbr{\frac{{W}_i-x}{h_n}}-{K}_n\rbr{\frac{{W}_i-x}{h_n}}}\\
    &=\abs{\frac{1}{2\pi n_1h_n}\sum_{i=1}^{n_1}\int_{-M_0}^{M_0}\exp\rbr{-\mathrm{i}t\frac{W_i-x}{h_n}}\frac{\phi_K(t)}{\phi_\varsigma(t/h_n)\phi_{\tilde\varsigma}(t/h_n)}\cbr{\phi_\varsigma(t/h_n)-\phi_{\tilde\varsigma}(t/h_n)}dt}\\
    &\le\frac{1}{\sqrt{2}e\pi h_n^2}\abs{\varsigma-\tilde\varsigma}\int_{0}^{M_0}|t\phi_K(t)|\exp\rbr{\frac{t^2(\varsigma^2+\tilde\varsigma^2)}{2h_n^2}}dt.
\end{align}
Theorem 4.4 in \citet{bellec2022observable} implies that $|\varsigma-\tilde\varsigma|=O_\mathrm{p}(n_1^{-1/2})$ since
\begin{align}
\label{eq:tau-conv}
    \abs{\varsigma-\tilde\varsigma}(\varsigma+\tilde\varsigma)
    =\abs{\varsigma^2-\tilde\varsigma^2}
    \le\frac{1}{\tilde\mu^2\mu_1^2}\cbr{\mu_1^2\abs{\tilde\sigma^2-\sigma_1^2}+\sigma_1^2\abs{\mu_1^2-\tilde\mu^2}}.
\end{align}
Hence, as we choose $h_n=(c_h\log n_1)^{-1/2}$ such that $M_0^2(\varsigma^2+\tilde\varsigma^2)c_h/2+c\le 1/2$ for some $c>0$, we obtain
\begin{align}
    \frac{1}{n_1h_n}\abs{\sum_{i=1}^{n_1}\tilde{K}_n\rbr{\frac{{W}_i-x}{h_n}}-{K}_n\rbr{\frac{{W}_i-x}{h_n}}}=O_\mathrm{p}\rbr{(\log n_1)n_1^{-c}}.
\end{align}

Next, we bound \eqref{eq:lemdcnv3}. For any $x,x'\in\R$, we have
\begin{align}
    |e^{-\mathrm{i}tx}-e^{-\mathrm{i}tx'}|
    &=\rbr{\cbr{\cos(-tx)-\cos(-tx')}^2+\cbr{\sin(-tx)-\sin(-tx')}^2}^{1/2}\\
    &\le\sqrt{2}t|x-x'|,
\end{align}
where the last inequality follows from 1-Lipschitz continuity of $\cos(\cdot)$ and $\sin(\cdot)$. 
Since $\phi_K(\cdot)$ is supported on $[-M_0,M_0]$, we have, 
\begin{align}
    &\frac{1}{n_1h_n}\abs{\sum_{i=1}^{n_1}\tilde{K}_n\rbr{\frac{\tilde{W}_i-x}{h_n}}-\tilde{K}_n\rbr{\frac{{W}_i-x}{h_n}}}\\
    &=\abs{\frac{1}{2\pi n_1 h_n}\sum_{i=1}^{n_1}\int_{-M_0}^{M_0}\frac{\phi_K(t)}{\phi_\varsigma(t/h_n)}\cbr{\exp\rbr{-\mathrm{i}t\frac{(\Tilde{W}_i-x)}{h_n}}-\exp\rbr{-\mathrm{i}t\frac{(W_i-x)}{h_n}}}dt}\\
    &\le\frac{\sqrt{2}}{\pi n_1 h_n^2}\sum_{i=1}^{n_1}\abs{\Tilde{W}_i-W_i}\int_{0}^{M_0}|t\phi_K(t)|\exp\rbr{\frac{t^2\varsigma^2}{2h_n^2}}dt.
\end{align}
Here, we can use the fact that, by the triangle inequality,
\begin{align}
    |\tilde{W}_i-W_i|
    \le|\varsigma-\tilde\varsigma|Z_i+|W_i-\bbeta^\top{\bX_i^{(1)}}-\tilde\varsigma Z_i|=O_\mathrm{p}(n_1^{-1/2}),
\end{align}
where the equality follows from \eqref{eq:tau-conv} and \eqref{eq:belThm43}.
Thus, since we choose $h_n=(c_h\log n)^{-1/2}$ such that $M_0^2\varsigma^2c_h/2+c\le 1/2$ for some $c>0$, we have
\begin{align}
    &\frac{1}{n_1h_n}\abs{\sum_{i=1}^{n_1}\tilde{K}_n\rbr{\frac{\tilde{W}_i-x}{h_n}}-\tilde{K}_n\rbr{\frac{{W}_i-x}{h_n}}}\\
    &=O_\mathrm{p}\rbr{\frac{1}{{n_1^{1/2}}h_n^2}\exp\rbr{\frac{M_0^2\varsigma^2}{2}c_h\log n_1}}\\
    &=O_\mathrm{p}((\log n_1)n_1^{-c}).
\end{align}
This concludes the convergence of {\rc $C_{1,n}$ and} the term \eqref{eq:lemdcnv3}.
Repeating the arguments above for \eqref{eq:lemdcnv1}-\eqref{eq:lemdcnv2} implies that $\sup_{a\le x\le b}|\hat{g}(x)-g(x)|=O_p((\log n_1)^{-m/2})$.
\end{proof}

\begin{proof}[Proof of Theorem \ref{thm:link-est}]
    Since $\bbeta^\top{\bX_i^{(1)}}\sim\mathcal{N}(0,1)$ by Assumption \ref{asmp:feature}, Lemma \ref{lem:fan1993} implies that, for $\tilde{g}(\cdot)$ defined in \eqref{eq:lip-deconv},
    \begin{align}
        \sup_{a\le x\le b}\abs{\Tilde{g}(x)-g(x)}=O_\mathrm{p}\rbr{\frac{1}{(\log n_1)^{m/2}}}. \label{ineq:gtilde_g}
    \end{align}
    Thus, we obtain
    \begin{align}
        \sup_{a\le x\le b}\abs{\hat{g}(x)-g(x)}
        &\le \sup_{a\le x\le b}\abs{\breve{g}(x)-\Tilde{g}(x)}
        +\sup_{a\le x\le b}\abs{\Tilde{g}(x)-g(x)}\\
        &=O_\mathrm{p}\rbr{\frac{1}{(\log n_1)^{m/2}}}.
    \end{align}
    The last equality follows Lemma \ref{lem:g-approx} and \eqref{ineq:gtilde_g}.
    Also, the first inequality follows the triangle inequality and a property of each choice of the monotonization operator $\mR[\cdot]$.
    If we select the naive  $\mR_\mathrm{naive}[\cdot]$, 
    we obtain the following for $x \in [a,b]$:
    \begin{align}
        |\hat{g}(x) - g(x)| &= \left|\sup_{x' \in [a,x]} \breve{g}(x') - g(x) \right| \\
        &= \left|\sup_{x' \in [a,x]} \breve{g}(x') - \sup_{x' \in [a,x]} g(x') \right| \\
        &\leq \sup_{x' \in [a,x]}|\breve{g}(x') -  g(x')|, 
    \end{align}
    by the monotonicity of $g(\cdot)$.
    If we select the rearrangement operator $\mR^a[\cdot]$, Proposition 1 in \citet{chernozhukov2009improving} yields the same result for $x \in [a,b]$.
    Thus, whichever monotonization is chosen, we obtain the statement.
\end{proof}

\subsection{Proof of Theorem \ref{thm:asyN-ridge}}
\begin{lemma}
\label{lem:adj-consist-ridge}
Let Assumption \ref{asmp:feature}-\ref{asmp:link} hold. Define
\begin{align}
\label{eq:mu-n1}
    \mu_n={\bbeta^\top\hat{\bbeta}(\hat{g})},\quad
    \sigma_n^2=\|{\bP_{\bbeta}^\perp\hat{\bbeta}(\hat{g})\|}^2,
\end{align}
where $\bP_{\bbeta}^\perp=\bI_p-\bbeta\bbeta^\top$.
Then, we have
\begin{align}
    \abs{\hat{\mu}({\hat{g}})-\mu_n}\pconv0,\quad\mathrm{and}\quad\abs{\hat{\sigma}^2(\hat{{g}})-\sigma_n^2}\pconv0.
\end{align}
\end{lemma}
\begin{proof}[Proof of Lemma \ref{lem:adj-consist-ridge}]
Theorem 4.4 in \citet{bellec2022observable} implies that as $n_2 \to \infty$, we have
\begin{align}
    {\hat{v}_\lambda^2\hat{t}^2}{\dot{r}^{-4}}\abs{\hat{\mu}^2(\hat{g})-\mu_n^2}\pconv0,\quad{\hat{v}_\lambda^2\hat{t}^2}{\dot{r}^{-4}}\abs{\hat{\sigma}^2(\hat{g})-\sigma_n^2}\pconv0,
\end{align}
with $\hat{t}^2=(\hat{v}_\lambda+\lambda)^2 \|\hat\bbeta(\hat{g})\|^2-\kappa_2\dot{r}^2$ and $\dot{r}^2=n_2^{-1}\|{\by^{(2)}-\hat{g}({\bX^{(2)}}\hat{\bbeta}(\hat{g}))\|}^2$.
Recall that $\hat{v}_\lambda$ and $\bD$ are defined in Section \ref{sec:inferential_estimator}.
Thus, it is sufficient to show that $\hat{v}_\lambda^2,\hat{t}^2$, and $\dot{r}^{-4}$ are asymptotically lower bounded away from zero. First, the fact that $\trace(\bD)\ge n_2c_g^{-1}>0$ holds by Assumption \ref{asmp:link} and Proposition 3.1 in \citet{bellec2022observable} imply that there exists a constant $\hat{c}>0$ such that $\hat{v}_\lambda\ge c_g^{-1}/(1+\hat{c})-4\hat{c}/n_2$ holds. 
Next, since ridge penalized regression estimators satisfy $\|\hat\bbeta(\hat{g})\|\le C_\lambda'$ with a constant $C_\lambda'>0$ depending on the regularization parameter $\lambda>0$, we have $\dot{r}^2=O_\mathrm{p}(1)$. 
Also, Theorem 4.4 in \citet{bellec2022observable} implies that $\hat{t}^2\pconv((\hat{v}_\lambda+\lambda)\bbeta^\top\hat\bbeta(\hat{g}))^2$. Thus, we have
$\abs{\hat{\mu}({\hat{g}})-\mu_n}\pconv0$ and $\abs{\hat{\sigma}^2(\hat{{g}})-\sigma_n^2}\pconv0$ as $n_2\to\infty$ since the sign of $\mu_n$ is specified by an assumption.
\end{proof}

\begin{proof}[Proof of Theorem \ref{thm:asyN-ridge}] 
We use the notations defined in \eqref{eq:mu-n1}.
First, we can apply Theorem \ref{thm:master-zsc} and obtain
    \begin{align}
        \frac{\sqrt{p}(\hat{\bbeta}_j(\hat{g})-\mu_n\bbeta_j)}{\sigma_n}\dconv\mathcal{N}(0,1).
    \end{align}
    This is because we can skip Step 1 in the proof of Theorem \ref{thm:master-zsc} by $\bSigma=\bI_p$ and repeat Steps 2--3 since $J(\tilde\bU\bb)=J(\bb)$ for any orthogonal matrices $\tilde\bU\in\R^{p\times p}$. 
Hence, we have
\begin{align}
    \sqrt{p}\frac{\hat\beta_j(\hat{g})-\hat\mu(\hat{g})\beta_j}{\hat\sigma(\hat{g})}
    =\sqrt{p}\frac{\hat\beta_j(\hat{g})-\mu_n\beta_j}{\sigma_n}\frac{\sigma_n}{\hat\sigma(\hat{g})}
    +\sqrt{p}\frac{(\mu_n-\hat\mu(\hat{g}))\beta_j}{\hat\sigma(\hat{g})}\dconv\mathcal{N}(0,1),
\end{align}
where the convergence follows from the facts that $\hat\mu(\hat{g})\pconv\mu_n$ and $\hat\sigma^2(\hat{g})\pconv\sigma_n^2$ by Lemma \ref{lem:adj-consist-ridge}. This concludes the proof of \eqref{eq:asyN-ridge1}.

Next, we consider an orthogonal matrix $\bU\in\R^{p\times p}$ with the first row $\bU_1=\bv^\top$. Since $\bU\hat\bbeta(\hat{g})$ is the estimator given by \eqref{eq:estimator} with covariates $\bU\bX_i^{(2)}$ and the true coefficient vector $\bU\bbeta$, applying \eqref{eq:asyN1} to this with $j=1$ yields that, for any sequence of non-random vectors $\bv_n$ such that $\|\bv_n\|=1$ and $\sqrt{p}\tau(\bv_n)\bv_n^\top\bbeta=O(1)$, 
    \begin{align}
    \label{eq:asyN-ridge2}
        \frac{\sqrt{p}\bv_n^\top(\hat\bbeta(\hat{g})-\hat{\mu}(\hat{g})\bbeta)}{\hat\sigma(\hat{g})/\tau(\bv_n)}\dconv\mathcal{N}(0,1),
    \end{align}
    where $\tau^2(\bv_n)=(\bv_n^\top\bTheta\bv_n)^{-1}$.  
Finally, \eqref{eq:asyN3} follows from \eqref{eq:asyN-ridge2} and the Cram\'er-Wold device.
\end{proof}

{\rc 
\begin{remark}
    Under the Assumption E in \cite{bellec2022observable}, we obtain the explicit rate of convergence $|\mu_n-\hat\mu(\hat{g})|=O_p(n^{-1/2})$. Then, we can improve the condition $\sqrt{p}\beta_j=O(1)$ to $\beta_j=o(1)$ since
    \begin{align}
        \sqrt{p}\frac{(\mu_n-\hat\mu(\hat{g}))\beta_j}{\hat\sigma(\hat{g})}
        =O_p(\beta_j),
    \end{align}
    under the Assumption E in \cite{bellec2022observable}.
\end{remark}
}

\subsection{Proof of Theorem \ref{thm:asyN}}
First, we define the notations used in the proof. We consider an invertible matrix $\bL\in\R^{p\times p}$ satisfying $\bSigma=\bL\bL^\top$. Define, for each $i\in\{1,\ldots,n\}$,
\begin{align}
\label{eq:theta}
    \tilde\bX_i=\bL^{-1}\bX_i^{(2)},\quad
    \btheta = \bL^\top\bbeta,\quad
    \hat\btheta:=\hat\btheta(\hat{g})=\bL^\top\hat\bbeta(\hat{g}).
\end{align}

\begin{lemma}
\label{lem:mu-conv-2}
Let Assumption \ref{asmp:feature}-\ref{asmp:link}(1) hold. Using the notations \eqref{eq:theta}, define
\begin{align}
\label{eq:mu-n}
    \mu_0={\btheta^\top\hat{\btheta}},\quad
    \sigma_0^2=\|{\bP_{\btheta}^\perp\hat{\btheta}\|}^2,
\end{align}
where $\bP_{\btheta}^\perp=\bI_p-\btheta\btheta^\top$.
Then, we have
\begin{align}
    \abs{\hat{\mu}_0({\hat{g}})-\mu_0}\pconv0,\quad\mathrm{and}\quad\abs{\hat{\sigma}_0^2(\hat{{g}})-\sigma_0^2}\pconv0.
\end{align}
\end{lemma}
\begin{proof}[Proof of Lemma \ref{lem:mu-conv-2}]
Theorem 4.4 in \citet{bellec2022observable} implies that as $n_2 \to \infty$, we have
\begin{align}
    {\hat{v}_0^2\hat{t}_0^2}{\dot{r}_0^{-4}}\abs{\hat{\mu}_0(\hat{g})-\mu_0}\pconv0,\quad
    {\hat{v}_0^2\hat{t}_0^2}{\dot{r}_0^{-4}}\abs{\hat{\sigma}_0^2(\hat{g})-\sigma_0^2}\pconv0,
\end{align}
with $\hat{t}_0^2 ={n_2^{-1}\|{{\bX^{(2)}}\hat{\bbeta}(\hat{g})\|}^2}\hat{v}_0^2-\kappa_2(1-\kappa_2){\dot{r}_0^2}$ and $\dot{r}_0^2=n_2^{-1}\|{\by^{(2)}-\hat{g}({\bX^{(2)}}\hat{\bbeta}(\hat{g}))\|}^2$.
Recall that $\hat{v}_0$ is obtain by the definition of $\hat{v}_\lambda$ in Section \ref{sec:inferential_estimator} and setting $\lambda = 0$.
Thus, it is sufficient to show that $\hat{v}_0^2,\hat{t}_0^2$, and $\dot{r}_0^{-4}$ are asymptotically lower bounded away from zero. First, $\trace(\bD)\ge n_2c_g^{-1}>0$ by Assumption \ref{asmp:link} and Proposition 3.1 in \citet{bellec2022observable} imply that there exists a constant $\hat{c}'>0$ such that $\hat{v}_0\ge c_g^{-1}/(1+\hat{c}')-4\hat{c}'/n_2$. 
Next, we assume that $\|\hat\bbeta(\hat{g})\|\le C$ with probability approaching one, we have $\dot{r}_0^2=O_\mathrm{p}(1)$. 
Also, Theorem 4.4 in \citet{bellec2022observable} implies that $\hat{t}_0^2\pconv\hat{v}_0\mu_0$. Thus, we have
$\abs{\hat{\mu}_0({\hat{g}})-\mu_0}\pconv0$ and $\abs{\hat{\sigma}_0^2(\hat{{g}})-\sigma_0^2}\pconv0$ as $n_2\to\infty$ since the sign of $\mu_0$ is specified by an assumption.
\end{proof}

\begin{proof}[Proof of Theorem \ref{thm:asyN}] 
At first, the first step of the proof of Theorem \ref{thm:master-zsc} implies that, for any coordinate $j=1,\ldots,p$,
\begin{align}
    \tau_j\frac{\hat\beta_j-\hat\mu(\hat{g})\beta_j}{\hat\sigma(\hat{g})}
    =\frac{\hat\theta_j-\hat\mu(\hat{g})\theta_j}{\hat\sigma(\hat{g})},
\end{align}
where $\tau_j^{-2}=(\bSigma^{-1})_{jj}$. Here, $\btheta$ and $\hat{\btheta}$ are defined in \eqref{eq:theta}. Thus, we consider $\hat\btheta$ instead of $\hat\bbeta(\hat{g})$.
We have
\begin{align}
    \sqrt{p}\frac{\hat\theta_j-\hat\mu(\hat{g})\theta_j}{\hat\sigma(\hat{g})}
    =\sqrt{p}\frac{\hat\theta_j-\mu_0\theta_j}{\sigma_0}\frac{\sigma_n}{\hat\sigma(\hat{g})}
    +\sqrt{p}\frac{(\mu_0-\hat\mu(\hat{g}))\theta_j}{\hat\sigma(\hat{g})}.
\end{align}
Thus, the facts that $\hat\mu_0(\hat{g})\pconv\mu_0$ and $\hat\sigma_0^2(\hat{g})\pconv\sigma_0^2$ by Lemma \ref{lem:mu-conv-2} conclude the proof of \eqref{eq:asyN1}.
The rest of the proof follows from repeating the arguments in the proof of Theorem \ref{thm:asyN-ridge}.
\end{proof}

\subsection{Proof of Theorem \ref{thm:adj-equiv}}
\begin{lemma}
\label{lem:lip-xb}
Let $c_g^{-1}\le g'(\cdot)$ hold. 
Consider censoring of $\hat\bbeta(\hat{g})^\top{\bX_i^{(2)}}$ and $\hat\bbeta({g})^\top{\bX_i^{(2)}}$ for all $i\in[n_2]$ in $[a,b]$.
Under the setting of Lemma \ref{lem:fan1993} with $k=3$, we have
\begin{align}
\label{eq:max-Z-conv}
    \max_{i=1,\ldots,n_2}\abs{\hat\bbeta(g)^\top\bX_i^{(2)}-\hat\bbeta(\hat{g})^\top\bX_i^{(2)}}\pconv0.
\end{align}
\end{lemma}
\begin{proof}[Proof of Lemma \ref{lem:lip-xb}]
We can assume $\bX_i^{(2)}\sim\mathcal{N}(\bzero,\bI_p)$ for each $i=1,\ldots,n_2$ without loss of generality by the first step of the proof of Theorem \ref{thm:master-zsc}. In this proof, we omit $(2)$ on $\bX^{(2)}$ for simplicity of the notation.
To begin with, we write the KKT condition of the estimation: 
\begin{align}
    f(\hat\bbeta(g),g)=\bzero,
\end{align}
where we define $f(\bb,g)=n_2^{-1/2}\bX^\top(g(\bX\bb)-\by)$. 
We write $f_j(\bb,g)=n_2^{-1/2}\bX_{\cdot j}^\top(g(\bX\bb)-\by)$ for $j=1,\ldots
,p$. 
Since $\partial/(\partial b_j)f_j(\bb,g)=n_2^{-1/2}\bX_{\cdot j}^\top\bD(\bb)\bX_{\cdot j}$ with $\bD(\bb)=\mathrm{diag}(g'(\bX\bb))$, by the mean value theorem, there exists a constant $c\in[0,1]$ such that $\bar\bb=c\hat\bbeta({g})+(1-c)\hat\bbeta(\hat{g})$ satisfies
\begin{align}
\label{eq:lLip-fj}
    \frac{f_j(\hat\bbeta(g),g)-f_j(\hat\bbeta(\hat{g}),g)}{\hat\beta_j(g)-\hat\beta_j(\hat{g})}
    =\rbr{\frac{1}{\sqrt{n_2}}\bX_{\cdot j}^\top\bD(\bar\bb)\bX_{\cdot j}}>0.
\end{align}
Define $R_{k} := \hat{g}(\bX_k^\top\hat\bbeta(\hat{g}))-g(\bX_k^\top\hat{\bbeta}(\hat{g}))$. We have
\begin{align}
    &\sqrt{n_2}\rbr{\hat\beta_j(g)-\hat\beta_j(\hat{g})}\\
    &= \rbr{n_2^{-1}\bX_{\cdot j}^\top\bD(\bar\bb)\bX_{\cdot j}}^{-1}\rbr{f_j(\hat\beta_j(g),g)-f_j(\hat\beta_j(\hat{g}),g)}\\
    &= \rbr{n_2^{-1}\bX_{\cdot j}^\top\bD(\bar\bb)\bX_{\cdot j}}^{-1}
    \cbr{{f_j(\hat\beta_j(g),g)}+\rbr{f_j(\hat\beta_j(\hat{g}),\hat{g})-f_j(\hat\beta_j(\hat{g}),g)}-{f_j(\hat\beta_j(\hat{g}),\hat{g})}}\\
    &= \rbr{n_2^{-1}\bX_{\cdot j}^\top\bD(\bar\bb)\bX_{\cdot j}}^{-1}
    \rbr{\frac{1}{\sqrt{n_2}}\sum_{k=1}^{n_2}X_{kj}R_k},
\end{align}
where the second equality follows from the first-order conditions. 
In sequel, for simplicity, we consider the leave-one-out estimator $\hat\bbeta_{-i}$ and $\hat{g}_{-i}$ constructed by the observations without the $i$-th sample. Define 
\begin{align}
    &\Tilde{R}_k:=(\log {n_2})\rbr{\hat{g}_{-1}(\bX_k^\top\hat\bbeta_{-1}(\hat{g}_{-1}))-g(\bX_k^\top\hat{\bbeta}_{-1}(\hat{g}_{-1}))},
\end{align}
and
\begin{align}
    &T_j:=\rbr{n_2^{-1}\bX_{-1, j}^\top\bD_{-1}(\bar\bb)\bX_{-1, j}}^{-1},
\end{align}
where we define the leave-one-out design terms as $\bX_{-1,j}:=(X_{2j},\ldots,X_{{n_2}j})^\top\in\R^{{n_2}-1}$, and $\bD_{-1}(\bar\bb):=\diag({g}_{-1}'(\bX_2^\top\bar\bb),\ldots,{g}_{-1}'(\bX_{n_2}^\top\bar\bb))\in\R^{({n_2}-1)\times({n_2}-1)}$.
We obtain
\begin{align}
    \abs{\bX_1^\top\hat\bbeta_{-1}(g)-\bX_1^\top\hat\bbeta_{-1}(\hat{g}_{-1})}
    &=\abs{\sum_{j=1}^pX_{1j}\rbr{\hat\beta_{-1,j}(g)-\hat\beta_{-1,j}(\hat{g}_{-1})}}\\
    &\le \abs{\frac{1}{{{n_2}\log {n_2}}}\sum_{j=1}^pX_{1j}T_j
    {\sum_{k=2}^{n_2}X_{kj}\tilde{R}_k}}\label{eq:mclt}.
\end{align}
Here, define a filtration $\mF_k=\sigma(\{\hat{g}_{-1},\hat{\bbeta}_{-1}(\hat{g}_{-1}),T_1,\ldots,T_p,\bX_2,\ldots,\bX_{k+1}\})$ with an initialization $\mF_0=\sigma(\{\hat{g}_{-1},\hat{\bbeta}_{-1}(\hat{g}_{-1}),T_1,\ldots,T_p\})$. Define a random variable $\Tilde{S}_k={n_2}^{-1}\sum_{j=1}^pT_jX_{1j} X_{kj}$. Then, $\tilde{R}_k\Tilde{S}_k$ is a martingale difference sequence since $\E[\tilde{R}_k\Tilde{S}_k\mid\mF_{k-1}]=0$ and
$\E|{\tilde{R}_k\Tilde{S}_k}|
    \le\E[\tilde{R}_k^2]^{1/2}\E[\tilde{S}_k^2]^{1/2}<\infty.$
This follows from the fact that 
\begin{align}
    \E\sbr{\Tilde{S}_k^2}
    &=\E\sbr{\rbr{n_2^{-1}\sum_{j=1}^pT_jX_{1j}X_{kj}}^2}\\
    &=n_2^{-2}\sum_{j=1}^p\E\sbr{T_j^2X_{1j}^2X_{kj}^2}+2\sum_{j<j'}\E\sbr{T_jX_{1j}X_{kj}T_{j'}X_{1j'}X_{kj'}}\\
    &=n_2^{-2}\sum_{j=1}^p\E[T_j^2X_{kj}^2]
    \le n_2^{-2}\sum_{j=1}^p\E[T_j^4]^{1/2}\E[X_{kj}^4]^{1/2}.
\end{align}
The last inequality follows from the Cauchy-Schwartz inequality.
Since $X_{kj}$ is the standard Gaussian, $\E[X_{kj}^4]=3$ holds.
Also, we have $0<T_j\le c_g({{n_2}}^{-1}\sum_{l=2}^{n_2}X_{lj}^2)^{-1}$, where $({({n_2}-1)^{-1}\sum_{l=2}^{n_2}X_{lj}^2})^{-1}$ follows the inverse gamma distribution with parameters $(({n_2}-1)/2,2/({n_2}-1))$ and the bounded fourth moment $(2/({n_2}-1))^4\Gamma(2/({n_2}-1)-4)/\Gamma(2/({n_2}-1))$.

Let $(\log {n_2})^{-m/2}\tilde{R}:=\sup_{a\le x\le b}\abs{\hat{g}_{-1}(x)-g(x)}$. 
Note that, since $\tilde{R}=O_\mathrm{p}(1)$ by Lemma \ref{lem:fan1993}, for any $\epsilon_1>0$, there exists $\Bar{c}>0$ such that we have $\Pr\rbr{R_k>\Bar{c}}\le\epsilon_1$.
Also, note that censoring does not affect this fact since $\hat{g}(\cdot)$ is given independent of $\bX$.
Hence, we obtain, for $\Bar{c}$ and any $t_n>0$ satisfying $t_n=o(\sqrt{n_2})$,
\begin{align}
    &\Pr\rbr{\left.\frac{1}{\sqrt{n_2}}\max_{2\le k\le {n_2}}\abs{\tilde{R}_k\sum_{j=1}^pT_jX_{1j}X_{kj}}>t_n\ \right|\ |\tilde{R}|\le\Bar{c}}\\
    &\le\Pr\rbr{\frac{\Bar{c}}{\sqrt{n_2}}\max_{2\le k\le {n_2}}\abs{\sum_{j=1}^pT_jX_{1j}X_{kj}}>t_n}
    \\
    &\le\Pr\rbr{\left.\frac{\Bar{c}}{\sqrt{{n_2}}}\max_{2\le k\le {n_2}}\abs{\sum_{j=1}^pX_{1j}X_{kj}T_j}>t_n\ \right|\ \max_{1\le j\le p}|T_j|\le u}+\Pr\rbr{\max_{1\le j\le p}|T_j|> u}\\
    &\le2{n_2}\exp\rbr{-\frac{ct_n^2}{\Bar{c}^2K^2}}+\Pr\rbr{\max_{1\le j\le p}|T_j|> u}.
    \label{ineq:R_tilde}
\end{align}
with some $c>0$ depending on $u$, where the last inequality follows from the union bound and Bernstein's inequality. Here, $K$ is the sub-exponential norm of $uX_{11}X_{21}$. 
Since we have $T_j\le c_g({n_2}^{-1}\sum_{l=2}^{n_2}X_{lj}^2)$, it holds that
\begin{align}
    \Pr\rbr{\max_{1\le j\le p}|T_j|> u}
    &\le\Pr\rbr{\max_{1\le j\le p}{\frac{1}{{n_2}}\sum_{l=2}^{n_2}X_{lj}^2-1}>u_*} \\
    &\le p\exp\rbr{-c\rbr{\frac{u_*^2}{K^2}\wedge\frac{u_*}{K}}{n_2}},\label{ineq:Tj}
\end{align}
where $u_*=c_g/u-1$.
Using the bounds, Azuma-Hoeffding's inequality yields, for any $x,u_n>0$ and $t_n>0$ satisfying $t_n=o(\sqrt{{n_2}})$,
\begin{align}
    &\Pr\rbr{\frac{1}{(\log {n_2})^{m/2}}\max_{1\le i\le {n_2}}\abs{\frac{1}{{n_2}}\sum_{k\neq i}^{n_2}\tilde{R}_k\sum_{j=1}^{n_2}X_{ij}X_{kj}}>x}\\
    &\le\Pr\rbr{\left.\frac{1}{(\log {n_2})^{m/2}}\max_{1\le i\le {n_2}}\abs{\frac{1}{{n_2}}\sum_{k\neq i}^{n_2}\tilde{R}_k\sum_{j=1}^{n_2}X_{ij}X_{kj}}>x\ \right|\ |{\tilde{R}}|\le\Bar{c}}+\epsilon_1\\
    &\le {n_2}\Pr\rbr{\left.\frac{1}{(\log {n_2})^{m/2}}\abs{\frac{1}{{n_2}}\sum_{k=2}^{n_2}\tilde{R}_k\sum_{j=1}^{n_2}X_{1j}X_{kj}}>x\ \right|\ |{\tilde{R}}|\le\Bar{c}}+\epsilon_1\\
    &\le {n_2}\Pr\rbr{\left.\frac{1}{(\log {n_2})^{m/2}}\abs{\frac{1}{{n_2}}\sum_{k=2}^{n_2}\tilde{R}_k\sum_{j=1}^{n_2}X_{1j}X_{kj}}>x\ 
    \right|\ \frac{1}{n}\max_{2\le k\le {n_2}}\abs{\tilde{R}_k\sum_{j=1}^pX_{1j}X_{kj}}\le\frac{t_n}{\sqrt{{n_2}}}, |\tilde{R}|\le\Bar{c}}\\
    &\quad+{n_2}\Pr\rbr{\left.\frac{1}{{n_2}}\max_{2\le k\le {n_2}}\abs{\tilde{R}_k\sum_{j=1}^pX_{1j}X_{kj}}>\frac{t_n}{\sqrt{{n_2}}}\ \right|\ |\tilde{R}|\le\Bar{c}}+\epsilon_1\\
    &\le2{n_2}\exp\rbr{-\frac{x^2(\log {n_2})^{m}}{2t_n^2}}
    +2n_2^2\exp\rbr{-\frac{ct_n^2}{\Bar{c}^2K^2}}\\
    &\quad +n_2^2\exp\rbr{-c\rbr{\frac{u_*^2}{K^2}\wedge\frac{u_*}{K}}{n_2}}
    +\epsilon_1,
\end{align}
where the last inequality follows from \eqref{ineq:R_tilde} and \eqref{ineq:Tj}.
Thus, one can choose, for instance, $m=3$, $t_n=(\log {n_2})^{3/5}$, and $\Bar{c}=\log\log {n_2}$ so that we have $\frac{1}{(\log {n_2})^{m/2}}\max_{1\le i\le {n_2}}|{\frac{1}{{n_2}}\sum_{k\neq i}^{n_2}\tilde{R}_k\sum_{j=1}^{n_2}X_{ij}X_{kj}}|=o_\mathrm{p}(1)$ and $\epsilon_1\to0$.
\end{proof}

\begin{proof}[Proof of Theorem \ref{thm:adj-equiv}]
In this proof, we omit the superscript $(2)$ on $\bX^{(2)}$ and $\by^{(2)}$ for simplicity of the notation.
We firstly rewrite the inferential parameters defined in Section \ref{sec:link_estimation_inferential} as
\begin{align}
    \hat{\mu}_{0c}^2(g)=\frac{\|\iota(\bX\hat\bbeta(g))\|^2}{n_2}-\kappa_2(1-\kappa_2)\hat\sigma^2(g),\quad
    \hat\sigma^2_{0c}(g)=\frac{n_2^{-1}\|\by-g(\iota(\bX\hat\bbeta(g))\|^2}{\rbr{n_2^{-1}\trace(\bV(g))}^2},
\end{align}
where $\bV_c(g)=\bD_c(g)-\bD_c(g)\bX(\bX^\top\bD_c(g)\bX)^{-1}\bX^\top\bD_c(g)$.
Since we have 
\begin{align}
    &\abs{\hat\sigma_{0c}^2(\hat{g})-\hat\sigma_{0c}^2(g)}\\
    &\le\frac{1}{\rbr{n_2^{-1}\trace(\bV_c(\hat{g}))}^2\rbr{n_2^{-1}\trace(\bV_c(g))}^2}\\ 
    &\qquad \times 
    \left\{\frac{\|\by-g(\iota(\bX\hat\bbeta(g))\|^2}{n_2}\abs{\rbr{\frac{\trace(\bV_c(\hat{g}))}{n_2}}^2-\rbr{\frac{\trace(\bV_c(g))}{n_2}}^2}\right.\\
    &\quad\left.+\rbr{\frac{\trace(\bV_c(g))}{n_2}}^2\abs{\frac{\|\by-g(\iota(\bX\hat\bbeta(g))\|^2}{n_2}-\frac{\|\by-g(\bX\hat\bbeta(\hat{g}))\|^2}{n_2}}\right\},
\end{align}
It is sufficient to show the following properties:
\begin{gather}
    n_2^{-1}\abs{{\|\iota(\bX\hat\bbeta(\hat{g}))\|^2}-{\|\iota(\bX\hat\bbeta({g}))\|^2}}=o_\mathrm{p}(1),\label{eq:SEg-1}\\
    n_2^{-1}\abs{{\|\by-g(\iota(\bX\hat\bbeta(\hat{g})))\|^2}-{\|\by-g(\iota(\bX\hat\bbeta({g})))\|^2}}=o_\mathrm{p}(1),\label{eq:SEg-2}\\
    n_2^{-1}\abs{{\trace(\bV_c(\hat{g}))}-{\trace(\bV_c(g))}}=o_\mathrm{p}(1).\label{eq:SEg-3}
\end{gather}

For \eqref{eq:SEg-1}, immediately we have
\begin{align}
    &n_2^{-1}\abs{{\|\iota(\bX\hat\bbeta(\hat{g}))\|^2}-{\|\iota(\bX\hat\bbeta({g}))\|^2}}\\
    &=\abs{n_2^{-1}\sum_{i=1}^{n_2}\rbr{\iota(\bX_i^\top\hat\bbeta(\hat{g}))-\iota(\bX_i^\top\hat\bbeta(g))}\rbr{\iota(\bX_i^\top\hat\bbeta(\hat{g}))+\iota(\bX_i^\top\hat\bbeta(g))}}.
\end{align}
Since $\max_{i=1,\ldots,{n_2}}|{\iota(\bX_i^\top\hat\bbeta(g))-\iota(\bX_i^\top\hat\bbeta(\hat{g}))}|\pconv0$ as ${n_2}\to\infty$ by Lemma \ref{lem:lip-xb}, this term converges in probability to zero.

Next, for \eqref{eq:SEg-2}, since we have
\begin{align}
    &n_2^{-1}\abs{{\|\by-g(\iota(\bX\hat\bbeta(\hat{g})))\|^2}-{\|\by-g(\iota(\bX\hat\bbeta({g})))\|^2}}\\
    &=\abs{n_2^{-1}\sum_{i=1}^{n_2}\rbr{\hat{g}(\iota(\bX_i^\top\bbeta(\hat{g})))-{g}(\iota(\bX_i^\top\bbeta(g)))}\rbr{2y_i-\hat{g}(\iota(\bX_i^\top\bbeta(\hat{g})))-{g}(\iota(\bX_i^\top\bbeta(g)))}},
\end{align}
we should bound $\hat{g}(\iota(\bX_i^\top\bbeta(\hat{g})))-{g}(\iota(\bX_i^\top\bbeta(g)))$. Indeed, using the triangle inequality reveals
\begin{align}
    &\abs{\hat{g}(\iota(\bX_i^\top\bbeta(\hat{g})))-{g}(\iota(\bX_i^\top\bbeta(g)))}\\
    &\le\abs{{g}(\iota(\bX_i^\top\bbeta(\hat{g})))-{g}(\iota(\bX_i^\top\bbeta(g)))} 
    +\abs{\hat{g}(\iota(\bX_i^\top\bbeta(\hat{g})))-{g}(\iota(\bX_i^\top\bbeta(\hat{g})))}.
\end{align}
The first term on the right-hand side is $o_\mathrm{p}(1)$ by the Lipschitz continuity of $g(\cdot)$ and Lemma \ref{lem:lip-xb}. Also, the second term is upper bounded by $\sup_x|\hat{g}(x)-g(x)|$, and is $o_\mathrm{p}(1)$ by Lemma \ref{lem:fan1993}.

To achieve \eqref{eq:SEg-3}, we first have
\begin{align}
    &n_2^{-1}\abs{{\trace(\bV_c(\hat{g}))}-{\trace(\bV_c(g))}}\\
    &\le n_2^{-1}\abs{\trace\rbr{\bD_c(\hat{g})-\bD_c(g)}}\\
    &\quad+n_2^{-1}|\trace(\bD_c(\hat{g})\bX(\bX^\top\bD_c(\hat{g})\bX)^{-1}\bX^\top\bD_c(\hat{g}) \\
    & \qquad -\bD_c(g)\bX(\bX^\top\bD_c(g)\bX)^{-1}\bX^\top\bD_c(g))|.
\end{align}
For the first term, we have 
\begin{align}
    &n_2^{-1}\abs{\trace\rbr{\bD_c(\hat{g})-\bD_c(g)}}\\
    &\le n_2^{-1}\sum_{i=1}^{n_2}\abs{\hat{g}'(\iota(\bX_i^\top\hat\bbeta(\hat{g})))-g(\iota(\bX_i^\top\hat\bbeta(g)))}\\
    &\le\sup_{a\le x\le b}|\hat{g}'(x)-g'(x)|+Bn_2^{-1}\sum_{i=1}^{n_2}\abs{\iota(\bX_i^\top\hat\bbeta(\hat{g}))-\iota(\bX_i^\top\hat\bbeta({g}))}\\
    &=o_\mathrm{p}(1),
\end{align}
by Lemma \ref{lem:deriv-unif} and Lemma \ref{lem:lip-xb}. 
For the second term, the triangle inequality yields
\begin{align}
    &n_2^{-1}|\trace(\bD_c(\hat{g})\bX(\bX^\top\bD_c(\hat{g})\bX)^{-1}\bX^\top\bD_c(\hat{g})\\
    &\quad -\bD_c(g)\bX(\bX^\top\bD_c(g)\bX)^{-1}\bX^\top\bD_c(g))|\\
    &\le n_2^{-1}\abs{\trace\rbr{\cbr{\bD_c({g})-\bD_c(\hat{g})}\bX(\bX^\top\bD_c(g)\bX)^{-1}\bX^\top\bD_c(g)}}\label{eq:SEg2-1}\\
    &\quad+n_2^{-1}\abs{\trace\rbr{\bD_c(\hat{g})\bX\cbr{(\bX^\top\bD_c(g)\bX)^{-1}-(\bX^\top\bD_c(\hat{g})\bX)^{-1}}\bX^\top\bD_c(g)}}\label{eq:SEg2-2}\\
    &\quad+n_2^{-1}\abs{\trace\rbr{\bD_c(\hat{g})\bX(\bX^\top\bD_c(\hat{g})\bX)^{-1}\bX^\top\cbr{\bD_c(g)-\bD_c(\hat{g})}}}\label{eq:SEg2-3}.
\end{align}
Using the Cauchy-Schwartz inequality, \eqref{eq:SEg2-1} is bounded by
\begin{align}
n_2^{-1}\norm{\bD_c(g)-\bD_c(\hat{g})}_F\norm{\bX(\bX^\top\bD_c(g)\bX)^{-1}\bX^\top\bD_c(g)}_F.    
\end{align}
Here, we have 
\begin{align}
    &n_2^{-1/2}\norm{\bD_c(g)-\bD_c(\hat{g})}_F\\
    &\le\rbr{\frac{1}{{n_2}}\sum_{i=1}^{n_2}\cbr{g'(\iota(\bX_i^\top\hat\bbeta(g)))-\hat{g}'(\bX_i^\top \hat{\bbeta}(\hat{g}))}^2}^{1/2}\\
    &\le\rbr{\frac{1}{{n_2}}\sum_{i=1}^{n_2}\sbr{2\cbr{g'(\iota(\bX_i^\top\hat\bbeta(g)))-{g}'(\bX_i^\top \hat{\bbeta}(\hat{g}))}^2
    +2\cbr{{g}'(\iota(\bX_i^\top\hat\bbeta(\hat{g})))-\hat{g}'(\bX_i^\top \hat{\bbeta}(\hat{g}))}^2}}^{1/2}\\
    &\le\rbr{\frac{2}{{n_2}}\sum_{i=1}^{n_2}\cbr{g'(\iota(\bX_i^\top\hat\bbeta(g)))-{g}'(\bX_i^\top \hat{\bbeta}(\hat{g}))}^2}^{1/2} \\
     & \quad +\rbr{\frac{2}{{n_2}}\sum_{i=1}^{n_2}\cbr{{g}'(\iota(\bX_i^\top\hat\bbeta(\hat{g})))-\hat{g}'(\bX_i^\top \hat{\bbeta}(\hat{g}))}^2}^{1/2}\\
    &\le  2B\max_{i=1,\ldots,{n_2}}\abs{\iota(\bX_i^\top\hat\bbeta(g))-\iota(\bX_i^\top\hat\bbeta(\hat{g}))} 
    +2\sup_{a\le x\le b}|\hat{g}'(x)-g'(x)| \\
    &=o_\mathrm{p}(1).
\end{align}
by Lemma \ref{lem:deriv-unif} and Lemma \ref{lem:lip-xb}. Also, we have
\begin{align}
    &n_2^{-1/2}\norm{\bX(\bX^\top\bD_c(g)\bX)^{-1}\bX^\top\bD_c(g)}_F\\
    &=n_2^{-1/2}\norm{\bD_c(g)^{-1/2}\bD_c(g)^{1/2}\bX(\bX^\top\bD_c(g)\bX)^{-1}\bX^\top\bD_c(g)^{1/2}\bD_c(g)^{1/2}}_F\\
    &\le n_2^{-1/2}\norm{\bD_c(g)^{-1/2}}_{\rm op}\norm{\bD_c(g)^{1/2}}_{\rm op}\norm{\bD_c(g)^{1/2}\bX(\bX^\top\bD_c(g)\bX)^{-1}\bX^\top\bD_c(g)^{1/2}}_F\\
    &=n_2^{-1/2}\norm{\bD_c(g)^{-1/2}}_{\rm op}\norm{\bD_c(g)^{1/2}}_{\rm op}\sqrt{\trace(\bI_{n_2})}\\
    &=\norm{\bD_c(g)^{-1/2}}_{\rm op}\norm{\bD_c(g)^{1/2}}_{\rm op}.
\end{align}
Since $\norm{\bD_c(g)^{1/2}}_{\rm op}\leq \sup_x g'(x)^{1/2}$ and $\norm{\bD_c(g)^{-1/2}}_{\rm op}\leq (\inf_x g'(x))^{-1/2}$ are constants by an assumption of $g(\cdot)$,
we conclude that \eqref{eq:SEg2-1} is $o_\mathrm{p}(1)$. \eqref{eq:SEg2-3} is also shown to be $o_\mathrm{p}(1)$ in a similar manner. Since $\bA^{-1}-\bB^{-1}=-\bA^{-1}(\bA-\bB)\bB^{-1}$ for two invertible matrices $\bA$ and $\bB$, \eqref{eq:SEg2-2} can be rewritten as 
\begin{align}
    n_2^{-1}\abs{\trace\rbr{\bD_c(\hat{g})\bX(\bX^\top\bD_c(\hat{g})\bX)^{-1}\bX^\top\cbr{\bD_c(g)-\bD_c(\hat{g})}\bX(\bX^\top\bD_c(g)\bX)^{-1}\bX^\top\bD_c(g)}},
\end{align}
and a similar technique used above provides the upper bound,
\begin{align}
    &n_2^{-1}\norm{\bD_c(\hat{g})}_{\rm op}^{1/2}\norm{\bD_c(\hat{g})^{-1}}_{\rm op}^{1/2}\norm{\bD_c({g})}_{\rm op}^{1/2}\norm{\bD_c({g})^{-1}}_{\rm op}^{1/2}
    \norm{\bD_c(g)-\bD_c(\hat{g})}_{\rm op}\\
    &\quad\times\norm{\bD_c(\hat{g})^{1/2}\bX(\bX^\top\bD_c(\hat{g})\bX)^{-1}\bX^\top\bD_c(\hat{g})^{1/2}}_F \\
    &\quad \times \norm{\bD_c({g})^{1/2}\bX(\bX^\top\bD_c({g})\bX)^{-1}\bX^\top\bD_c({g})^{1/2}}_F\\
    &=\norm{\bD_c(\hat{g})}_{\rm op}^{1/2}\norm{\bD_c(\hat{g})^{-1}}_{\rm op}^{1/2}\norm{\bD_c({g})}_{\rm op}^{1/2}\norm{\bD_c({g})^{-1}}_{\rm op}^{1/2}
    \norm{\bD_c(g)-\bD_c(\hat{g})}_{\rm op}.
\end{align}
Here, $\norm{\bD_c({g})}_{\rm op}^{1/2}\norm{\bD_c({g})^{-1}}_{\rm op}^{1/2}$ is a constant by an assumption, and also the term $\norm{\bD_c(\hat{g})}_{\rm op}^{1/2}\norm{\bD_c(\hat{g})^{-1}}_{\rm op}^{1/2}$ is asymptotically bounded by the uniform consistency of $\hat{g}'$ for $g'$ by Lemma \ref{lem:deriv-unif}. Finally, we have
\begin{align}
    &\norm{\bD_c(g)-\bD_c(\hat{g})}_{\rm op}\\
    &=\max_{i=1,\ldots,{n_2}}\abs{g'(\iota(\bX_i^\top\hat\bbeta(g)))-\hat{g}'(\iota(\bX_i^\top\hat\bbeta(\hat{g})))}\\
    &\le\max_{i=1,\ldots,{n_2}}\abs{g'(\iota(\bX_i^\top\hat\bbeta(g)))-g'(\iota(\bX_i^\top\hat\bbeta(\hat{g})))}\\
    & \quad +\max_{i=1,\ldots,{n_2}}\abs{g'(\iota(\bX_i^\top\hat\bbeta(\hat{g})))-\hat{g}'(\iota(\bX_i^\top\hat\bbeta(\hat{g})))}\\
    &\le B\max_{i=1,\ldots,{n_2}}\abs{\iota(\bX_i^\top\hat\bbeta(g))-\iota(\bX_i^\top\hat\bbeta(\hat{g}))}
    +\sup_{a\le x\le b}|g'(x)-\hat{g}'(x)|\\
    &=o_\mathrm{p}(1),
\end{align}
by Lemma \ref{lem:deriv-unif} and Lemma \ref{lem:lip-xb}. Thus, \eqref{eq:SEg2-2} is $o_\mathrm{p}(1)$. Combining these results concludes the proof.
\end{proof}

\subsection{Proof of Proposition \ref{prop:efficiency}}
\begin{proof}[Proof of Proposition \ref{prop:efficiency}]
Let $\hat\bbeta:=\hat\bbeta(\hat{g})$ for simplicity of the notation.
Recall that when $J(\cdot)\equiv\bzero$,
\begin{align}
    \tilde\mu_\mathrm{LS}^2&=n_1^{-1}\|{\bX\Tilde{\bbeta}_\mathrm{LS}\|}^2-(1-\kappa_1)\tilde\sigma_\mathrm{LS}^2,\quad
    \tilde\sigma_\mathrm{LS}^2=\frac{\kappa_1}{n_1(1-\kappa_1)^2}\|{\by-\bX\Tilde{\bbeta}_\mathrm{LS}\|}^2,\\
    \hat{\mu}^2(\hat{g})&=n_2^{-1}\|\bX\hat\bbeta(\hat{g})\|^2-(1-\kappa_2)\hat{\sigma}^2(\hat{g}),\quad
    \hat{\sigma}^2(\hat{g})=\frac{\kappa_2}{n_2\hat{v}_\lambda^2}{\|\by-\hat{g}(\bX\hat\bbeta(\hat{g}))\|^2}.
\end{align}
Since we have
\begin{align}
    \frac{\tilde\mu_\mathrm{LS}^2}{\tilde\sigma_\mathrm{LS}^2}
    =\frac{\|\bX\tilde\bbeta_\mathrm{LS}\|^2}{\frac{\kappa_1}{(1-\kappa_1)^2}\|\by-\bX\tilde\bbeta_\mathrm{LS}\|^2}-(1-\kappa_1),
\end{align}
and
\begin{align}
    \frac{\hat\mu^2(\hat{g})}{\hat\sigma^2(\hat{g})}
    =\frac{\|\bX\hat\bbeta(\hat{g})\|^2}{\frac{\kappa_2}{\hat{v}_\lambda^2}\|\by-\hat{g}(\bX\hat\bbeta(\hat{g}))\|^2}-(1-\kappa_2),
\end{align}
$\hat\sigma^2(\hat{g})/\hat\mu^2(\hat{g})<\tilde\sigma_\mathrm{LS}^2/\tilde\mu_\mathrm{LS}^2$ is equivalent to
\begin{align}
    {\frac{\|\bX\hat\bbeta(\hat{
g})\|}{\|\bX\tilde\bbeta_\mathrm{LS}\|}
        \cdot\frac{\abs{\hat{v}_\lambda}}{1-\kappa_1}
        \cdot\frac{\|\by-\bX\tilde\bbeta_\mathrm{LS}\|}{\|\by-\hat{g}(\bX\hat\bbeta(\hat{g}))\|}}>1.
\end{align}
Next, when $J(\bb)=\lambda\|\bb\|^2$, recall that
\begin{align}
    \tilde\mu^2&=\|\tilde\bbeta\|^2-\tilde\sigma^2, \quad
    \tilde\sigma^2=\kappa_1n_1^{-1}\|\by-\bX\tilde\bbeta\|^2(\Tilde{v}^2+\lambda_1)^{-2}\\
    \hat\mu^2(\hat{g})&=\|\hat\bbeta(\hat{g})\|^2-\hat\sigma^2(\hat{g}), \quad
    \hat\sigma^2(\hat{g})=\kappa_2n_1^{-1}\|\by-\bX\hat\bbeta(\hat{g})\|^2(\hat{v}_\lambda^2+\lambda)^{-2}.
\end{align}
Thus, in a similar way as when $J(\cdot)\equiv\bzero$, we conclude the proof.
\end{proof}

\section{Additional Numerical Experiments}
In this section, we present an additional numerical experiment illustrating variable selection based on the asymptotic normality results derived in this work. 
Specifically, by plugging the observable adjustments into our inference procedure, we can compute p-values for each variable, which allows us to construct a variable selection method with false discovery rate control via the Benjamini--Hochberg procedure \citep{benjamini1995controlling}. 
The construction of the test statistics follows the same steps as in Section \ref{sec:experiment-coef}. 
As a benchmark, we also implement the Model-X knockoff procedure using the lasso coefficient difference as the feature importance statistic.

The experimental setup is as follows. We consider a sample size of $n=2000$ and a feature dimension of $p=500$, with the target false discovery rate level set to $q=0.05$. 
Also, $\bm X_i\overset{\rm iid}{\sim} \mathcal{N}(\bm 0,\bm I_p)$, and we made a knockoff by using the knowledge of this distribution.
The data-generating process and the estimation procedure follow the \textsf{Piecewise} setting described in the paper. 
\begin{figure}[h]
\centering
\begin{minipage}[b]{0.45\columnwidth}
    \centering
    \includegraphics[width=0.9\columnwidth]{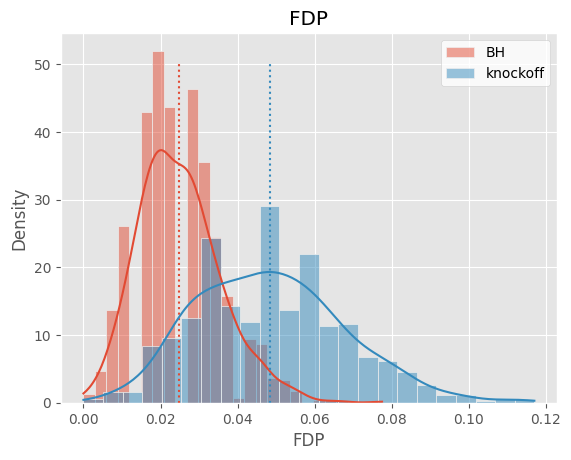}
    \caption{False discovery rate.}
    \label{fig:fs1}
\end{minipage}
\begin{minipage}[b]{0.45\columnwidth}
    \centering
    \includegraphics[width=0.9\columnwidth]{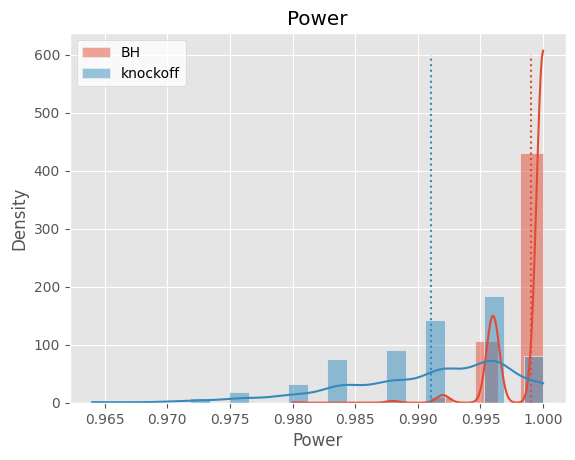}
    \caption{Power.}
    \label{fig:fs2}
\end{minipage}
\end{figure}
Figures \ref{fig:fs1}--\ref{fig:fs2} illustrate the selection results over 1000 replications via histograms of 
\begin{align}
    \frac{\#\{\text{falsely selected}\}}{\#\{\text{selected}\}}\quad\text{and}\quad
    \frac{\#\{\text{correctly selected}\}}{\#\{\text{non-null}\}},
\end{align}
respectively.
The dotted lines represent sample means.
We can see that our proposed method controls the false discovery rate under the nominal level and achieves higher power than Model-X knockoff.

\subsection{Real Data Analysis}
\label{sec:real-data-detail}
In Figure \ref{fig:lam1}--\ref{fig:lam2}, we present a detailed dependence of the estimated effective asymptotic variance on $\lambda$ for the two real datasets described in Section \ref{sec:real-data}. 
This illustrates the advantage of the proposed method over a range of values of $\lambda$.

\begin{figure}[h]
\centering
\begin{minipage}[b]{0.49\columnwidth}
    \centering
    \includegraphics[width=0.9\columnwidth]{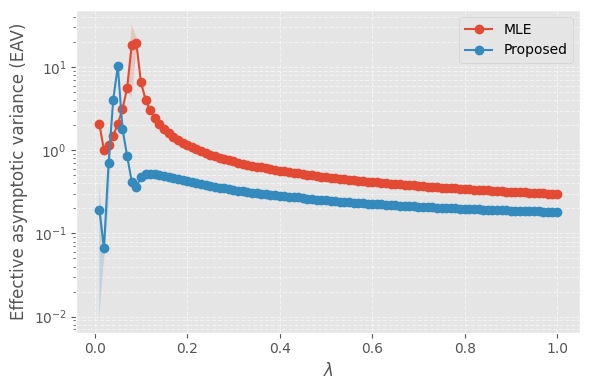}
    \caption{DARWIN data.}
    \label{fig:lam1}
\end{minipage}
\begin{minipage}[b]{0.49\columnwidth}
    \centering
    \includegraphics[width=0.9\columnwidth]{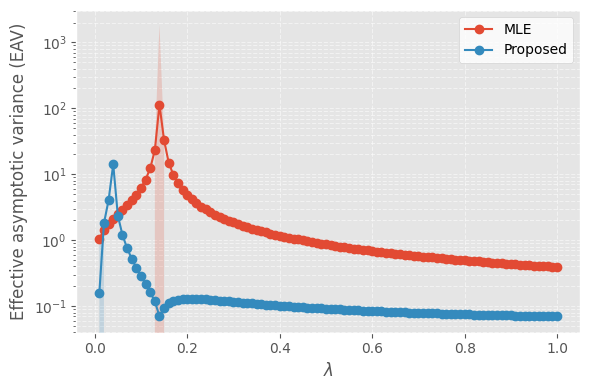}
    \caption{Speech data.}
    \label{fig:lam2}
\end{minipage}
\end{figure}

\bibliographystyle{chicago}
\bibliography{main}

\end{document}